\documentclass[a4paper,11pt,oneside  ]{article}	

\makeindex
\usepackage{hyperref}
\usepackage{graphicx}
\usepackage{enumerate}
\usepackage{mathrsfs}
\hypersetup{colorlinks, linkcolor=blue, urlcolor=red, filecolor=green, citecolor=blue}
\usepackage{amsmath}
\usepackage{amsthm}
\usepackage{amssymb}
\usepackage{authblk}
\usepackage{enumerate}
\usepackage{fullpage}
\usepackage{esint}

\theoremstyle{plain}
\newtheorem{theorem}{Theorem}[section]
\newtheorem{corollary}[theorem]{Corollary}
\newtheorem{lemma}[theorem]{Lemma}
\newtheorem{proposition}[theorem]{Proposition}

\theoremstyle{definition}
\newtheorem{definition}[theorem]{Definition}

\newtheorem{assumption}[theorem]{Assumption}

\theoremstyle{remark}
\newtheorem{remark}{Remark}

\newcommand{\N}{\mathbb{N}}
\newcommand{\R}{\mathbb{R}}

\newcommand{\Z}{\mathbb{Z}}

\newcommand\e{\varepsilon}

\newcommand\divv{\operatorname{div}}

\def\Xint#1{\mathchoice
   {\XXint\displaystyle\textstyle{#1}}%
   {\XXint\textstyle\scriptstyle{#1}}%
   {\XXint\scriptstyle\scriptscriptstyle{#1}}%
   {\XXint\scriptscriptstyle\scriptscriptstyle{#1}}%
   \!\int}
\def\XXint#1#2#3{{\setbox0=\hbox{$#1{#2#3}{\int}$}
     \vcenter{\hbox{$#2#3$}}\kern-.5\wd0}}

\def\fint{\Xint-}

\newcommand\dist{\operatorname{dist}}

\newcommand{\super}[1]{^{(#1)}}

\newcommand{\SO}[1]{\operatorname{SO}(#1)}
\newcommand\id{\operatorname{id}}
\newcommand\Id{\operatorname{Id}}

\newcommand{\bfa}{{\bf{a}}}

\newcommand{\dd}{\mathrm d}

\newcommand\ho{{\operatorname{hom}}}
\newcommand\per{{\operatorname{per}}}
\newcommand\loc{{\operatorname{loc}}}

\newcommand{\step}[1]{\medskip\noindent\textbf{Step #1. }}
\newcommand{\substep}[1]{\medskip\noindent\textit{Substep #1. }}

 \title{Lipschitz estimates and existence of correctors for nonlinearly elastic, periodic composites subject to small strains}
\date{\today}

 \author[1]{Stefan Neukamm\thanks{stefan.neukamm@tu-dresden.de}}
\author[1]{Mathias Sch\"affner\thanks{mathias.schaeffner@tu-dresden.de}}
\affil[1]{Faculty of Mathematics, Technische Universit\"at Dresden}

\begin{document}

\maketitle

\tableofcontents

\begin{abstract}
 We consider periodic homogenization of nonlinearly elastic composite materials. Under suitable assumptions on the stored energy function (frame indifference; minimality, non-degeneracy and smoothness at identity; $p\geq d$-growth from below), and on the microgeometry of the composite (covering the case of smooth, periodically distributed inclusions with touching boundaries), we prove that in an open neighbourhood of the set of rotations, the multi-cell homogenization formula of non-convex homogenization reduces to a single-cell formula that can be represented with help of a corrector. This generalizes a recent result of the authors by significantly relaxing the spatial regularity assumptions on the stored energy function. 
As an application, we consider the nonlinear elasticity problem for $\e$-periodic composites, and prove that minimizers (subject to small loading and well-prepared boundary data) satisfy a Lipschitz estimate that is uniform in $0<\e\ll1$. A key ingredient of our analysis is a new Lipschitz estimate (under a smallness condition) for monotone systems with spatially piecewise-constant coefficients. The estimate only depends on the geometry of the coefficient's discontinuity-interfaces, but not on the distance between these interfaces.

\smallskip


\end{abstract}

\section{Introduction}

\subsection{Informal summary of the results}

A major difficulty in periodic homogenization of nonlinear elasticity is that long wavelength buckling may occur and thus the homogenized energy is in general not characterized by averaging over a single periodicity cell. This genuine nonlinear effect is not present in monotone or linear models for elasticity. Indeed, the classic result by Braides~\cite{Braides86} and M\"uller~\cite{Mueller87} on periodic homogenization of integral functionals shows that the functional
\begin{equation}\label{def:ene}
W^{1,p}(A,\R^d)\ni u\mapsto\int_A W(\tfrac{x}\e,\nabla u(x))\,dx,
\end{equation}
(under suitable growth conditions) $\Gamma$-converges as $\e\downarrow 0$ to a functional of the form $\int_A W_\ho (\nabla u)\,dx$, where $W_\ho$ denotes a homogenized energy density that is given by the \textit{multi-cell homogenization formula}
\begin{equation}\label{def:whom}
 W_\ho(F):=\inf_{k\in\N}W_\ho\super{k}(F),\qquad W_\ho\super{k}(F):=\inf_{\phi\in W_\per^{1,p}(Q_k,\R^d)}\fint_{Q_k}W(y,F+\nabla \phi)\,dy,
\end{equation}
where $Q_k:=(-\frac{k}{2},\frac{k}{2})^d$ and $Q_1$ is the periodicity cell. In the special case that $W=V$ is convex in the second variable, the multi-cell formula simplifies to a \textit{single-cell formula} and can be represented with help of a \textit{corrector}, i.e.,
\begin{equation*}
  V_\ho(F)=V_\ho\super1(F)=\min_{\phi\in W^{1,p}_\per(Q_1,\R^d)}\int_{Q_1}V(y,F+\nabla \phi(y))\,dy=\int_{Q_1}V(y,F+\nabla \phi(y,F))\,dy,
\end{equation*}
where the corrector $\phi(F)\in W^{1,p}_\per(Q_1,\R^d)$ is defined as the unique minimizer subject to $\int_{Q_1}\phi(F)=0$. The corrector is an important object in homogenization and yields an analytic relation between the material's heterogeneities and oscillations in the solution. The existence of a corrector can be the starting point of a more refined analysis, e.g.,\ to formulate a two-scale expansion and to establish estimates on the homogenization error. In the case of nonlinear elasticity, convexity of the integrand is ruled out by frame indifference and thus a reduction of $W_{\ho}$ to the single cell formula $W\super 1_\ho$ cannot be expected in general. In contrast, very recently in \cite{NS17} we showed that $W_\ho=W_\ho\super1$ is valid for small (but finite) strains. More precisely, our result in \cite{NS17} can be summarized as follows:
\begin{enumerate}[(i)]
 \item If $F\mapsto W(y,F)$ satisfies standard assumptions of hyperelasticity such as frame indifference; minimality, non-degeneracy and smoothness at identity; and $p\geq d$-growth from below, 
 \item and if the composite is sufficiently regular in the sense that $y\mapsto W(y,F)$ is smooth or of the form $W(y,F)=W(y_1,F)$ (for $F$ close to $\SO d$),
\end{enumerate}
then the following implication holds:
\begin{equation}\label{1cell:intro}
 \exists \rho>0:\quad\dist(F,\SO d)\leq \rho\quad\Rightarrow\quad W_\ho(F)=W_\ho\super{1}(F)=\int_{Q_1}W(y,F+\nabla\phi(y,F))\,dy,
\end{equation}
with a corrector $\phi(F)\in W^{1,p}_\per(Q_1,\R^d)$ that is characterized as the solution of a monotone system. This result extends earlier asymptotic results in \cite{MN11,GN11} that establish a reduction to a single cell-formula in the case of infinitesimally small strains. In view of a classic example by Stefan M\"uller (see \cite[Theorem 4.3]{Mueller87}) the above result is optimal in the sense that there exist $W$ satisfying (i) and (ii) and $F\in \R^{d\times d}$ such that $W_\ho(F)<W_\ho\super k(F)$ for all $k\in\N$.  
While assumption (i) is reasonable in view of applications to nonlinear elasticity, assumption (ii) is unsatisfactory, since it rules out to treat elastic composites beyond laminates. 
\smallskip

The first result of the present paper establishes \eqref{1cell:intro} for composites satisfying a significantly weaker assumption on the microgeometry of the composite than (ii). In particular, we allow for composites with closely spaced interfacial boundaries, see Theorem~\ref{T:1cell} below. As in our previous work \cite{NS17}, the proof of \eqref{1cell:intro} relies on two main ingredients---a careful convexification that yields a reduction to the convex case, and Lipschitz regularity for monotone systems. More precisely, in a convexification step (see Section~\ref{sec:convexreduction}) we construct a \textit{matching convex lower bound} $V$ for the stored energy function $W$ perturbed by a Null-Lagrangian. This construction, which has its origins in \cite{CDMK06,FT02}, yields a lower bound  for the perturbed, non-convex $W$ in \eqref{def:whom} by a strictly convex energy density $V$. Surprisingly, $V$ is equal to the perturbed $W$ in an open neighbourhood of $\SO d$. This allows us to express $W_{\ho}(F)$ in terms of $V_{\ho}(F)=\int_{Q_1}V(y,F+\nabla\phi(y,F))\,dy$, provided the Lipschitz-norm of the corrector $\phi(F)$ (associated with $V$) is sufficiently small. In \cite{NS17} the required Lipschitz estimate is obtained in the case when $y\mapsto W(y,F)$ is smooth by appealing to the implicit function theorem, and in the case of a layered material by an ODE argument, respectively. In the present contribution, we establish directly Lipschitz estimates for monotone elliptic systems with piecewise-constant coefficients, see Section~\ref{sec:introlip} for the precise assumptions and results. These estimates are of independent interest and their main novelty is that they depend on the geometry of the coefficient's discontinuity-interfaces, but not on the distance between those interfaces. Our estimates extend previous results valid for linear equations and systems \cite{LV00,LN03} and for nonlinear, scalar equations in \cite{BK16,BK17} to the case of monotone elliptic systems, see below Theorem~\ref{Th1} for a more detailed survey of the literature. Let us only mention that our Lipschitz estimates requires for $d\geq 3$ a smallness condition, see~\eqref{ass:smallnesTh1}. The smallness condition is necessary and natural, since for the monotone system under consideration, Lipschitz estimates for $d\geq 3$  are in general not available---even in the cases of homogeneous coefficients. 

\smallskip

The second result of the present paper is a \textit{uniform Lipschitz estimate} for minimizers of energies of the type \eqref{def:ene} subject to small body forces and well-prepared, periodic boundary conditions. More precisely, we study the variational problem
\begin{equation}\label{intro:ene:eps}
 \mbox{minimize}\quad \mathcal I_\e(u):=\int_{Q_1} W(\tfrac{x}\e,G+\nabla u)-f\cdot u\,dx\qquad\mbox{subject to}\quad u\in W_{\per}^{1,p}(Q_1).
\end{equation}
Assuming that $\dist (G,\SO d ) +\| f\|_{L^q}$, $q>d$ is sufficiently small, we prove existence and uniqueness (up to an additive constant) for solutions to \eqref{intro:ene:eps}, and provide Lipschitz estimates that are independent of $\e>0$, cf.~Theorem~\ref{T:Lipschitzeps}. 
The topic of obtaining \textit{uniform regularity estimates} for systems with coefficients oscillating on a small scale goes back to the seminal work of Avellaneda and Lin \cite{AL87}. They considered linear elliptic systems with periodic coefficients and observed that the good regularity theory of the homogenized (constant-coefficient) system can be lifted to the original system with oscillating coefficients. Based on this observation, they obtained a strong improvement of the regularity at scales that are large compared to the period of the coefficients. Combined with a small-scale regularity (that requires some regularity of the coefficients on small scales), they obtained uniform regularity estimates on the $L^p$-, H\"older, and Lipschitz scale, see~\cite{AL87,AL91}. The result of Avellaneda and Lin has been extended in various directions, e.g.,\ in \cite{LN03} for piecewise regular coefficients or in \cite{KLS13} for global estimates for the Neumann problem; see also the recent lecture notes by Shen~\cite{ShenNotes} for an up-to-date and in depth overview for uniform estimates for linear elliptic systems with oscillating (periodic) coefficients. Very recently the philosophy of Avellaneda and Lin has been extended to non-periodic situations, in particular in the case of stochastic homogenization large-scale regularity results have been established in \cite{AS16,AM16,GNO14} (for linear systems and nonlinear equations), see also \cite{ASh16} for uniform Lipschitz estimates for linear systems with almost periodic coefficients. Finally, we also mention \cite{KM08}, where uniform regularity estimates are obtained for highly oscillatory (periodic) nonlinear elliptic systems, see Remark~\ref{rem:kristensen} for a comparison with the results obtained in the present paper.

\smallskip

Our Lipschitz estimate  (see Theorem~\ref{T:Lipschitzeps}) seems to be the first result that applies to nonlinear elasticity albeit for small (but finite) data. It is noteworthy that due to the non-convexity of $W$ (we do not assume quasi-convexity of $W$) existence and in particular uniqueness of a solution of \eqref{intro:ene:eps} might fail in the absence of a smallness assumption of the data. Similar to the validity of the single-cell formula, the strategy to prove those estimates relies on the reduction argument to a monotone system with help of the \textit{matching convex lower bound}. We combine it with a large-scale regularity theory for monotone elliptic systems with periodic coefficients that comes in form of an \textit{intrinsic excess decay estimate}, see Proposition~\ref{P:largescale:holder}. This estimate is of independent interest and  extends a classical regularity statement for (constant-coefficient) monotone systems to the periodic case. To put the result into perspective, recall the classical $\e$-regularity statement for monotone systems (see e.g.,~\cite{Giaquinta83,Giu03,Ming07}): Assume $\bfa:\R^{d\times d}\to\R^{d\times d}$ is a smooth, strongly monotone coefficient  with $\bfa (0)=0$ and $|\bfa(F)|\leq c|F|$. Then for any H\"older-exponent $\alpha\in(0,1)$, there exists  $\kappa>0$ such that any $\bfa$-harmonic field $u\in H^1_{\loc}(\R^d)$ (i.e., $\divv \bfa (\nabla u)=0$) with \textit{small excess} in the sense of $E(\nabla u,B_R):=\inf_{F\in\R^{d\times d}}\left(\fint_{B_R}|\nabla u - F|^2\right)^\frac12\leq \kappa$ belongs to the class $\nabla u\in C^{0,\alpha}(B_\frac{R}{2})$, and satisfies the decay estimate
\begin{equation*}
 \forall r\in(0,R]:\qquad E(\nabla u,B_r)\leq c \left(\frac{r}{R}\right)^{\alpha} E(\nabla u,B_R),
\end{equation*}
where $c<\infty$ depends only on the smoothness and monotonicity of $\bfa$, and the dimension $d$. In other words, if the excess $E$ is small on one scale, say $R>0$, then we have an excess decay. In Proposition~\ref{P:largescale:holder}, we provide an analogous result for systems with periodic coefficients on large scales (i.e., for $R\geq r\geq 1$, where $1$ stands for the periodicity of the coefficient) and we consider an \textit{intrinsic} version of the excess $E$, see \eqref{def:excess}. 

\subsection{Results for nonlinear elasticity}\label{S:1.2}
In this section we state the assumption on $W$ and present the regularity results for nonlinear elasticity; the regularity results for monotone systems are presented in Sections~\ref{S:1.2}, \ref{S:3}, and \ref{sec:lip:hom}. For notation, we refer to Section~\ref{sec:notation}. We first introduce a class of frame-indifferent stored energy functions that are minimized, non-degenerate and smooth at identity, and satisfy a growth condition from below.
\begin{definition}[Stored energy density]\label{def:walphap}
For $\alpha>0$ and $p>1$, we denote by $\mathcal W_{\alpha}^p$ the class of Borel functions $W:\R^{d\times d}\to[0,+\infty]$ which satisfy the following  properties:
\begin{itemize}
\item[(W1)] $W$ satisfies $p$-growth from below, i.e.,
\begin{equation*}
 \alpha|F|^p-\frac{1}{\alpha}\leq W(F)\quad\mbox{for all $F\in\R^{d\times d}$};
\end{equation*}
\item[(W2)] $W$ is frame indifferent, i.e.,
\begin{equation*}
 W(RF)=W(F)\quad\mbox{for all $R\in\SO d$, $F\in\R^{d\times d}$};
\end{equation*}
\item[(W3)] $F=\Id$ is a natural state and $W$ is non-degenerate, i.e.,\ $W(\Id)=0$ and 
\begin{align*}
 W(F)&\geq \alpha\dist^2(F,\SO d)\quad\mbox{for all $F\in\R^{d\times d}$;}
\end{align*}
\item[(W4)] $W$ is $C^3$ in the neighborhood $U_\alpha:=\{F\in \R^{d\times d}\,:\,\dist(F,\SO d)<\alpha\}$ of $\SO d$ and
\begin{equation*}
 \|W\|_{C^3(\overline {U_{\alpha}})}<\frac{1}{\alpha}.
\end{equation*}
\end{itemize}
\end{definition}
We consider periodic, piecewise-constant composite materials with ``regular'' phases, i.e.,~heterogeneities  that locally can be approximated by laminates. The precise assumption is as follows:
\begin{definition}[layered and regular tessellation]\label{def:regD}
  \begin{itemize}
  \item A \textit{tessellation} (of $\R^d$) is a sequence $\mathcal D:=\{D_\ell\}_{\ell\in\Z}$ of mutually disjoint, open subsets of $\R^d$ such that $\R^d=\bigcup_{\ell\in\Z} \overline D_\ell$. 
  \item For $0<s\leq 1$,  $x\in \R^d$ and tessellations $\mathcal D:=\{D_\ell\}_{\ell\in\Z}$ and $\mathcal D':=\{D_\ell'\}_{\ell\in\Z}$ we set
    \begin{equation*}
      E_{x,s}(\mathcal D,\mathcal D'):=\sup_{0<r\leq 1}r^{-s}\Big(|B_r|^{-1}\sum_{\ell\in\Z}\big|(D_\ell \triangle D_{\ell}')\cap B_r(x)\big|\Big)^\frac12,
    \end{equation*}
    where $\triangle$ and $|\cdot|$ denote the symmetric difference and $d$-dimensional Lebesgue measure of sets in $\R^d$, respectively.
  \item A tessellation $\{D_\ell\}_{\ell\in\Z}$ is called $e$-\textit{layered} if for some $e\in\R^d$ and a strictly monotone sequence $\{h_\ell\}_{\ell\in\Z}$ we have $D_\ell=\{x\in\R^d\,:\,h_\ell<x\cdot e<h_{\ell+1}\}$ for all $\ell\in\Z$. We call a tessellation $\{D_\ell\}_{\ell\in\Z}$ \textit{layered} if it is $e$-layered for some $e\in\R^d$.
  \item A tessellation $\mathcal D$ is called \textit{$(E,s)$-regular} (where $0<s\leq 1$ and $E<\infty$), if for all $x \in\R^d$ there exists a layered tessellation $\mathcal D_x$ such that $E_{s,x}(\mathcal D,\mathcal D_x)\leq E$.
\end{itemize}
\end{definition}
The above definition of a $(E,s)$-regular tessellation is rather implicit. An instructive example is as follows: Suppose that $\mathcal D$ is of the form $\mathcal D:=\{D_0,D_m+z\,\,|\, m=1,\dots,L,\,z\in\Z^d\}$, where $D_1,\dots,D_L$ are sufficiently smooth, say $C^{1,\alpha}$, inclusions that are open, disjoint and compactly contained in $Q_1$, and $D_0=\R^d\setminus \cup_{z\in\Z^d}\cup_{m=1}^L (D_m + z)$.  Then, $\mathcal D$ defines a $(E,s)$-regular partition, where $E$ and $s$ depend on $\alpha$, $L$ and the $C^{1,\alpha}$ norms of the boundary of $D_m$, $m=1,\dots,L$, but not on the distance between two inclusions. In particular, the inclusions might touch (see~\cite[Lemma~5.2]{LV00}).
\begin{definition}[heterogeneous coefficient fields and energy densities]
  Let $\mathcal F$ denote a class of coefficients (resp.~energy densities) on $\R^{d\times d}$ (e.g., $\mathcal F=\mathcal W^p_\alpha$). A \textit{coefficient field (resp.~energy density) of class $\mathcal F$} is a function $f$ defined on $\R^d\times\R^{d\times d}$ such that $f(\cdot,F)$ is Borel-measurable for all $F\in\R^{d\times d}$ and $f(x,\cdot)\in\mathcal F$ for a.e.~$x\in\R^d$. Additionally, $f$ is called
  \begin{itemize}
  \item \textit{spatially homogeneous}, if $f(\cdot,F)$ is constant for all $F\in\R^{d\times d}$.
  \item \textit{layered} (resp.~\textit{$e$-layered} or \textit{$(E,s)$-regular}), if there exists a layered (resp.~$e$-layered or $(E,s)$-regular) tessellation $\{D_\ell\}_{\ell\in\Z}$ such that
    $f(\cdot,F)$ is constant on each $D_\ell$ for all $F\in\R^{d\times d}$ and $\ell\in\Z$.
  \item \textit{periodic}, if $f(x+z,F)=f(x,F)$ for all $F\in\R^{d\times d}$, all $z\in\Z^d$, and a.e.~$x\in\R^d$.
  \end{itemize}
\end{definition}

\begin{assumption}\label{ass:W:1}
  Fix $0<\alpha,s\leq 1$, $E<\infty$ and $p\geq d$. We suppose
  that $W$ is a $(E,s)$-regular \& periodic energy density of class $\mathcal
  W^p_\alpha$.
\end{assumption}
Our first main result proves the validity of the single-cell formula for small, but finite strains:
\begin{theorem}\label{T:1cell}
Suppose Assumption~\ref{ass:W:1} is satisfied. Then there exists $\bar{\bar\rho}>0$  such that for all 
\begin{equation*}
  F\in U_{\bar{\bar\rho}}:=\big\{\,F\in \R^{d\times d}\,:\,\dist(F,\SO d)<\bar{\bar\rho}\,\big\}.
\end{equation*}
the following properties are satisfies:
\begin{enumerate}[(a)]
\item (Single-cell formula). 
  \begin{equation*}
    W_\ho(F)=W_\ho\super{1}(F).
  \end{equation*}
\item (Corrector). There exists a unique corrector $\phi(F)\in W^{1,p}_{\per,0}({Q_1},\R^d)$ such that 
  \begin{equation*}
      W_\ho(F)=\inf_{\varphi\in W_\per^{1,p}(Q_1)}\int_{Q_1}W(y,F+\nabla \varphi(y))\,dy=\int_{Q_1}W(y,F+\nabla \phi(y,F))\,dy.
  \end{equation*}
  The corrector satisfies $\phi(F)\in W^{1,\infty}({Q_1})$ and it is the unique weak solution of
  \begin{equation}\label{eq:correctorpde}
   -\divv DW(y,F+\nabla \phi)=0\quad\mbox{satisfying}\quad \|\dist(F+\nabla \phi)\|_{L^\infty(Q_1)}< \delta,
  \end{equation}
  and $\int_{Q_1}\phi=0$,  where $\delta=\delta(\alpha,d,p)>0$ is given in Lemma~\ref{C:wv} below.
\item (Regularity and quadratic expansion). $W_\ho\in C^2(U_{\bar{\bar\rho}})$ and for all $G\in\R^{d\times d}$ we have 
\begin{align*}
 DW_\ho(F)[G]=&\int_{Q_1}DW(y,F+\nabla \phi(y,F))[G]\,dy\notag,\\
 D^2W_\ho(F)[G,G]=&\inf_{\psi\in H_\per^1({Q_1})}\int_{Q_1}D^2W(y,F+\nabla \phi(y,F))[G+\nabla \psi(y),G+\nabla \psi(y)]\,dy,
\end{align*}
where $\phi(F)$ denotes the corrector defined in (b).
\item (Strong rank-one convexity). There exists $c>0$ such that 
\begin{equation*}
  D^2W_\ho(F)[a\otimes b,a\otimes b]\geq  c|a\otimes b|^2\quad\mbox{for all $a,b\in\R^d$}.
\end{equation*}
\end{enumerate}
\end{theorem}
The proof of Theorem~\ref{T:1cell} is given in Section~\ref{sec:1cell}. 

\begin{remark}\label{rem:below1cell}
 In \cite[Theorem~1]{NS17}, the conclusion of Theorem~\ref{T:1cell} is proven under the assumption that the energy densities are smooth in the space variable $x$ or are $e_d$-layered. Here, we extend this result to $(E,s)$-regular energy densities and thus include more realistic models for composite materials. 
\end{remark}
 
\begin{remark}[Quantitative two-scale expansion]\label{rem:2scale}
 In \cite{NS17}, the existence of a unique corrector is used to establish a quantitative two-scale expansion for small, but finite, loads. In particular, we studied the variational problem 
 \begin{equation*}
  \mbox{minimize}\quad \mathcal I_\e(u):=\int_A W(\tfrac{x}\e,\nabla u)-f\cdot u\,dx\qquad\mbox{subject to}\quad u-g\in W_0^{1,p}(A).
 \end{equation*}
 In \cite[Theorem~3]{NS17}, we proved that if $\|f\|_{L^q(A)}+\|g-\id\|_{W^{2,q}(A)}$ for some $q>d$ is sufficiently small, then for every $u\in g+W_0^{1,p}(A)$, it holds
 \begin{equation*}
 \|u-u_0\|_{L^2(A)}+\|u-(u_0+\e\phi(\tfrac{\cdot}\e,\nabla u))\|_{H^1(A)}\lesssim \sqrt{\e}+\left(\mathcal I_\e(u)-\inf_{g+W_0^{1,p}(A)}\mathcal I_\e\right)^\frac12, 
\end{equation*}
where $u_0$ denotes the unique minimizer of
\begin{align*}
 \mathcal I_\ho(u):=\int_A W_{\hom}(\nabla u)-f\cdot u\,dx\qquad\mbox{subject to}\quad u-g\in W_0^{1,p}(A).
\end{align*}
 For this, the smoothness (or laminate structure) of the energy density $W$ is used to show that $x\mapsto \phi(x,\nabla u_0(x))$ is in $H^1(A)$. In the case of $(E,s)$-regular coefficients this is not clear. In future work we address the issue of convergence rates also in the case of $(E,s)$-regular coefficients. Let us mention that in Step~2 of the proof of Proposition~\ref{P:largeholer1} below, we provide an estimate on the homogenization error for rather irregular boundary conditions with a modified two-scale expansion. In the situation of smooth boundary data as considered in \cite[Theorem~3]{NS17} this approach yields an error estimate with $\sqrt{\e}$ replaced by $\e^\frac14$ and $\nabla (\phi(\tfrac{x}\e,\nabla u_0))$ replaced by  $\nabla \phi(\tfrac{\cdot}\e,F)$ where $F$ is a suitable piecewise constant approximation of $\nabla u$.   
\end{remark}

The second main result of the present contribution is the following Lipschitz-regularity statement for minimizers of integral functionals of type \eqref{def:ene} with small data and periodic boundary conditions:

\begin{theorem}[Uniform Lipschitz estimate I]\label{T:Lipschitzeps}
Suppose Assumption~\ref{ass:W:1} is satisfied. Suppose that $\e^{-1}\in\N$ and fix $q>d$. There exist $\bar \rho=\bar\rho(\alpha,d,E,q,s)>0$ and $c=c(\alpha,d,E,q,s)<\infty$  such that the following statement holds:

\smallskip

For given $f\in L^q(Q_1)$, with $\int_{Q_1}f=0$, and $F\in\R^{d\times d}$, consider the minimization problem 
\begin{equation}\label{minprob:ueps}
 \inf\left\{ \int_{Q_1}W(\tfrac{x}\e,F+\nabla u(x))-f(x)\cdot u(x)\,dx\, \big|\, u\in W_{\per,0}^{1,p}(Q_1)\right\}.
\end{equation}
Suppose $f$ and $F$ satisfy the smallness condition
\begin{equation}\label{def:Lambda:intro}
 \Lambda(f,F):=\|f\|_{L^q(Q_1)}+\dist(F,\SO d)<\bar\rho.
\end{equation}
Then:
\begin{itemize}
 \item There exists a unique minimizer $u$ of \eqref{minprob:ueps}.
 \item The minimizer $u$ of \eqref{minprob:ueps} satisfies $u\in W^{1,\infty}(Q_1)$ and the estimate 
\begin{align*}
 \|\dist(F+\nabla u,\SO d)\|_{L^\infty(Q_1)}\leq c\Lambda(f,F).
\end{align*}  
 Moreover, $u$ is the unique weak solution in $W_{\per,0}^{1,p}(Q_1)$ of
 \begin{equation}\label{eq:eulerlagrange}
  -\divv DW(\tfrac{x}\e,F+\nabla u)=f\quad\mbox{satisfying}\quad \|\dist(F+\nabla u)\|_{L^\infty(Q_1)}< \delta,
  \end{equation}
 where $\delta=\delta(\alpha,d,p)>0$ is given in Lemma~\ref{C:wv} below.

\end{itemize}
\end{theorem}

\begin{remark}
 Notice that we do not assume that $W$ is quasi-convex and thus existence of a minimizer for \eqref{minprob:ueps} does not follow from the direct method of the calculus of variations. Moreover, since we do not assume any growth condition for $W$ from above, it is in general not clear that a minimizer for \eqref{minprob:ueps} satisfies the Euler-Lagrange equation of the form \eqref{eq:eulerlagrange}. 

\smallskip

\begin{remark}
  In this work, we do not address boundary regularity and thus we are
  forced to consider well-prepared periodic boundary conditions in
  Theorem~\ref{T:Lipschitzeps}. The proof of
  Theorem~\ref{T:Lipschitzeps} relies on a convexification step, see
  Section~\ref{sec:convexreduction}, and uniform Lipschitz estimates
  for the corresponding monotone system, see~Theorem~\ref{T:3}
  below. The periodic boundary conditions in
  Theorem~\ref{T:Lipschitzeps} could be replaced by Dirichlet boundary
  conditions provided uniform global Lipschitz estimates for the
  corresponding monotone system are established. In the case of
  discontinuous coefficients, as considered in
  Theorem~\ref{T:Lipschitzeps}, this is not expected in general, even
  in the case of linear laminates (see
  e.g.,~\cite[Section~5.3]{AL87}). For sufficiently regular
  coefficients (or $(E,s)$-regular coefficients and additional
  assumptions on the relation between the geometry of the macroscopic
  domain and the microstructure), we expect that those estimates hold
  true but this is beyond the scope of the present paper.
\end{remark}

\end{remark}

\subsection{Reduction to monotone systems}\label{sec:convexreduction}

The proof of Theorem~\ref{T:1cell} follows the strategy of \cite{NS17}. The starting point is the observation that $W\in\mathcal W_\alpha^p$ implies the existence of a ``matching convex lower bound''. For the precise statement we introduce the following class of strongly convex functions:
\begin{definition}[Convex energy density]\label{def:vbeta}
  For $\beta>0$ we denote by $\mathcal V_{\beta}$ the set of functions $V\in C^2(\R^{d\times d})$ satisfying for all $F,G\in\R^{d\times d}$
   \begin{align*}
   &\beta |F|^2-\frac{1}{\beta}\leq V(F)\leq \frac{1}{\beta}(|F|^2+1),\\
      & |DV(F)[G]|\leq \frac1\beta(1+|F|)|G|,\\
   &\beta|G|^2\leq D^2V(F)[G,G]\leq \frac{1}{\beta}|G|^2. 
 \end{align*}
\end{definition}

The following lemma is proven in \cite{NS17} extending a construction that appeared earlier in the context of discrete energies in \cite{CDMK06,FT02}
\begin{lemma}[Matching convex lower bound (see \cite{NS17}, Corollary~2.3)]\label{C:wv}
Suppose that Assumption~\ref{ass:W:1} is satisfied. Then there exist $\delta,\mu,\beta>0$ (depending on $\alpha,d$ and $p$), and a $(E,s)$-layered \& periodic energy density $V$ of class $\mathcal V_{\beta}$ such that for almost every $x\in \R^d$ we have $V(x,\cdot)\in C^3(\R^{d\times d})$, and
\begin{align}
  &W(x,F)+\mu \det F\ \geq V (x,F)\qquad\mbox{for all $F\in \R^{d\times d}$,}\label{WgeqV}\\
  &W(x,F)+\mu \det F\ =V (x,F)\qquad \mbox{for all }F\in\R^{d\times d}\mbox{ with }\dist(F,\SO d)<\delta,\label{W=V}\\
  &V(x,RF)=V(x,F)\qquad\text{for all }F\in \R^{d\times d}\mbox{ and }R\in\SO d.\notag
\end{align}
\end{lemma}

As a rather direct consequence of \eqref{WgeqV}, the convexity of $V$ and the fact that $\R^{d\times d}\ni F\mapsto \det(F)$ is a Null-Lagrangian, we obtain the lower bound
\begin{equation*}
 W_\ho(F)\geq V_\ho^{(1)}(F)-\mu \det(F)\qquad\mbox{for all $F\in\R^{d\times d}$.}
\end{equation*}
Thus, for the proof of the single-cell formula $W_\ho(F)=W_\ho^{(1)}(F)$ (for $F\in\R^{d\times d}$ sufficiently close to $\SO d$) we need to establish the corresponding upper bound. Our argument, which is based on the matching property \eqref{W=V}, requires that the corrector of the convex energy density $V$, i.e.,~the minimizer to
\begin{equation}\label{eq:convex-corr}
H_{\per,0}^1(Q_1,\R^d)\ni\phi\mapsto \fint_{Q_1} V(y,F+\nabla \phi(y))\,dy,
\end{equation}
satisfies a Lipschitz estimate. In \cite{NS17} the Lipschitz estimate is obtained via the implicit function theorem and critically uses smoothness of $W(x,F)$ in $x$. In the present paper, we proceed by a different argument that relies on variational methods and applies to $(E,s)$-regular energy densities.
\subsection{Lipschitz estimates for monotone systems}\label{sec:introlip}
The minimizer  $\phi$ to the convex functional \eqref{eq:convex-corr} is characterized by the associated Euler-Lagrange equation $-\divv(\bfa_F(\nabla \phi))=0$ which is a monotone system with coefficient field $\bfa_F(x,G):=DV(x,F+G)$. In this section we present a Lipschitz estimate for such systems with $(E,s)$-regular coefficients. To be precise, we introduce a class of monotone coefficients:
\begin{definition}[Monotone coefficients]
Fix $\beta\in(0,1]$. We denote by $\mathcal A_\beta^0$ the class of coefficients $\bfa:\R^{d\times d}\to\R^{d\times d}$ satisfying 
\begin{align}
 \bfa(0)=&0,\label{ass:a}\\
 \beta|F-G|^2\leq& \langle \bfa(F)-\bfa(G),F-G\rangle,\label{ass:monotonicity}\\
 \beta|\bfa(F)-\bfa(G)|\leq& |F-G|.\label{ass:lip}
\end{align}
For a given modulus of continuity $\omega:[0,\infty)\to [0,1]$ that is concave, continuous, monotone increasing, and satisfying $\omega(0)=0$, we denote by $\mathcal A_{\beta,\omega}$ the class of coefficients $\bfa\in A_\beta^0$ which additionally satisfy
\begin{equation}
 \beta|D\bfa(F)-D\bfa(G)|\leq \omega(|F-G|).\label{ass:Dareg}
\end{equation}
Moreover, we denote by $\mathcal A_{\beta}$ the class $\mathcal A_{\beta,\omega}$, where the modulus of continuity $\omega$ is given by
\begin{equation}\label{def:omega}
 \omega(t):=\min\{ t,1\}.
\end{equation}
\end{definition}
\begin{theorem}[Lipschitz estimate $(E,s)$-regular coefficients]\label{Th1}
Fix $d\geq2$, $\beta,s\in(0,1]$, $E<\infty$ and $q>d$. Let $\bfa$ be a $(E,s)$-regular coefficient field of class $\mathcal A_\beta$. There exists $\overline \kappa=\overline \kappa(\beta,d,E,q,s)>0$ and $c=c(\beta,d,E,q,s)\in[1,\infty)$ such that if $v\in H^1(B)$ and $f\in L^q(B)$, where $B\subset\R^d$ is a ball of radius $1$, satisfy
\begin{equation}\label{eq:Th1}
  \divv \bfa(\nabla v)= f\qquad\text{in }\mathscr D'(B),
\end{equation}
and the smallness condition
\begin{equation}\label{ass:smallnesTh1}
 \max\{\|\nabla v\|_{ L^2(B)},\|f\|_{L^q(B)}\}\leq \begin{cases}\infty&\mbox{if $d=2$,}\\\overline \kappa&\mbox{if $d\geq3$,}\end{cases}
\end{equation}
then
\begin{equation}\label{est:lipschitzTh1}
 \|\nabla v\|_{L^\infty(\frac12 B)}\leq c(\|\nabla u\|_{L^2(B)} + \|f\|_{L^q(B)}).
\end{equation}
\end{theorem}
The proof of Theorem~\ref{Th1} is given in Section~\ref{sec:lipniren}. In the following we review some (partially well-known, partially new) estimates for elliptic systems and give a rough outline of the ideas for the proof of Theorem~\ref{Th1}. For simplicity we restrict in the following to systems with right-hand side $f=0$.
\begin{itemize}
\item (linear \& constant coefficients). In the simplest situation, \eqref{eq:Th1} reduces to a \textit{linear}, strongly elliptic system with \textit{constant coefficients}, i.e.,\ $\bfa(F)=\mathbb L F$, where $\mathbb L$ is a strongly elliptic 4th order tensor. In that case, we may differentiate the equation multiple times via the ``difference-quotient method'', and a combination of Caccioppoli's inequality and Sobolev's inequality yields a priori the local Lipschitz estimate \eqref{est:lipschitzTh1}, e.g.,~see~\cite{GM12}
\item (linear \& layered coefficients). In the case of \textit{linear}, strongly elliptic systems with \textit{layered coefficients},  local a priori Lipschitz-estimates are proven in \cite{CKV86} (see also \cite{LN03}). Next, we outline the proof given in \cite{LN03}: Suppose that the coefficients are $e_d$-layered. Then, it is still possible to differentiate the equation multiple times in directions $e_1,\dots,e_{d-1}$ that are orthogonal to the layers and in combination with a suitable anisotropic Sobolev inequality H\"older-continuity of $\nabla'u:=(\partial_1 u,\dots,\partial_{d-1} u)$ follows. Furthermore, \eqref{eq:Th1} (with $f=0$) yields that certain components of the stress namely
\begin{equation*}
J_d:=(\mathbb L\nabla u)e_d
\end{equation*}
satisfies similar H\"older-estimates as $\nabla'u$. Finally, ellipticity in form of $|\nabla u|\lesssim |\nabla' u| + |J_d|$ yields the desired Lipschitz estimate. For convenience of the reader (and since the result is crucial for the statements in this paper), we present a proof in Appendix~\ref{sec:proofkinderlehrer}.

\item (linear \& perturbation of layered coefficients). In the approach of \cite{LV00,LN03}, one passes from layered to $(E,s)$-regular coefficients with help of a perturbation argument which solely uses the local Lipschitz estimate for solutions of equations with layered coefficients. Let us emphasize that the main effort in \cite{LV00,LN03} is to obtain higher regularity, namely piecewise $C^{1,\alpha}$, estimates. Moreover, we mention that in \cite{BRW10}, gradient $L^p$-estimates are proven under weaker assumptions (in our notation $(E,0)$-regular, with $0<E\ll1$ depending on $p$).
\item (monotone \& constant coefficients). The situation is genuinely different for nonlinear monotone systems. Even in the constant coefficient-case differentiation of the system is possible only once, and  yields a linear system with elliptic, measurable coefficients. Thus, only $L^p$-estimates (with $0<p-2\ll 1$) for the second derivatives of $u$ can be obtained. As a consequence, in $d\geq 3$---even for constant coefficients---we cannot expect Lipschitz estimates in general, e.g.,~see the example of \cite{SY02}. However, solutions still enjoy at least \textit{partial regularity} that is they are smooth except on a set of small measure, see e.g.,~\cite{Giaquinta83,Giu03} (see \cite{Ming07} for a modern overview of the literature and the main techniques). Classically partial regularity results are proven as consequences of so-called \textit{$\e$-regularity} statements, which are roughly of the following type: For every $\alpha\in(0,1)$ there exists $\e>0$ such that if $u\in H^1(B_R(0))$ satisfies 
\begin{equation}\label{eq:modelpartialreg}
 \divv \bfa(\nabla u)=0\quad\mbox{in $\mathscr D'(B_R(0))$}\qquad\mbox{and}\qquad \|\nabla u-(\nabla u)_{B_R(0)}\|_{\underline L^2(B_R(0))}\leq \e,
\end{equation}
then $u\in C^{1,\alpha}(B_\frac{R}{2})$ (here and for the rest of the paper $\|\cdot\|_{\underline L^p(A)}$ denotes the normalized $L^p$-norm defined in \eqref{def:barlp}). This can be proven in the following way (for precise proofs see e.g.,~\cite{GM12,Giaquinta83,Giu03}): By differentiating the equation in \eqref{eq:modelpartialreg}, we obtain 
$$\divv D\bfa(\nabla u)\nabla (\partial_i u)=0\quad\mbox{in $\mathscr D'(B_R(0))$}$$
for $i=1,\dots,d$. Then, one compares $\partial_i u$ to the unique solution $v_i\in \partial_i u + H_0^1(B_R(0))$ of $\divv D\bfa ((\nabla u)_{B_R(0)})\nabla v_i=0$. Since, $v_i$ solves a linear elliptic system with constant coefficients, it satisfies good a priori estimates. Finally, if $\nabla u$ and $(\nabla u)_{B_R(0)}$ are sufficiently close (in the sense of \eqref{eq:modelpartialreg}) the solutions $\partial_i u$ and $v_i$ are sufficiently close to start a suitable iteration which yields the desired result. 
\item (monotone \& layered coefficients). For monotone systems with layered coefficients, it is possible to adapt the scheme of the previous bullet point by appealing to Lipschitz estimates for layered linear systems. Roughly speaking, we obtain the following: For a $e_d$-layered coefficient field $\overline\bfa$ of class $\mathcal A_\beta$, there exist $\alpha\in(0,1)$ and $\e>0$ such that
\begin{equation*}
 \divv \overline\bfa(x_d,\nabla u)=0\quad\mbox{in $\mathscr D'(B_R(0))$}\qquad\mbox{and}\qquad \|\nabla' u-(\nabla' u)_{B_R(0)}\|_{\underline L^2(B_R(0))}\leq \e,
\end{equation*}
imply $\nabla'u\in C^{0,\alpha}(B_\frac{R}{2}(0))$, $\overline\bfa(\cdot,\nabla u)e_d\in C^{0,\alpha}(B_\frac{R}{2}(0))$ and $\nabla u\in L^\infty(B_\frac{R}{2})$, see Proposition~\ref{L:estmorrey} and Corollary~\ref{P:nonlinearlayer}. We present the proof of this results, which are new to our knowledge, in Section~\ref{sec:estnonlinearlayer}. Let us mention that in the case of nonlinear monotone (scalar) equations, Lipschitz estimates for layered coefficients are proven in \cite{BK16} (without smallness assumption but by appealing to De Giorgi-Nash-Moser regularity, the latter is not available in the vectorial case). 
\item (monotone \& $(E,s)$-regular coefficients). The passage from layered coefficients to $(E,s)$-regular coefficients relies on a careful perturbation procedure  (which is inspired by \cite{KM12} and \cite{BK17}). After a suitable rotation, we may approximate the coefficient field $\bfa$ by a coefficient field $\bar \bfa$ which is $e_d$-layered. Then we compare the quantity $A:=(\nabla' u, \bfa(\cdot,\nabla u)e_d)$ with $\bar A:=(\nabla' \bar u,\bar\bfa(\cdot,\nabla \bar u))$ where $\bar u$ is a suitable solution of $\divv \bar \bfa(\cdot,\nabla u)=0$ which is close to $u$. As mentioned in the previous bullet point, smallness of $\nabla \bar u$ in $L^2$ yields continuity of $\bar A$ and by an approximation argument we obtain (pointwise) boundedness of $A$ and thus the desired Lipschitz estimates for $u$. We note that the idea of considering a quantity of type $A$ appeared already in the literature in \cite{BK17} (in a scalar situation). Moreover, we mention that in \cite{BK16} gradient $L^p$, $p<\infty$ are derived for scalar monotone operators under weaker assumptions on the coefficients (in our notation $(E,0)$-regularity with $0<E$ depending on $p$). 
\end{itemize}
Next, we apply the Lipschitz estimate of Theorem~\ref{Th1} to homogenization. We first recall the definition of correctors for monotone systems:
\begin{lemma}[Homogenized coefficients and correctors for monotone systems] \label{L:phi}
Fix $d\geq 2$ and $\beta\in(0,1]$.  Let $\bfa$ be a periodic coefficient field of class $\mathcal A_{\beta}^0$. 
  \begin{itemize}
  \item (Corrector). For all $F\in\R^{d\times d}$ there exists a unique solution in $\phi(F)\in H_{\per,0}^1({Q_1},\R^d)$ to
    \begin{equation*}
      \divv \bfa(F+\nabla \phi(F))=0\quad\mbox{in $\mathscr D'(\R^d)$}.
    \end{equation*}
  \item (Homogenized coefficient). The homogenized coefficient $\bfa_0:\R^{d\times d}\to\R^{d\times d}$ defined by
    \begin{equation*}
      \bfa_0(F):=\fint_{Q_1} \bfa(y,F+\nabla \phi(y,F))\,dy\qquad\mbox{for all $F\in\R^{d\times d}$}
    \end{equation*}
    satisfies $\bfa_0\in\mathcal A_{\beta^3}^0$. Moreover, if $\bfa$ is a periodic coefficient field of class $\mathcal A_{\beta,\omega}$, then $\bfa_0\in \mathcal A_{\beta',\omega'}$ for some $\beta'=\beta'(\beta,d)>0$ and a modulus of continuity $\omega'$, given by $\omega'(\cdot)=\omega(c\cdot)^\frac1q$ with $c=c(\beta,d)\in[1,\infty)$ and $q=q(\beta,d)\in[1,\infty)$. 
  \item (Flux-corrector). For all $F\in\R^{d\times d}$ there exists a unique solution $\sigma(F)\in H_{\per,0}^1({Q_1},\R^{d\times d\times d})$ to (using Einstein's summation convention)
    \begin{align*}
      -\Delta \sigma_{ijk}(F) =& \partial_k J(F)_{ij}-\partial_j J(F)_{ik},\\
      \sigma_{ijk}(F)=&-\sigma_{ikj}(F),\\
      -\partial_k \sigma_{ijk}(F)=&J(F)_{ij},
    \end{align*}
    in $\mathscr D'(\R^d)$ where $J(x,F):=\bfa(x,F+\nabla \phi(x,F)) - \bfa_0(F)$.
  \item (Lipschitz continuity in $F$). There exists $c=c(\beta)<\infty$ such that
    \begin{equation}\label{est:lipphisigma}
      \|\nabla \phi(F)-\nabla \phi(G)\|_{L^2(Q_1)}+\|\nabla \sigma(F)-\nabla \sigma(G)\|_{L^2(Q_1)}\leq c|F-G|.
    \end{equation}
  \end{itemize}
\end{lemma}
The proof of Lemma~\ref{L:phi} is standard and can be found e.g.,~in \cite{CPZ05,NS17}. In the appendix, we present a proof for the fact that $\bfa\in \mathcal A_{\beta',\omega'}$, see Section~\ref{sec:rega}. 

In combination with the Lipschitz estimate Theorem~\ref{Th1} we obtain as an immediate consequence a Lipschitz bound on the correctors:
\begin{corollary}[Lipschitz estimate for the corrector]\label{C:lipcorrector}
Fix $d\geq2$, $\beta,s\in(0,1]$ and $E<\infty$. Let $\bfa$ denote a periodic and $(E,s)$-regular coefficient field of class $\mathcal A_\beta$. Then, there exists $\overline \kappa=\overline \kappa (\beta,d,E,s)>0$ and $c=c(\beta,d,E,s)\in[1,\infty)$ such that
\begin{equation*}
 |F|\leq \begin{cases}
          \infty&\mbox{if $d=2$,}\\
          \overline \kappa&\mbox{if $d\geq3$,}
         \end{cases}
\end{equation*}
implies,
\begin{equation*}
 \|\phi(F)\|_{W^{1,\infty}(Q_1)}\leq c|F|.
\end{equation*}
\end{corollary}
(The proof of Corollary~\ref{C:lipcorrector} is obvious and left to the reader.)

By combining Theorem~\ref{Th1} with the regularity theory of Avellaneda and Lin we obtain the following uniform Lipschitz estimate:
\begin{theorem}[Uniform Lipschitz estimate II]\label{T:3}
Fix $d\geq 2$, $\beta,s\in(0,1]$, $E<\infty$ and $q>d$. Let $\bfa$ be a $(E,s)$-regular \& periodic coefficient field of class $\mathcal A_\beta$. There exists $\overline \kappa_0=\overline \kappa_0(\beta,d,E,q,s)>0$ and $c=c(\beta,d,E,q,s)\in[1,\infty)$ such that if $u\in H^1(B_R)$ and $f\in L^q(B_R)$, for some $R>0$, satisfy
\begin{equation}\label{eq:uepsf}
 \divv \bfa(\tfrac{x}\e,\nabla u)= f\qquad \mbox{in $\mathscr D'(B_R)$,}
\end{equation}
and the smallness condition
\begin{equation}\label{eq:uepssmall}
 \max\{\|\nabla u\|_{\underline L^2(B_R)},R\|f\|_{\underline L^q(B_R))}\}\leq\begin{cases}
                                      \infty&\mbox{if $d=2$}\\
                                      \overline \kappa_0&\mbox{if $d\geq3$}
                                     \end{cases},
\end{equation}
then
\begin{equation}\label{est:liplarge}
  \|\nabla u\|_{L^\infty(B_\frac{R}{2})}\leq c(\|\nabla u\|_{\underline L^2(B_R)}+ R\|f\|_{\underline L^q(B_R)}).
\end{equation}
\end{theorem}
(For the proof see Section~\ref{sec:lip:hom})

\begin{remark}\label{rem:kristensen}
 We recall that without a smallness condition of the form \eqref{eq:uepssmall} the Lipschitz estimate \eqref{est:liplarge} may fail even in the case of spatially homogeneous coefficients. Uniform regularity estimates for solutions of \eqref{eq:uepsf} are obtained in \cite{KM08} in weaker norms without any smallness assumption. More precisely, assuming that $\bfa$ is a periodic coefficient field of class $\mathcal A_\beta^0$ that satisfies some mild regularity conditions in the space variable, the authors of \cite{KM08}  lift the $W^{2,2+\delta}$-regularity of the homogenized operator $\bfa_0$ to the heterogeneous problem and obtain estimates for $\nabla u$ in certain Morrey seminorms $\|\cdot\|_{L^{2,q}}$ with $0<q-2\ll1$ (in fact the main effort in \cite{KM08} is to obtain explicit lower bounds on $\delta>0$  independent of the dimension $d\geq3$). 
\end{remark}

\subsection{Notation.}\label{sec:notation}

For given $x\in\R^d$ and $R>0$, we denote by $B_R(x)$ (and $B_R$ for $x=0$) the open Euclidean ball with center $x$ and radius $R$. For given $r>0$ and any open ball $B\subset\R^d$ with $B=B_R(x)$ for some $x\in\R^d$ and $R>0$, we set $rB:=B_{rR}(x)$. Similarly, for given $x\in\R^d$ and $R>0$, we set $Q_R(x):=x_0+(-\frac12 R,\frac12 R)^d$. We denote by $f\cdot g$ (resp. $\langle F,G\rangle$) the standard scalar product in $\R^d$ (resp. $\R^{d\times d}$). For $f:\R^d\times \R^{d\times d}\to\R^n$ with $n\in\N$, we denote
\begin{itemize}
 \item by $\nabla f$ the Jacobian matrix, i.e., the derivative with respect to the spatial variable: $(\nabla f(x,F))_{ij}=\partial_{x_j}f_i(x,F)=\partial_jf_i(x,F)$.
 \item by  $D^kf$ the $k$-th Fr\'echet-derivative with respect to the second component. If $n=1$, we identify the linear (resp. bilinear) map $Df(\cdot,F)$ (resp. $D^2f(\cdot, F)$) with the matrix (resp. fourth order tensor) defined by $Df(\cdot,F)[G]=\langle Df(\cdot,F),G\rangle$ (resp. $D^2f(\cdot,F)[G,H]=\langle D^2f(\cdot,F)G,H\rangle$).
\end{itemize}
For any measurable set $E\subset\R^d$, we denote by $|E|$ the Lebesgue measure of $E$. For any bounded Lipschitz domain $A\subset\R^d$, we set $\fint_A :=|A|^{-1}\int$. For every $p\in[1,\infty)$, we introduce the following normalized $L^p(A)$-norm of a function $f\in L^p(A)$
\begin{equation}\label{def:barlp}
 \|f\|_{\underline L^p(A)}:=\left(\fint_A |f(x)|^p\,dx\right)^\frac1p.
\end{equation}
For every $f\in L^1(A)$, we set $(f)_A:=\fint_A f(x)\,dx$. For a function $u\in W^{1,p}(A,\R^d)$, $p \in[1,\infty]$ we denote by $\nabla u:=(\partial_1 u,\dots,\partial_d u)$ the weak gradient of $u$. Furthermore, we denote by $\nabla' u:=(\partial_1u,\dots,\partial_{d-1}u)$ the matrix-field of weak derivative of $u$ except $\partial_d u$. We denote by $H^{-1}(A)$ the dual of $H_0^1(A)$ and use the following normalized version of the $H^{-1}(A)$ norm given by
\begin{equation}\label{def:hminus}
 \|F\|_{\underline H^{-1}(A)}:=\sup\left\{|\fint_A F\varphi\,dx|\,:\, \varphi\in H_0^1(A),\, \|\nabla \varphi\|_{\underline L^2(A)}\leq1\right\}.
\end{equation}

A measurable function $u:\R^d\to\R$ is called $Q_1$-periodic, if it satisfies $u(y+z)=u(y)$ for almost every $y\in\R^d$ and all $z\in\Z^d$. We consider the function spaces
\begin{align*}
 L_\per^p({Q_1})&:=\{u\in L_{\loc}^p(\R^d)\, |\, u\mbox{ is ${Q_1}$-periodic}\},\quad L_{\per,0}^p({Q_1}):=\{u\in L_\per^p({Q_1})\, |\, \int_{Q_1} u=0\}\\
 W_\per^{1,p}({Q_1})&:=\{u\in L_{\per}^p({Q_1})\, |\, u\in W_\loc^{1,p}(\R^d)\},\quad  W_{\per,0}^{1,p}({Q_1}):=\{u\in W_\per^{1,p}({Q_1})\, |\, \int_{Q_1} u=0\}.
\end{align*}
and likewise $L_\per^p({Q_1},\R^d)$, $L_{\per,0}^p({Q_1},\R^d)$ etc. Throughout the paper, we drop the explicit dependence of the target space whenever it is clear from the context.

\section{Proofs of the results for nonlinear elasticity}\label{sec:mainelastic}

Before we prove Theorem~\ref{T:1cell} and Theorem~\ref{T:Lipschitzeps}, we relate the matching convex lower bound of $W$ to $(E,s)$-regular coefficient fields of class $\mathcal A_\beta$:
\begin{lemma}\label{L:DVabeta}
 Suppose that Assumption~\ref{ass:W:1} is satisfied. For every $R\in\SO d$, the map $\bfa_R:\R^d\times \R^{d\times d}\to\R^{d\times d}$ given by $\bfa_R(x,F):= DV(x,R+F)-DV(x,R)$ defines a $(E,s)$-regular coefficient field of class $\mathcal A_\beta$. 
\end{lemma}

\begin{proof}[Proof of Lemma~\ref{L:DVabeta}]
 Lemma~\ref{C:wv} implies that $\bfa_R$ is a $(E,s)$-regular coefficient field of class $\mathcal A_\beta^0$ for some $\beta\in(0,1]$. Moreover, using $V(x,\cdot)\in C^3(\R^{d\times d})$, we obtain
 $$
 |D\bfa (F)-D\bfa (G)|\leq \left(\int_0^1|D^3V(F+s(G-F))|\,ds\right)|F-G|.
 $$
 Hence, it suffices to show that $|D^3V(F)|\leq c$ with $c=c(d,\alpha)<\infty$. This follows by a closer inspection of the construction of $V$ given in \cite[Lemma~1]{NS17}: Indeed, there exists a quadratic function $Q$ and $r>0$ such that $V=Q$ on $\{F\in\R^{d\times d}\,:\,|F|>r\}$ (see \cite[Proof of Lemma~1, Step~3]{NS17}) and thus $|D^3 V(F)|=0$ for $|F|>r$ which finishes the proof.  
\end{proof}

\subsection{Single-cell formula, Proof of Theorem~\ref{T:1cell}}\label{sec:1cell}

\begin{proof}[Proof of Theorem~\ref{T:1cell}]

By appealing to Lemma~\ref{C:wv} we may associated with $W$ a matching convex lower bound $V$ with parameters $\delta,\mu,\beta>0$.  For every $F\in\R^{d\times d}$, we denote by $\phi(F)\in H_{\per,0}^1(Q_1,\R^d)$ the corrector for $V$ and $F$, i.e.,~the unique minimizer of
\begin{equation*}
 H_{\per,0}^1({Q_1},\R^d)\ni \phi\mapsto \int_{{Q_1}} V(y,F+\nabla \phi)\,dy.
\end{equation*}
Recall that $\phi(F)$ is characterised as the unique solution to the Euler-Lagrange equation
\begin{equation}\label{eq:correctorconvex}
 \int_{{Q_1}} \langle DV(y,F+\nabla \phi(F)),\nabla \eta\rangle\,dy=0\qquad\mbox{for all $\eta\in H_\per^1({Q_1},\R^d)$.}
\end{equation}

\step 1 Corrector estimates.

We claim that there exists $\rho=\rho(\beta,d,\delta,E,s)>0$ such that 
\begin{equation}\label{est:correctorconvexlb:lip}
 \dist(F,\SO d)< \rho\qquad\mbox{implies}\qquad \|\dist(F+\nabla \phi(F))\|_{L^\infty(Q_1)}< \delta.
\end{equation}
\substep{1.1} We claim that there exists $c=c(\beta)<\infty$ such that  
\begin{equation}\label{est:correctorconvexlb:l2}
 \|\nabla \phi(F)\|_{L^2(Q_1)}\leq c|F-R|\qquad\mbox{for every $R\in\SO d$.}
\end{equation}
Indeed, the minimality of $W(x,\cdot)$ at $\SO d$ implies $DV(x,R)=DW(x,R)+\mu D\det(R)=\mu D\det(R)$ for almost every $x\in\R^d$ and every $R\in\SO d$. Hence, \eqref{eq:correctorconvex}, the monotonicity and Lipschitz continuity of $DV$ imply for every $R\in\SO d$
\begin{align*}
 \beta\|\nabla \phi(F)\|_{L^2(Q_1)}^2\leq& \int_{Q_1}\langle DV(y,F+\nabla  \phi(F))-DV(y,F),\nabla \phi(F)\rangle\,dy\\
 =&\int_{Q_1}\langle DV(y,R)- DV(y,F),\nabla \phi(F)\rangle\,dy\\
 \leq&\frac1\beta|F-R|\int_{Q_1}|\nabla \phi(F)|\,dy.
\end{align*}
The inequality \eqref{est:correctorconvexlb:l2} follows by Youngs inequality.

\substep{1.2} Proof of \eqref{est:correctorconvexlb:lip}.

For given $R\in\SO d$, let $\bfa_R$ be given as in Lemma~\ref{L:DVabeta}. Equation \eqref{eq:correctorconvex} implies 
\begin{equation*}
  \divv \bfa_R (x,F-R+\nabla \phi(F))=0\qquad\mbox{in $\mathscr D'(\R^d)$.}
\end{equation*}
Hence, Theorem~\ref{Th1} and \eqref{est:correctorconvexlb:l2} imply that there exist $\rho=\rho(\beta,d,E,s)>0$ and $c=c(\beta,d,E,s)<\infty$ such that $|F-R|<\rho$ implies
\begin{equation*}
 \|F-R+\nabla \phi(F)\|_{L^\infty(Q_1)}\leq c|F-R|.
\end{equation*}
Choosing $R\in\SO d$ such that $|F-R|=\dist(F,\SO d)$, we obtain \eqref{est:correctorconvexlb:lip}.

\smallskip

\step{2} Single-cell formula; proof of (a).

Let $\rho>0$ be given as in Step~1. We claim for all $F\in U_{\rho}$
\begin{equation}\label{pf:onecell}
  W_\ho(F)=W\super1_\ho(F)=\int_{Q_1}W(y,F+\nabla\phi(F))\,dy=V_\ho(F)-\mu\det(F),
\end{equation}
where $\phi(F)$ denotes the corrector associated with $V$ given above.

\substep{2.1} It holds 
\begin{equation}\label{WhgeqVh}
 W_\ho(F)\geq V_\ho(F)-\mu\det F\quad\mbox{for all $F\in\R^{d\times d}$.}
\end{equation}
A detailed proof of inequality \eqref{WhgeqVh} is given in \cite[proof of Theorem~2, Step~3]{NS17}. The main ingredients are \eqref{WgeqV}, $\int_{Q_1}\det(F+\nabla\phi(y))\,dy=\det(F)$ for every $\phi\in W_\per^{1,p}(Q_1)$ with $p\geq d$ and the validity of the single-cell formula for convex integrands, see~\cite[Lemma~4.1]{Mueller87}.  

\substep{2.2} Proof of \eqref{pf:onecell}.

Fix $F\in U_{\rho}$. In view of Step~1, we have that $\|\dist(F+\nabla \phi,\SO d)\|_{L^\infty({Q_1})}<\delta$. By the matching property \eqref{W=V}, the fact that the determinant is a Null-Lagrangian and the definition of $\phi(F)$, we obtain
\begin{equation*}
 \int_{Q_1}W(y,F+\nabla \phi(F))\,dy=\int_{Q_1}V(y,F+\nabla \phi(F))\,dy-\det(F)=V_\ho(F)-\mu\det (F).
\end{equation*}
Hence, we get
\begin{align*}
 W_\ho(F)\leq W_\ho^{(1)}(F)\leq \int_{Q_1}W(y,F+\nabla \phi(F))\,dy = V_\ho(F)-\mu\det (F)&\stackrel{\eqref{WhgeqVh}}{\leq}W_\ho(F),\notag
\end{align*}
and thus \eqref{pf:onecell}.

\step 3 Uniqueness and regularity of the corrector; proof of (b).
\smallskip

Let $\rho>0$ be as in Step~1 and consider $F\in U_{\rho}$. From \eqref{pf:onecell} we learn that
\begin{equation}\label{Whcor}
  W_\ho(F)=\int_{Q_1}W(y,F+\nabla\tilde\phi)\,dy
\end{equation}
is satisfied for $\tilde\phi=\phi(F)$. If $\tilde\phi\in W^{1,p}_{\per,0}({Q_1})$ denotes another function satisfying \eqref{Whcor}. Then 
\begin{eqnarray*}
  W_\ho(F)&\stackrel{\eqref{Whcor}}{=}&\int_{Q_1}W(y,F+\nabla\tilde\phi)\,dy\stackrel{\eqref{WgeqV}}{\geq}\int_{Q_1} V(y,F+\nabla\tilde\phi)-\mu\det(F)\\
  &\geq& V_\ho(F)-\mu\det F\stackrel{\eqref{pf:onecell}}{=}W_\ho(F).
\end{eqnarray*}
Hence, $\tilde\phi$ is a minimizer of $\psi\mapsto \int_{Q_1}V(y,F+\nabla\psi)$. By strong convexity of $V$, minimizers are unique and we deduce that $\tilde\phi=\phi(F)$. This proves uniqueness of $\phi(F)$. For the asserted regularity of $\phi(F)$ see \eqref{est:correctorconvexlb:lip}. Moreover, we observe that \eqref{est:correctorconvexlb:lip} and the matching property \eqref{W=V} imply $DV(y,F+\nabla \phi(y,F))=DW(y,F+\nabla \phi(y,F))+\mu D\det(F+\nabla \phi(y,F))$ for a.e.\ $y\in Q_1$ and thus $\phi(F)$ solves \eqref{eq:correctorpde}. Finally, uniqueness follows since every solution of \eqref{eq:correctorpde} is also a solution of \eqref{eq:correctorconvex} which is unique up to a additive constant.

\step 4 The proof of (c) and (d)

In \cite{NS17}, we proved (c) and (d) for more regular integrands $W$ but used only to the identity \eqref{pf:onecell} for all $F\in U_\rho$. Hence, the argument works verbatim in the present situation. We thus omit further details here and refer to \cite[proof of Theorem~2, Step~4 and 6]{NS17}.
\end{proof}

\subsection{Uniform Lipschitz estimate, Proof of Theorem~\ref{T:Lipschitzeps}}

\begin{proof}[Proof of Theorem~\ref{T:Lipschitzeps}]

With Theorem~\ref{T:3} at hand, the proof of Theorem~\ref{T:Lipschitzeps} follows in the same way as the proof of Theorem~\ref{T:1cell}. By Lemma~\ref{C:wv}, we may associated with $W$ a matching convex lower bound $V$ with parameters $\delta,\mu,\beta>0$.

\smallskip

By the strong convexity of $V$ and the assumption $(f)_{Q_1}=0$, we find a unique minimizer $v\in H_{\per,0}^1({Q_1})$ such that
\begin{equation}\label{min:ueconvex}
 \int_{Q_1} V(\tfrac{x}\e,F+\nabla v) - f\cdot v\,dx\leq \int_{Q_1} V(\tfrac{x}\e,F+\nabla \varphi) - f\cdot \varphi\,dx\qquad\mbox{for all $\varphi\in H_{\per}^1(Q_1)$.}
\end{equation}
%

\smallskip

\step{1} Lipschitz estimate for $v$. 

We claim that there exist $\overline\rho=\overline \rho(\beta,d,\delta,E,q,s)>0$ and $c=c(\beta,d,E,q,s)\in[1,\infty)$ such that \eqref{def:Lambda:intro} implies
\begin{equation}\label{eq:claimvepslip}
 \|\dist(F+\nabla v,\SO d)\|_{L^\infty(Q_1)}\leq c\Lambda(f,F)< \delta.
\end{equation}
\substep{1.1} We claim that there exists $c_1=c_1(\beta,d)\in[1,\infty)$ such that 
\begin{equation}\label{eq:apriorithlip}
 \|\nabla v\|_{L^2(Q_1)}\leq c_1\Lambda(f,F).
\end{equation}
The Euler-Lagrange equation for $v$ implies
\begin{equation}\label{euler:ueconvex}
 \int_{Q_1}\langle DV(\tfrac{x}\e,F + \nabla v),\nabla \eta\rangle\,dx = \int_{Q_1} f\cdot \eta\,dx\quad\mbox{for all $\eta\in H_\per^1(Q_1)$,}
\end{equation}
where we use $(f)_{Q_1}=0$ and thus $\int_{Q_1} f\eta =\int_{Q_1} f\cdot(\eta-(\eta)_{Q_1})$. Fix $R\in\SO d$ and note that the minimality of $W$ at $\SO d$ yields $DV(x,R)=\mu D\det(R)$ for almost every $x\in\R^d$. Hence, the monotonicity and Lipschitz continuity of $DV$ (using $V\in\mathcal V_\beta$, cf.\ Definition~\ref{def:vbeta}) imply 
\begin{align*}
 \beta\|\nabla v\|_{L^2(Q_1)}^2\leq& \int_{Q_1}\langle DV(\tfrac{x}\e,F+ \nabla v)- DV(\tfrac{x}\e,F),\nabla v\rangle \,dx\\
 =&\int_{Q_1}f\cdot v -\langle DV(\tfrac{x}\e,F)- DV(\tfrac{x}\e,R),\nabla v\rangle \,dx\\
 \leq&\int_{Q_1} f\cdot v\,dx+\frac1\beta|F-R|\int_{Q_1} |\nabla v|\,dx,
\end{align*}
and inequality \eqref{eq:apriorithlip} follows by Young's inequality , the Poincar\'e inequality and the arbitrariness of $R\in \SO d$.

\substep{1.2} Proof of \eqref{eq:claimvepslip}. 

Choose $R\in\SO d$ such that $|F-R|=\dist(F,\SO d)$ and consider $\tilde v_R\in H_{\loc}^1(\R^d)$ given by $\tilde v_R(x):=v(x)+(F-R) x$. Let $\bfa_R$ be given as in Lemma~\ref{L:DVabeta}. Clearly, we have 
\begin{equation*}
 -\divv \bfa_R(\tfrac{x}\e,\nabla \tilde v_R)=\tilde f\quad \mbox{in $\mathscr D'(\R^d)$},
\end{equation*}
where $\tilde f\in L^q_\loc(\R^d)$ denotes the $Q_1$-periodic extension of $f\in L^q(Q_1)$. Moreover, estimate \eqref{eq:apriorithlip} implies $\|\nabla \tilde v\|_{L^2(Q_1)}\leq c_1\Lambda(f,F)+\dist(F,\SO d)\leq (c_1+1) \Lambda(f,F)$. Hence, we can apply Theorem~\ref{T:3} and find $\rho=\rho(\beta,d,E,q,s)>0$ and $c_2=c_2(\beta,d,E,q,s)\in[1,\infty)$ such that $\Lambda(f,F)\leq \rho$ implies  
\begin{equation*}
  \|\nabla \tilde v_R\|_{L^\infty(Q_1)}\leq c_2\Lambda(f,F).
\end{equation*}
Finally, we obtain that \eqref{def:Lambda:intro} with $\bar \rho=\min\{\frac{\delta}{c_2},\rho\}$ yield \eqref{eq:claimvepslip}: 
\begin{equation*}
 \|\dist(F+\nabla v,\SO d)\|_{L^\infty(Q_1)}\leq \|\nabla \tilde v_R\|_{L^\infty(Q_1)}\leq c_2\Lambda(f,F)<\delta.
\end{equation*}

\step{2} Conclusion.

Inequality \eqref{WgeqV} and the fact that the determinant is a Null-Lagrangian yield
\begin{equation}\label{ineq:lipvepslb}
 \int_{Q_1} V(\tfrac{x}\e,F+\nabla v) - f\cdot v\,dx\leq \int_{Q_1} W(\tfrac{x}\e,\nabla \varphi) - f\cdot \varphi\,dx + \mu\det(F),
\end{equation}
for all $\varphi\in W_{\per}^{1,p}(Q_1)$. Estimate \eqref{eq:claimvepslip} and the matching property \eqref{W=V} yield 
\begin{align*}
 \int_{Q_1} W(\tfrac{x}\e,F+\nabla v) - f\cdot v\,dx =\int_{Q_1} V(\tfrac{x}\e,F+\nabla v)-f\cdot v\,dx-\mu\det(F),
\end{align*}
which implies (in combination with \eqref{ineq:lipvepslb}) that $v$ is a minimizer of \eqref{minprob:ueps} and satisfies the claimed Lipschitz estimate. Finally, as in Step~3 of the proof of Theorem~\ref{T:1cell}, we can use the fact that the convex minimization problem \eqref{min:ueconvex} has a unique solution which is characterised by \eqref{euler:ueconvex} to show that $v$ is the unique minimizer of \eqref{minprob:ueps} and the unique solution in $W_{\per,0}^{1,p}(Q_1)$ of \eqref{eq:eulerlagrange}.

\end{proof}

\section{Lipschitz estimates for monotone systems}\label{S:3}

\subsection{Strategy and proof of Theorem~\ref{Th1}}

As mentioned earlier Theorem~\ref{Th1} is obtained by a combination of perturbation methods and some basic estimates for elliptic systems with layered coefficients. We start with presenting the regularity for layered coefficients and state a Campanato-type decay estimate for the derivatives and the flux orthogonal to the layers.
\begin{proposition}[Decay estimate for layered coefficients]\label{L:estmorrey}
Fix $d\geq2$, $\beta\in(0,1]$ and let $\bfa$ be a $e_d$-layered coefficient field of class $\mathcal A_\beta$. There exists $\gamma=\gamma(\beta,d)\in(0,1)$, $\overline \kappa_0=\overline\kappa_0(\beta,d)\in(0,\infty]$ and $c=c(\beta,d)<\infty$ such that the following is true: Suppose that $v\in H^1(B)$, where $B\subset\R^d$ is a ball of radius $R>0$, satisfies 
\begin{equation}\label{eq:nonlinearxdf}
  \divv \bfa(\nabla v)=0\qquad \mbox{in $\mathscr D'(B)$,}
\end{equation}
and the smallness condition
\begin{equation*}
 \|{\nabla'} v - ({\nabla'} v)_{B}\|_{\underline L^2(B)} \leq \begin{cases}
                                                               \infty&\mbox{if $d=2$,}\\
                                                               \overline \kappa_0&\mbox{if $d\geq3$.}
                                                              \end{cases}
\end{equation*}
Then, 
\begin{align}\label{est:decaynablav}
 \sup_{r\in(0,1]}r^{-\gamma}\|\nabla' v - (\nabla' v)_{rB}\|_{\underline L^2(r B)}\leq& c\|{\nabla'} v - ({\nabla'} v)_{B}\|_{\underline L^2(B)},\\
 \sup_{r\in(0,1]}r^{-\gamma}\|J_d(v) - (J_d(v))_{rB}\|_{\underline L^2(r B)}\leq& c\|{\nabla'} v - ({\nabla'} v)_{B}\|_{\underline L^2(B)}\label{est:decayJ},
\end{align}
where $J_d(v)$ is defined by
\begin{equation*}
 J_d(v):=\bfa(\cdot,\nabla v)e_d.
\end{equation*}
\end{proposition}
The proof of Proposition~\ref{L:estmorrey} is deferred to Section~\ref{sec:pf:L:estmorrey}. 

\begin{remark}
 In the proof of Proposition~\ref{L:estmorrey}, we also show that for every $\gamma\in(0,1)$ there exist $\kappa_0(\beta,d,\gamma)>0$ and $c=c(\beta,\gamma,d)<\infty$ such that \eqref{eq:nonlinearxdf} and $\| {\nabla'} v - ({\nabla'} v)_{B}\|_{\underline L^2(B)} \leq\kappa_0$ yield \eqref{est:decaynablav} and \eqref{est:decayJ} (see Step~2 and 3 of the proof of Proposition~\ref{L:estmorrey}).
\end{remark}

A consequence of Proposition~\ref{L:estmorrey} is the following Lipschitz estimate:
\begin{corollary}[Lipschitz estimate for layered coefficients]\label{P:nonlinearlayer}
Fix $\beta\in(0,1]$, $d\geq2$ and let $\bfa$ be a layered coefficient field of class $\mathcal A_\beta$. There exists $\overline \kappa_1=\overline\kappa_1(\beta,d)>0$ and $c=c(\beta,d)<\infty$ such that if $v\in H^1(B)$, with a ball $B\subset\R^d$, satisfies \eqref{eq:nonlinearxdf} and 
\begin{equation}\label{kappasmall}
 \|{\nabla'} v - ({\nabla'} v)_{B}\|_{\underline L^2(B)} \leq \begin{cases}\infty&\mbox{if $d=2$,}\\ \overline \kappa_1&\mbox{if $d\geq3$,}\end{cases}  
\end{equation}
then, 
\begin{equation}
 \|\nabla v\|_{L^\infty(\frac12 B)}\leq c \|\nabla v\|_{\underline L^2(B)}\label{P:nonlinearlayer:claim2}.
\end{equation}
\end{corollary}
We provide a proof of Corollary~\ref{P:nonlinearlayer} in Section~\ref{sec:nonlinearlayer}. In the proof of Corollary~\ref{P:nonlinearlayer} we appeal to monotonicity of $\bfa$ in form of:
\begin{lemma}\label{l:edmonoton}
Fix $\beta\in(0,1]$, $d\geq2$ and $\bfa\in\mathcal A_\beta^0$. Then it holds for every $F\in\R^{d\times d}$
\begin{equation}\label{ed:monoton}
 \beta |F e_d|^2\leq \tfrac1\beta \left( |\bfa(F)e_d|^2+\tfrac1\beta |F(\Id-e_d\otimes e_d)|^2\right).
\end{equation}
\end{lemma}

\begin{proof}
The monotonicity \eqref{ass:monotonicity} implies
\begin{align*}
 \beta |F e_d|^2\leq& \langle \bfa(F)-\bfa(F(\Id-e_d\otimes e_d)),Fe_d\otimes e_d\rangle= (\bfa(F)-\bfa(F(\Id-e_d\otimes e_d)))e_d \cdot Fe_d.
\end{align*}
The inequality~\eqref{ed:monoton} follows after applying Youngs inequality and using $|\bfa(G)|\leq \frac1\beta|G|$, cf.~\eqref{ass:a} and \eqref{ass:lip}.
\end{proof}

Next, we treat the case of $(E,s)$-regular coefficients for which we establish regularity by lifting the Lipschitz-estimate for layered coefficients with help of perturbation arguments. To this end, we introduce a quantity which measures the difference between two operators in a suitable way:
\begin{definition}
Let $d\geq2$. For coefficients $\bfa,\bar \bfa\in C(\R^{d\times d},\R^{d\times d})$ we define 
\begin{align}\label{def:a1-a2}
  \dd(\bfa,\bar \bfa):=\sup_{F\in \R^{d\times d}\setminus \{0\}}  \frac{|\bfa(F)-\bar \bfa(F)|}{|F|}.
\end{align}
For heterogeneous coefficients $\bfa,\bar \bfa:\R^d\times\R^{d\times d}\to\R^{d\times d}$ we use the shorthand notation $\dd(\bfa,\bar \bfa)(x):=  \dd(\bfa(x,\cdot),\bar \bfa(x,\cdot))$ for $x\in\R^d$.
\end{definition}

\begin{remark}\label{rem:a1-a2}
 Note that $\dd$ defines a metric on $\mathcal A_\beta^0$ and in particular we have $\dd(\bfa,\bar \bfa)\leq \frac2\beta$ for all $\bfa,\bar \bfa\in\mathcal A_\beta^0$.
\end{remark}

Now, we state the crucial perturbation result:
\begin{proposition}\label{P:niren}
Fix $d\geq 2$, $\beta,s\in(0,1]$, $E<\infty$ and a ball $B\subset\R^d$. Let $\bfa, \bar \bfa$ be coefficient fields of class $\mathcal A_\beta$ and suppose that $\bar\bfa$ is layered. Suppose that 
\begin{equation}\label{ass:a-abd}
 \sup_{r\in (0,1]} r^{-s} \| \dd (\bfa,\overline \bfa)\|_{\underline L^2(rB)}\leq E.
\end{equation}
There exists $\overline \kappa_2=\overline \kappa_2(\beta,d,E,s)>0$ and $c=c(\beta,d,E,s)\in[1,\infty)$ such that if $u\in H^1(B)$ satisfy
\begin{equation*}
 \max\left\{\|\nabla u\|_{\underline L^2(B)},\sup_{r\in(0,1]}r^{-s}\|\divv \bfa(\nabla u)\|_{\underline H^{-1}(r B)}\right\}\leq\begin{cases}
                                      \infty&\mbox{if $d=2$,}\\
                                      \overline \kappa_2&\mbox{if $d\geq3$,}
                                     \end{cases}
\end{equation*}
then 
\begin{equation*}
  \sup_{r\in(0,1]}\|\nabla u\|_{\underline L^2(r B)}\leq c(\|\nabla u\|_{\underline L^2(B)}+ \sup_{r\in(0,1]}r^{-s}\|\divv \bfa(\nabla u)\|_{\underline H^{-1}(r B)}).
\end{equation*}
\end{proposition}

We provide a proof of Proposition~\ref{P:niren} in Section~\ref{sec:lipniren}. As the reader might already anticipate, Theorem~\ref{Th1} is now a rather direct consequence of Proposition~\ref{P:niren}:

\begin{proof}[Proof of Theorem~\ref{Th1}]

We first argue that in view of Proposition~\ref{P:niren}, it suffices to show that for every $z\in \frac12 B$ there exists a layered coefficient field $\overline \bfa$ of class $\mathcal A_\beta$ and $c=c(d)\in[1,\infty)$ such that
\begin{align}\label{eq:pfth11}
 &\sup_{r\in(0,1]}r^{-s}\|\dd (\bfa,\overline \bfa)\|_{\underline L^2(r B_\frac12 (z))}\leq c\frac{E}{\beta},\\
&\sup_{r\in(0,1]}r^{1-\frac{d}{q}}\|\divv \bfa (\nabla u)\|_{\underline H^{-1}(rB_\frac12(z))}\leq c\|f\|_{L^q(B)}.\label{eq:pfth12}
\end{align}
Indeed, with \eqref{eq:pfth11} and \eqref{eq:pfth12} at hand, we may choose $\overline \kappa=\overline \kappa(\beta,d,E,s,q)>0$ sufficiently small such that \eqref{ass:smallnesTh1} and Proposition~\ref{P:niren} imply that there exists $c=c(\beta,d,E,s,q)\in[1,\infty)$ such that
\begin{equation*}
 \sup_{r\in(0,1]}\|\nabla u\|_{\underline L^2(r B_\frac12(z))}\leq c(\|\nabla u\|_{L^2(B)}+ \|f\|_{L^q(B)}),
\end{equation*}
and the claim follows by the Lebesgue differentiation theorem and the arbitrariness of $z\in \frac12 B$.

\step1 Proof of \eqref{eq:pfth11}.

Fix $z\in \frac12 B$. Recall that by the definition of a $(E,s)$-regular coefficient field of class $\mathcal A_\beta$, there exists a tessellation $\mathcal D=\{D_\ell\}$ of $\R^d$ such that 
\begin{equation*}
 \bfa(x,F) =\sum_{\ell\in\Z^d}\bfa_\ell(F){\bf 1}_{D_\ell}(x)\qquad\mbox{for almost every $x\in\R^d$ and every $F\in\R^{d\times d}$,}
\end{equation*}
where $\bfa_\ell \in\mathcal A_\beta$ for every $\ell\in\Z$. Moreover, we find a layered tessellation $\mathcal D'=\{D_\ell'\}$ such that
\begin{equation}\label{est:Estes}
 \sup_{r>0}r^{-s}\left(|B_r|^{-1}\sum_{\ell\in\Z}|(D_\ell \Delta D'_\ell)\cap B_r(z)|\right)^\frac12\leq E.
\end{equation}
We denote by $\overline \bfa$ the layered coefficient of class $\mathcal A_\beta$ given by
\begin{equation*}
 \overline \bfa(x,F)=\sum_{\ell\in\Z^d}\bfa_\ell(F){\bf 1}_{D'_\ell}(x)\qquad\mbox{for almost every $x\in\R^d$ and every $F\in\R^{d\times d}$,}
\end{equation*}
and observe by \eqref{ass:lip} and \eqref{est:Estes} that 
\begin{align*}
 \|\dd (\bfa,\overline \bfa)\|_{\underline L^2(r B_\frac12 (z))}\leq \frac{2}{\beta} \left(|B_\frac{r}{2}|^{-1}\sum_{\ell\in\Z}|(D_\ell \Delta {D_z}_\ell)\cap (B_{\frac{r}2}(z))|\right)^\frac12\leq \frac{2}{\beta}\left(\frac{r}{2}\right)^{s}E,
\end{align*}
and thus \eqref{eq:pfth11} (with $c=2$).

\step{2} Proof of \eqref{eq:pfth12}.

Let $v\in H_0^1(B_\frac12(z))$ be the unique weak solution to $-\Delta v=f$ in $\mathscr D'(B_\frac12(z))$. By maximal regularity, we find $c=c(d,q)\in[1,\infty)$ such that
\begin{equation*}
 \|\nabla^2 v\|_{L^q(B_\frac12(z))}\leq c\|f\|_{L^q(B_\frac12(z))}\leq c\|f\|_{L^q(B)}.
\end{equation*}
Hence, using $\divv \bfa (\nabla u)=\divv \nabla v$ and the Poincar\'e inequality, we obtain for every $r\in(0,1)$
\begin{align*}
 \|\divv \bfa (\nabla u)\|_{\underline H^{-1}(rB_\frac12(z))}=&\|\nabla v - (\nabla v)_{rB_\frac12(z))}\|_{\underline L^2(rB_\frac12(z))} \leq c(d)r\|\nabla^2 v\|_{\underline L^2(rB_\frac12(z))}\\
 \leq& c(d)r^{1-\frac{d}{q}}\|\nabla^2 v\|_{L^q(B_\frac12(z))}\leq c(d,q)r^{1-\frac{d}{q}}\|f\|_{L^q(B)},
\end{align*}
which proves \eqref{eq:pfth12}.

\end{proof}

\subsection{Estimates for layered coefficients, Proofs of Proposition~\ref{L:estmorrey}, Corollary~\ref{P:nonlinearlayer}} \label{sec:estnonlinearlayer}

Using the difference quotient method and a reverse H\"older inequality, we obtain the following a priori higher differentiability and integrability result for layered coefficients:
\begin{lemma}[A priori estimate for layered coefficients]\label{L:energyestimates}
Fix $d\geq2$, $\beta\in(0,1]$ and suppose that  $\bfa$ is a $e_d$-layered coefficient field of class $\mathcal A_\beta^0$. There exist $c=c(\beta,d)\in[1,\infty)$ and $\mu=\mu(\beta,d)>0$ such that if $v\in H^1(B)$, where $B$ denotes a ball of radius $R>0$, satisfies
\begin{equation*}
 \divv \bfa(\nabla v) = 0 \qquad\mbox{in $\mathscr D'(B)$},
\end{equation*}
then
\begin{equation}\label{est:meyer:vi}
 \|\nabla \nabla' v\|_{\underline L^{2+\mu}(\frac12 B)}\leq cR^{-1}\|\nabla' v - (\nabla' v)_{B}\|_{\underline L^2(B)}.
\end{equation}
Furthermore, if $\bfa$ is additionally of class $\mathcal A_\beta$, it holds for $i=1,\dots,d-1$ that
\begin{equation}\label{eq:diffxi}
 \divv \mathbb L_v \nabla \partial_i v=0\qquad\mbox{in $\mathscr D'(B)$},
\end{equation}
where
\begin{equation}\label{def:Lv}
  \mathbb L_v:= D\bfa(\cdot,\nabla v).
\end{equation}

\end{lemma}

\begin{proof}[Proof of Lemma~\ref{L:energyestimates}]
Fix $i=1,\dots,d-1$. Since $\bfa(x_d,F)$ is independent of the $x_i$-variable, we can use the difference quotient method to obtain that $\partial_i v\in W_{\operatorname{loc}}^{1,2}(B)$ and for every ball $B_r(x)\subset B$ it holds
\begin{equation*}
 \|\nabla \partial_i v\|_{\underline L^2(B_\frac{r}{2}(x_0))}\leq cr^{-1}\|\partial_i v - (\partial_i v)_{B_r(x_0)}\|_{\underline L^2(B_r(x_0))}.
\end{equation*}
The higher integrability of $\nabla \partial_i v$ is now a consequence of the Poincar\'e-Sobolev inequality and the reverse H\"older inequality, see e.g.,~\cite[Theorem 6.38]{GM12}.
   
\end{proof}

Next, we recall a local Lipschitz estimate, due to Chipot et.~al.~\cite{CKV86}, for solutions of linear elliptic systems with 'layered' coefficients.
\begin{definition}
 For given $\beta\in(0,1]$, we denote by $\mathcal L_\beta$ the class of coefficients $\mathbb L\in \R^{d^4}$ satisfying
 \begin{equation*}
 \beta |G|^2\leq \langle \mathbb LG,G\rangle,\qquad |\mathbb L G|\leq \frac1\beta |G|\qquad\mbox{for every $G\in\R^{d\times d}$.}
\end{equation*}
\end{definition}

\begin{lemma}[\cite{CKV86} Theorem~2]\label{L:estlin1a}
Fix $d\geq2$ and $\beta\in(0,1]$. Let $\mathbb L\in L^\infty(\R^d,\R^{d^4})$ be such that $\mathbb L(x)=\mathbb L(x_d)\in\mathcal L_\beta$ for a.e. $x=(x',x_d)\in\R^d$. Let $u\in H^1(B)$  satisfy
\begin{equation}\label{eq:linpde1d}
 \divv \left(\mathbb L\nabla u\right)=0\quad\mbox{in $\mathscr D'(B)$},
\end{equation}
for some ball $B\subset\R^d$. Then, there exists $c=c(\beta,d)\in[1,\infty)$ such that
\begin{equation*}
 \|\nabla u\|_{L^\infty(\frac12 B)}\leq c\|\nabla u\|_{\underline L^2(B)}.
\end{equation*}
\end{lemma}
Lemma~\ref{L:estlin1a} is essentially contained in \cite[Theorem~2]{CKV86} (see also \cite[Theorem~1.6]{LN03}); for the convenience of the reader, we present a proof in the Appendix~\ref{sec:proofkinderlehrer}.

Before, we come to the proofs of Proposition~\ref{L:estmorrey} and Corollary~\ref{P:nonlinearlayer}, we state and prove a variant of the Poincar\'e inequality. 

\begin{lemma}\label{L:poincaretan}
Fix $d\geq2$. There exists $c=c(d)<\infty$ such that for all $u\in L_{\loc}^2(\R^d)$ and all $Q=Q_R(x_0)$, with $x_0\in\R^d$ and $R>0$, it holds
\begin{equation}\label{poincaretan}
 \|u-\overline{(u)}_{Q}\|_{\underline L^2(Q)}\leq c |Q|^\frac1d \|{\nabla'} u\|_{\underline L^2(Q)}, 
\end{equation}
where $\overline{(u)}_Q\in L_\loc^2(\R^d)$ is defined as
\begin{equation}\label{def:overlineq}
 \overline{(u)}_Q(x):=\fint_{x_0+(-\frac{R}{2},\frac{R}{2})^{d-1}}u(y',x_d)\,dy'.
\end{equation}

\end{lemma}
\begin{proof}

Without loss of generality, we can suppose $Q=(-\frac{R}{2},\frac{R}{2})^d$. Suppose $u\in C^1(\R^d)$. Then, $u_{z}(\cdot)=u(\cdot,z)$ satisfies $u_{z}\in C^1(\R^{d-1})$ for every $z\in\R$. The classical Poincar\'e inequality yields
\begin{align*}
 \fint_{(-\frac{R}2,\frac{R}2)^{d-1}}| u(x',z) - \fint_{(-\frac{R}2,\frac{R}2)^{d-1}} u(y',z)\,dy'|^2 \,dx' \leq c^2R^2\fint_{(-\frac{R}2,\frac{R}2)^{d-1}}|{\nabla'} u(x',z)|^2 \,dx'.
\end{align*}
Applying on both sides of the above inequality $\fint_{-\frac{R}{2}}^{\frac{R}{2}}\,dz$, we obtain \eqref{poincaretan} for $u\in C^1(\R^d)$. The general case follows by standard approximation arguments.
\end{proof}

\subsubsection{Proof of Proposition~\ref{L:estmorrey}}\label{sec:pf:L:estmorrey}

\begin{proof}[Proof of Proposition~\ref{L:estmorrey}]

Throughout the proof we write $\lesssim$ if $\leq$ holds up to a multiplicative constant which depends only on $\beta$ and $d$. Since, we use the Poincar\'e-type inequality Lemma~\ref{L:poincaretan} it is more convenient to work with cubes instead of balls. Fix $Q=Q_R(x_0)$ for some $x_0\in\R^d$ and $R>0$, and suppose that
\begin{equation}\label{eq:layercube}
 \divv \bfa(\nabla v)=0\qquad\mbox{in $\mathscr D'(Q)$.}
\end{equation}
We claim that there exist $\gamma\in(0,1)$, $\overline \kappa=\overline \kappa(\beta,d)>0$ and $c=c(\beta,d)\in[1,\infty)$ such that if $v$ satisfies \eqref{eq:layercube} and 
\begin{equation*}
 \|\nabla' v - (\nabla' v)_{Q}\|_{\underline L^2(Q)}\leq \begin{cases}
                                                          \infty&\mbox{if $d=2$,}\\
                                                          \overline \kappa&\mbox{if $d\geq3$,}
                                                         \end{cases}
\end{equation*}
then
\begin{align}\label{est:decaynablavcube}
 \sup_{r\in(0,1]}r^{-\gamma}\|\nabla' v - (\nabla' v)_{rQ}\|_{\underline L^2(r Q)}\leq& c\|{\nabla'} v - ({\nabla'} v)_{Q}\|_{\underline L^2(Q)}\\
 \sup_{r\in(0,1]}r^{-\gamma}\|J_d(v) - (J_d(v))_{rQ}\|_{\underline L^2(r Q)}\leq& c\|{\nabla'} v - ({\nabla'} v)_{Q}\|_{\underline L^2(Q)}\label{est:decayJcube}.
\end{align}
Clearly, this implies the statement of Proposition~\ref{L:estmorrey}. Moreover, we notice that Lemma~\ref{L:energyestimates} holds verbatim if we replace the ball $B$ by any cube $Q$, and we fix $\mu=\mu(\beta,d)>0$ to be the Meyers exponent given in  Lemma~\ref{L:energyestimates} (in its version for cubes). 

\smallskip

\step{1} One step improvement.

Define
\begin{equation*}
 \varphi(\tau,Q):= |\tau Q|^\frac{2}{d} \|\nabla \nabla' v\|_{\underline L^2(\tau Q)}^2\qquad(\tau>0,\quad Q\subset\R^d).
\end{equation*}
We claim that there exist $c_1=c_1(\beta,d)<\infty$ and $q=q(\beta,d)\in(0,1]$ such that for all $\tau\in (0,1]$, 
\begin{align}\label{apriorvarphi}
 \varphi(\tau,  Q) \leq& c_1\left(1+ \omega(c_1\varphi(1, Q)^\frac12)^q\tau^{-d}\right)\tau^2\varphi(1, Q),
\end{align}
where $\omega$ is given in \eqref{def:omega}.

\smallskip

Let $\mathbb L_v\in L^\infty(\R^d,\R^{d^4})$ be given by \eqref{def:Lv}, let $ \overline{(\nabla v)}_{Q}$ denote the average as in \eqref{def:overlineq}, and consider $\overline{\mathbb L}_{v}\in L^\infty(\R^d,\R^{d^4})$ given by
\begin{equation*}
 \overline{\mathbb L}_{v}(x):= D\bfa\left(x_d, \overline{(\nabla v)}_{Q}(x_d)\right).
\end{equation*}
Recall that $\bfa(x_d,\cdot)\in \mathcal A_\beta$ for a.e. $x_d\in\R$ implies
\begin{equation}\label{diff:lvlrx0}
 |\mathbb L_v-\overline{\mathbb L}_{v}|\leq \tfrac1\beta \omega(|\nabla v - \overline{(\nabla v)}_{Q}|)\quad\mbox{a.e.} 
\end{equation}
Let $w_i\in \partial_i v+H_0^1(\frac{1}2Q)$ be the unique solution to
\begin{equation}\label{eq:defwi:b}
 \divv \overline{\mathbb L}_{v}\nabla w_i=0\qquad\mbox{in $\mathscr D'(\frac{1}2Q)$}. 
\end{equation}
By the Lipschitz estimate Lemma~\ref{L:estlin1a}, we obtain for all $\tau\in(0,\frac12]$,
\begin{align*}
 \int_{\tau Q}|\nabla w_i(x)|^2\,dx\lesssim \tau^d\int_{\frac{1}2Q}|\nabla w_i(x)|^2\,dx\lesssim \tau^d\int_{Q}|\nabla \partial_i v(x)|^2\,dx,
\end{align*}
and thus 
\begin{align*}
 \int_{\tau Q}|\nabla \partial_i v(x)|^2\,dx \lesssim \tau^d\int_{Q}|\nabla \partial_i v(x)|^2\,dx +\int_{\frac{1}2Q} |\nabla \partial_i v(x) - \nabla w_i(x)|^2\,dx.
\end{align*}
Since $\partial_i v=w$ on $\partial(\frac12 Q)$, the ellipticity of $\overline{\mathbb L}_{v}$, \eqref{eq:diffxi} and \eqref{eq:defwi:b} yield
\begin{align*}
 \beta\|\nabla \partial_i v - \nabla w_i\|_{L^2(\frac{1}2Q)}^2
 \leq& \int_{\frac{1}2Q} \langle \overline{\mathbb L}_{v}(x_d) (\nabla w_i(x) - \nabla \partial_i v(x)), \nabla w_i(x) - \nabla \partial_i v(x)\rangle\,dx\\
 =&\int_{\frac{1}2Q} \langle (\mathbb L_v(x)- \overline{\mathbb L}_{v}(x_d)) \nabla \partial_i v(x), \nabla w_i(x) - \nabla \partial_i v(x)\rangle\,dx,
\end{align*}
and thus (using \eqref{diff:lvlrx0})
\begin{align*}
 \|\nabla \partial_i v - \nabla w_i\|_{L^2(\frac{1}2Q)}^2\lesssim \int_{\frac{1}2Q} \omega(|\nabla v(x)-\overline{(\nabla v)}_{Q}(x_d)|)^2|\nabla \partial_i v(x)|^2\,dx.
\end{align*}
By Meyers estimate \eqref{est:meyer:vi} and $\omega\leq1$, cf.~\eqref{def:omega}, we obtain
\begin{align*}
 \|\omega(|\nabla v- \overline{(\nabla v)}_{Q}|)\nabla \partial_i v\|_{\underline L^2(\frac{1}2Q)}
 \leq&\|\omega(|\nabla v- \overline{(\nabla v)}_{Q}|)\|_{\underline L^{2\frac{(2+\mu)}{\mu}}(\frac{1}2Q)}\|\nabla \partial_i v\|_{\underline L^{2+\mu}(\frac{1}2Q)}\\
 \lesssim&\|\omega(|\nabla v- \overline{(\nabla v)}_{Q}|)\|_{\underline L^1(\frac{1}2Q)}^\frac{\mu}{2(2+\mu)}\|\nabla \partial_i v\|_{\underline L^2(Q)}.
\end{align*}
Since $\omega$ is concave and monotone increasing, we obtain with the help of the Poincar\'e type inequality \eqref{poincaretan} and the H\"older inequality that there exists $c=c(d)<\infty$ such that
\begin{align*}
 \fint_{\frac{1}2Q} \omega(|\nabla v(x)-\overline{(\nabla v)}_{Q}(x_d)|)\,dx\leq& \omega\left(\fint_{\frac{1}2Q}|\nabla v(x) - \overline{(\nabla v)}_{Q}(x_d)|\,dx\right)\\
 \leq&  \omega(c\varphi(1,Q)^\frac12).
\end{align*}
Altogether, there exists $c=c(\beta,d)<\infty$ such that for all $i\in\{1,\dots,d-1\}$ and $\tau\in(0,\frac12]$, it holds
\begin{align*}
 \int_{\tau Q}|\nabla \partial_i v |^2\leq& c\left(\tau^d+\omega(c\varphi(1,Q)^\frac12)^\frac{\mu}{2+\mu}\right)\int_{Q}|\nabla \partial_i v|^2.
\end{align*}
The claim~\eqref{apriorvarphi} follows by summing over $i=1$ to $i=d-1$.

\step{2} Morrey space estimates for $\nabla \nabla' v$.

\smallskip

In this step we appeal to a smallness condition. We claim that for every $\gamma\in(0,1)$ there exists $\overline \kappa_1=\overline \kappa_1(\beta,d,\gamma)>0$ and $c=c(\beta,\gamma,d)<\infty$ such that \eqref{eq:layercube} and
\begin{equation}\label{eq:layersmallcube1}
 \|\nabla' v - (\nabla' v)_{Q}\|_{\underline L^2(Q)}\leq \overline \kappa_1,
\end{equation}
imply
\begin{equation}\label{est:zimorrey}
\sup_{r\in(0,\frac12]}r^{1-\gamma}\|\nabla \nabla' v\|_{\underline L^2(r Q)}\leq c|Q|^{-\frac1d}\|{\nabla'} v - ({\nabla'} v)_{Q}\|_{\underline L^2(Q)}.
\end{equation}
The inequality \eqref{est:zimorrey} follows from \eqref{apriorvarphi} by a standard iteration argument. For the convenience of the reader, we recall it here following closely the discussion in \cite[p.~170 ff.]{Giaquinta83}.

\substep{2.1} In order to iterate \eqref{apriorvarphi}, we have to establish smallness of $\varphi(1,Q)$ for some $Q\subset\R^d$. To this end, we note that there exists $c=c(\beta,d)<\infty$ such that 
\begin{equation}\label{varphiapriori1}
 \varphi(1,\tfrac12Q)\leq c \|{\nabla'} v - ({\nabla'} v)_{Q}\|_{\underline L^2(Q)}^2.
\end{equation}
Indeed, this is a direct consequence of \eqref{eq:diffxi} and the corresponding Caccioppoli inequality:
\begin{align*}
 \varphi(1,\tfrac12Q)=& |\tfrac12 Q|^\frac{2}{d}\|\nabla \nabla' v\|_{\underline L^2(\frac12 Q)}^2\lesssim\|\nabla' v - (\nabla' v)_{Q}\|_{\underline L^2(Q)}^2. 
\end{align*}

\substep{2.2} Iteration.

\smallskip

By Step~1, i.e.,~\eqref{apriorvarphi}, there exists $c_1=c_1(\beta,d)<\infty$ and $q=q(\beta,d)\in(0,1]$ such that for every $\tau\in(0,1)$ it holds
\begin{align*}
 \varphi(\tau,\tfrac12Q)\leq c_1(1+\omega(c_1 \varphi(1,\tfrac12Q)^\frac12)^q\tau^{-d})\tau^2\varphi(1,\tfrac12 Q).
\end{align*}
For given $\gamma\in(0,1)$, we choose $\tau=\tau(\beta,d,\gamma)\in(0,1)$ such that $2c_1\tau^{2(1-\gamma)}=1$. Furthermore, fix $\delta_1=\delta_1(\beta,d,\gamma)\in(0,1]$ satisfying $\omega(c_1 z^\frac12)^q\tau^{-d}\leq 1$ for all $z\in(0,\delta_1]$. Hence, for $\overline \kappa_1=\overline\kappa_1(\beta,\gamma,d)>0$, we obtain from \eqref{eq:layersmallcube1} and \eqref{varphiapriori1} that $\varphi(1,\frac12 Q)\leq\delta_1$ and thus with \eqref{apriorvarphi},
\begin{equation}\label{est:morrey:iter}
 \varphi(\tau,\tfrac12 Q)\leq \tau^{2 \gamma}\varphi(1,\tfrac12 Q).
\end{equation}
Clearly, \eqref{est:morrey:iter} can be iterated and we obtain for all $k\in\N$, 
\begin{align*}
 \varphi(\tau^k,\tfrac12Q)\leq& \tau^{2k\gamma}\varphi(1,\tfrac12 Q) \leq \delta_1 \tau^{2 k \gamma}.
\end{align*}
Thus, there exists $c=c(\beta,\gamma,d)<\infty$ such that for every $r\in(0,\frac{1}{2}]$ it holds
\begin{equation*}
 r^{2\gamma}\|\nabla \nabla' v\|_{\underline L^2(rQ)}^2\leq c |Q|^{-\frac{2}d} \|{\nabla'} v - ({\nabla'} v)_{Q}\|_{\underline L^2(Q)}^2, 
\end{equation*}
which proves the claim.

\step{3} Morrey space estimates for $\nabla J_d(v)$.

Suppose that \eqref{est:zimorrey} is valid. We claim that there exists $c=c(\beta,\gamma,d)<\infty$ such that 
\begin{equation}\label{est:sigmals}
 \sup_{r\in(0,\frac12]}r^{1-\gamma}\|\nabla J_d\|_{\underline L^2(r Q)}\leq c |Q|^{-\frac1d}\|{\nabla'} v - ({\nabla'} v)_{Q}\|_{\underline L^2(Q)}.
\end{equation}
First we notice that \eqref{est:zimorrey}, the chain rule and \eqref{ass:a}--\eqref{def:omega} imply that there exists $c=c(\beta,\gamma,d)<\infty$ such that
\begin{equation}\label{est:sigmamorrey1}
 \sup_{r\in(0,\frac12]}r^{1-\gamma}\|{\nabla'} \bfa(\cdot,\nabla v)\|_{\underline L^2(r Q)}\leq c|Q|^{-\frac1d}\|{\nabla'} v - ({\nabla'} v)_{Q}\|_{\underline L^2(Q)}.
\end{equation}
Clearly, \eqref{est:sigmamorrey1} implies \eqref{est:sigmals} with $\nabla J_d$ replaced by ${\nabla'}J_d$. Hence, it is left to show the corresponding estimate for $\partial_d J_d$. For this we use the equation \eqref{eq:nonlinearxdf}. For every $\eta\in C_c^\infty(Q,\R^d)$,
\begin{align*}
 \int_{Q}J_d(x) \cdot \partial_d \eta(x)\,dx=&-\sum_{j=1}^{d-1}\int_{Q} \bfa(x_d,\nabla v(x))e_j\cdot\partial_j \eta(x)\,dx\\
 =&\sum_{j=1}^{d-1}\int_{Q} \partial_j \bfa(x_d,\nabla v(x))e_j\cdot \eta(x)\,dx.
\end{align*}
Hence, $|\partial_d J_d|\lesssim |\nabla' J_d|\lesssim |\nabla' \bfa(\cdot,\nabla v)|$ a.e.~in $Q$, and in combination with \eqref{est:sigmamorrey1} we have \eqref{est:sigmals}.

\step{4} Conclusion

Choose, $\gamma=\gamma(\beta,d)\in(0,1)$ as
\begin{equation*}
 \gamma=\begin{cases}
         \frac{\mu}{2+\mu}&\mbox{if $d=2$,}\\\frac12&\mbox{if $d\geq3$.}
        \end{cases}
\end{equation*}
We first consider $d\geq3$. We  choose  $\overline \kappa=\overline \kappa(\beta,d,\frac12)>0$ to the effect that by the Poincar\'e inequality (cf.\ Lemma~\ref{L:poincaretan}), the combination of Step~2 and 3 yields  \eqref{est:decaynablavcube} and \eqref{est:decayJcube}. Finally, we consider the case $d=2$. Note that, it suffices to prove \eqref{est:zimorrey} for $\gamma=\frac{\mu}{2+\mu}$ without any smallness assumption. We have (thanks to the Meyer's estimate of Lemma~\ref{L:energyestimates})
\begin{align*}
 \sup_{r\in(0,\frac14]}r^{1-\gamma}\|\nabla\nabla' v\|_{\underline L^2(rQ)}\leq& \sup_{r\in(0,\frac14]}r^{\frac{2}{2+\mu}}\frac{|\frac14 Q|^\frac{1}{2+\mu}}{|r Q|^\frac{1}{2+\mu}}\|\nabla \nabla'v\|_{\underline L^{2+\mu}(\frac14Q)} \leq4^{-\frac{2}{2+\mu}}\|\nabla \nabla'v\|_{\underline L^{2+\mu}(\frac14Q)}
\\
 \lesssim& |Q|^{-\frac{1}{d}}\|\nabla' v - (\nabla' v)_{Q}\|_{\underline L^2(Q)},
\end{align*}
and thus \eqref{est:zimorrey} follows.

\end{proof}

\subsubsection{Proof of Corollary~\ref{P:nonlinearlayer}}\label{sec:nonlinearlayer}

\begin{proof}[Proof of Corollary~\ref{P:nonlinearlayer}]

Throughout the proof we write $\lesssim$ if $\leq$ holds up to a multiplicative constant which depends only on $\beta$ and $d$. 

Suppose that \eqref{kappasmall} holds with $\overline \kappa_1:=2^{-\frac{d}2}\overline \kappa_0$, with $\overline \kappa_0>0$, given as in Proposition~\ref{L:estmorrey}.

Fix $z\in \frac12 B$ and set $\hat B:=B_\frac{R}2(z)$, where $R>0$ denotes the radius of $B$. We claim that there exists $c=c(\beta,d)\in[1,\infty)$ such that
\begin{equation}\label{lip:layerz}
 \sup_{r\in(0,1]}\|\nabla u\|_{\underline L^2(r \hat B)}\leq c\|\nabla u\|_{\underline L^2(B)}.
\end{equation}
Clearly, estimate \eqref{lip:layerz}, the arbitrariness of $z\in \frac12 B$ and the Lebesgue differentiation theorem imply the desired Lipschitz estimate \eqref{P:nonlinearlayer:claim2}.

Let us now come to the proof of \eqref{lip:layerz}. In view of the trivial inequality $\|\nabla v\|_{\underline L^2(\hat B)}\leq 2^\frac{d}2\|\nabla v\|_{\underline L^2( B)}$ and the choice $\overline \kappa_1=2^{-\frac{d}2}\overline \kappa_0$, we can apply Proposition~\ref{L:estmorrey} with $B=\hat B$ and obtain that there exists $c_1=c_1(\beta,d)\in[1,\infty)$ and $\gamma\in(0,1)$ such that \eqref{kappasmall} implies 
\begin{equation}\label{est:layer1}
 \|\nabla' u - (\nabla' u)_{\tau \hat B}\|_{\underline L^2(\tau \hat B)}+\|J_d(u) - (J_d(u))_{\tau \hat B}\|_{\underline L^2(\tau \hat B)} \leq c_1 \tau^\gamma \|\nabla' u - (\nabla' u)_{\hat B}\|_{\underline L^2(\hat B)},
\end{equation}
for every $\tau\in(0,1)$. Choosing $\tau=\tau(\beta,d)\in(0,1)$ sufficiently small such that $ c_1\tau^\gamma\leq \frac12$, we can iterate the estimate \eqref{est:layer1} and obtain for $\hat B_k:=\tau^k\hat B$ that
\begin{equation*}
 \|\nabla' u - (\nabla' u)_{\hat B_k}\|_{\underline L^2(\hat B_k)}+\|J_d(u) - (J_d(u))_{\hat B_k}\|_{\underline L^2(\hat B_k)} \leq \frac1{2^k} \|\nabla' u - (\nabla' u)_{\hat B}\|_{\underline L^2(\hat B)}.
\end{equation*}
Hence for every $k\in\N$,
\begin{align*}
|(\nabla' u)_{\hat B_k}|\leq& \left|\sum_{i=0}^{k-1}(\nabla' u)_{\hat B_{i+1}}-(\nabla' u)_{\hat B_{i}}\right|+|(\nabla' u)_{\hat B}|\\
 \leq& \left(\sum_{i=0}^{k-1}\|\nabla' u - (\nabla' u)_{\hat B_{i}}\|_{\underline L^2(\hat B_{i+1})}\right)+\|\nabla' u\|_{\underline L^2(\hat B)}\\ 
 \leq&\left(\tau^{-\frac{d}{2}}\sum_{i=0}^{k-1}\|\nabla' u - (\nabla' u)_{\hat B_i}\|_{\underline L^2(\hat B_i)}\right)+\|\nabla u\|_{\underline L^2(\hat B)}\\
 \leq&\tau^{-\frac{d}{2}}\|\nabla' u - (\nabla' u)_{\hat B}\|_{\underline L^2(\hat B)}\sum_{i=0}^\infty2^{-i} +\|\nabla u\|_{\underline L^2(\hat B)}.
\end{align*}
Since $\tau=\tau(\beta,d)\leq1$, we obtain
\begin{align*}
 \|\nabla' u\|_{\underline L^2(\hat B_k)}\leq\|\nabla' u - (\nabla' u)_{\hat B_k}\|_{\underline L^2(\hat B_k)}+|(\nabla' u)_{\hat B_k}|\lesssim \|\nabla u\|_{\underline L^2(\hat B)}.
\end{align*}
Similarly,
\begin{equation*}
\|J_d(u)\|_{\underline L^2(\hat B_k)}\leq\|J_d(u) - (J_d( u))_{\hat B_k}\|_{\underline L^2(\hat B_k)}+|(J_d(u))_{\hat B_k}|\lesssim \|\nabla u\|_{\underline L^2(\hat B)}.
\end{equation*}
Hence, Lemma~\ref{l:edmonoton} (applied with $F=\nabla u$) yields
\begin{align*}
 \|\nabla u\|_{\underline L^2(\hat B_k)}\lesssim \|\nabla' u\|_{\underline L^2(\hat B_k)}+\|J_d(u)\|_{\underline L^2(\hat B_k)}\lesssim\|\nabla u\|_{\underline L^2(\hat B)},
\end{align*}
and thus \eqref{lip:layerz}.

\end{proof}

\subsection{Proof of Proposition~\ref{P:niren}}\label{sec:lipniren}

We start with a rather general perturbation lemma. The statement and its proof is similar to \cite[Lemma~3.1]{LN03} where an analogue result for linear systems is given.
\begin{lemma}[Perturbation lemma]\label{L:perturbation}
Fix $d\geq2$ and $\beta\in(0,1]$. Let $\bfa$ and $\bar \bfa$ be coefficient fields of class $\mathcal A_\beta^0$. There exist $c=c(\beta,d)\in[1,\infty)$ and $q=q(\beta,d)\in(0,1]$ such that for any ball $B\subset\R^d$ and every $w\in H^1(B)$ there exists $\bar w\in H^1(\frac12 B)$ with
\begin{align}
\divv \bar \bfa(\nabla\bar w)=&0\qquad\mbox{in $\mathscr D'(\frac12B)$,}\label{L:perturbation:eqv}\\
  \|\nabla\bar w\|_{\underline L^2(\frac12B)}\leq& c     \|\nabla w\|_{\underline L^2(B)},\label{est:pert2}\\
\| \nabla  w - \nabla \bar w\|_{\underline L^2(\frac12B)}\leq& c\left(\|\dd(\bfa,\bar \bfa)\|_{\underline L^2(B)}^q\|\nabla w\|_{\underline L^2(B)}+\|\divv \bfa(\nabla w)\|_{\underline H^{-1}(B)}\right).\label{est:pert}
\end{align}
\end{lemma}

\begin{proof}
In the following, we write $\lesssim$ for $\leq$ up to a multiplicative constant depending only on $\beta$ and $d$.

\smallskip

\step{1} The case with zero right-hand side.

Suppose that
\begin{equation}\label{pert:zero}
\divv \bfa(\nabla w)=0\qquad\mbox{in $\mathscr D'(B)$.}
\end{equation}
Let $\bar w\in H^1(\frac12 B)$ be the unique solution to
\begin{equation}\label{pert:eqv}
 \divv \overline \bfa(\nabla \bar w)=0\quad\mbox{in $\mathscr D'(\frac12 B)$}\quad\mbox{and}\quad \bar w- w\in H_0^1(\tfrac12 B).
\end{equation}
Clearly $\bar w$ satisfies \eqref{L:perturbation:eqv}. We claim that $\bar w$ satisfies \eqref{est:pert2} and \eqref{est:pert}

\substep{1.1} Proof of \eqref{est:pert2}.

Combining \eqref{ass:a}--\eqref{ass:lip} and \eqref{pert:eqv}, we obtain
\begin{equation*}
 \beta \|\nabla \bar w\|_{L^2(\frac12 B)}^2\stackrel{\eqref{ass:monotonicity}}{\leq} \int_{\frac12 B}\langle \overline \bfa(\nabla \bar w),\nabla \bar w\rangle \,dx\stackrel{\eqref{pert:eqv}}{=} -\int_{\frac12 B}\langle \overline \bfa(\nabla \bar w),\nabla  w\rangle \,dx\stackrel{\eqref{ass:a},\eqref{ass:lip}}{\leq} \frac1\beta \int_{\frac12 B} |\nabla \bar w| |\nabla  w|\,dx
\end{equation*}
and \eqref{est:pert2} follows from Youngs inequality.

\substep{1.2} Proof of \eqref{est:pert}.

By the interior higher integrability of $\nabla w$, there exists $\mu=\mu(\beta,d)>0$ such that
\begin{equation}\label{pert:intmeyer}
 \|\nabla w\|_{\underline L^{2+\mu}(\frac12 B)}\lesssim \|\nabla w\|_{\underline L^2(B)}.
\end{equation}
Furthermore, \eqref{pert:eqv} implies that with $F:=\overline \bfa (\nabla w - \nabla \bar w)- \overline \bfa(\nabla \bar w)$,
\begin{equation*}
 \divv \overline \bfa(\nabla w- \nabla \bar w)=\divv F\quad \mbox{in $\mathscr D'(\frac12 B)$.}
\end{equation*}
Combining \eqref{ass:lip} and \eqref{pert:intmeyer}, we obtain
\begin{equation*}
 \|F\|_{\underline L^{2+\mu}(\frac12 B)}\lesssim \|\nabla w\|_{\underline L^{2+\mu}(\frac12 B)}\lesssim \|\nabla w\|_{\underline L^2(B)}.
\end{equation*}
Hence, by a global version of Meyer's estimate (see Lemma~\ref{L:meyerglobal} below) we find $\mu_1=\mu_1(\beta,d)\in(0,\mu]$ such that 
\begin{equation*}
 \|\nabla w-\nabla \bar w\|_{\underline L^{2+\mu_1}(\frac12 B)}\lesssim \|F\|_{\underline L^{2+\mu_1}(\frac12 B)}\lesssim  \|\nabla w\|_{\underline L^2(B)}.
\end{equation*}
Equations \eqref{pert:zero} and \eqref{pert:eqv} imply for all $\eta\in H_0^1(\frac12 B)$,
\begin{align*}
 \int_{\frac12 B}\langle \bfa(\nabla w)- \bfa(\nabla \bar w),\nabla \eta\rangle\,dx=\int_{\frac12 B}\langle \overline\bfa(\nabla \bar w) - \bfa(\nabla \bar w),\nabla \eta\rangle\,dx.
\end{align*}
Testing with $\eta=w-\bar w$ and appealing to the monotonicity of $\bfa$ (cf.~\eqref{ass:monotonicity}), \eqref{def:a1-a2} and Remark~\ref{rem:a1-a2}, we obtain that
\begin{align*}
 \|\nabla (w-\bar w)\|_{\underline L^2(\frac12 B)}\lesssim& \||\dd (\bfa,\overline \bfa)||\nabla \bar w|\|_{\underline L^2(\frac12 B)}
 \leq \|\dd (\bfa,\overline \bfa)\|_{\underline L^\frac{2(2+\mu_1)}{\mu_1}(\frac12 B)}\|\nabla \bar w\|_{\underline L^{2+\mu_1}(\frac12 B)}\\
 \lesssim&\|\dd (\bfa,\overline \bfa)\|_{\underline L^2(B)}^\frac{\mu_1}{2+\mu_1}\|\nabla w\|_{\underline L^2(B)},
\end{align*}
and thus \eqref{est:pert} with $q=q(\beta,d)=\frac{\mu_1}{2+\mu_1}\in(0,1)$.

\step{2} The general case.

Let $v\in H^1(B)$ be the unique solution to
\begin{equation}\label{pert:eqvf}
 \divv \bfa(\nabla v)=0\quad\mbox{in $\mathscr D'(B)$}\quad\mbox{and}\quad v- w\in H_0^1(B).
\end{equation}
As in Substep~1.1, we obtain
\begin{equation}\label{pert:estv1}
 \|\nabla v\|_{L^2(B)}\lesssim \|\nabla w\|_{L^2(B)}.
\end{equation}
Moreover, the monotonicity of $\bfa$ and \eqref{pert:eqvf} yield
\begin{equation*}
 \beta\|\nabla w-\nabla v\|_{\underline L^2(B)}^2\leq \fint_B \langle \bfa(\nabla w)-\bfa(\nabla v),\nabla w-\nabla v\rangle\,dx=\fint_B \langle \bfa(\nabla w),\nabla w-\nabla v\rangle\,dx
\end{equation*}
and thus by definition of $\|\cdot\|_{\underline H^{-1}(B)}$
\begin{equation}\label{pert:estv2}
 \|\nabla w-\nabla v\|_{\underline L^2(B)}\lesssim \|\divv \bfa(\nabla w)\|_{\underline H^{-1}(B)}.
\end{equation}
Clearly \eqref{pert:eqvf}--\eqref{pert:estv2}, and Step~1 applied to $w=v$ yield the existence of $\bar w$ satisfying \eqref{L:perturbation:eqv}--\eqref{est:pert}.
\end{proof}

Finally, we come to the proof of Proposition~\ref{P:niren}, which is the core of the Lipschitz estimate:

\begin{proof}[Proof of Proposition~\ref{P:niren}]

Without loss of generality, we suppose that $\bar \bfa$ is a $e_d$-layered coefficient field. Throughout the proof we write $\lesssim$ if $\leq$ holds up to a multiplicative constant which depends only on $\beta$ and $d$. 

The proof is based on a decay estimate for the oscillation of the quantity
\begin{equation}\label{def:A}
 A:=(\nabla'u,J_d)\qquad\mbox{where}\qquad J_d:=\bfa(\cdot,\nabla u)e_d.
\end{equation}
Recall that in view of Lemma~\ref{l:edmonoton}, we find $\bar c=\bar c(\beta,d)\in[1,\infty)$ such that
\begin{equation}\label{est:au}
 \frac1{\bar c}|A|\leq |\nabla u|\leq \bar c |A| \qquad\mbox{a.e.}
\end{equation}

\step{1} Basic decay estimate.

Let $\gamma=\gamma(\beta,d)\in(0,1)$ be given as in Proposition~\ref{L:estmorrey}. In the following, $B_0\subset B$ denotes a arbitrary ball concentric with $B$. We claim that there exists $\overline\kappa=\overline\kappa(\beta,d)>0$, $c_1=c_1(\beta,d)\in[1,\infty)$ and $q=q(\beta,d)\in(0,1]$ such that if 
\begin{equation}\label{eq:smalluniren}
 \|A\|_{\underline L^2(B_0)} \leq\begin{cases}
                                         +\infty&\mbox{if $d=2$,}\\
                                         \overline \kappa&\mbox{if $d\geq3$,}
                                        \end{cases}
\end{equation}
then it holds for every $\tau\in(0,1]$,
\begin{align}\label{est:basicdecay}
 &\|A-(A)_{\tau B_0}\|_{\underline L^2(\tau B_0)}\notag\\
 \leq& c_1\big(\tau^\gamma \|A-(A)_{B_0}\|_{\underline L^2(B_0)}+\tau^{-\frac{d}{2}}\|\dd (\bfa,\overline \bfa)\|_{\underline L^2(B_0)}^q\|A\|_{\underline L^2(B_0)}+\tau^{-\frac{d}{2}}\|\divv \bfa(\nabla u)\|_{\underline H^{-1}(B_0)}\big).
\end{align}
Inequality \eqref{est:basicdecay} is a natural Schauder type estimate and  will be obtained by comparing $u$ with a suitable $\overline \bfa$ harmonic function. The first term on the right-hand side in \eqref{est:basicdecay} is a good term and comes from the regularity for layered coefficients, see Proposition~\ref{P:nonlinearlayer}. The second and third term are error terms which come from the perturbation of the coefficients and the right-hand side.

\smallskip

For the argument we note that by Lemma~\ref{L:perturbation}, there exists $\bar u\in H^1(\frac12 B_0)$, $c=c(\beta,d)\in[1,\infty)$ and $q=q(\beta,d)\in(0,1]$ such that
\begin{align}
 \divv \bar \bfa(\nabla \bar u)=&0\qquad\mbox{in $\mathscr D'(\frac12 B_0)$,}\label{eq:baru}\\
 \|\nabla \bar u\|_{\underline L^2(\frac12 B_0)}\leq& c\|\nabla u\|_{\underline L^2(B_0)} \label{est:baru2}\\
 \|\nabla u - \nabla \bar u\|_{\underline L^2(\frac12B_0)}\leq& c\left(\|\dd (\bfa,\bar \bfa)\|_{\underline L^2(B_0)}^q\|\nabla u\|_{\underline L^2(B_0)}+\|\divv \bfa(\nabla u)\|_{\underline H^{-1}(B_0)}\right)\label{est:baru1}.
\end{align}
Set 
\begin{equation*}
 \bar A:=(\nabla' \bar u,\bar J_d(\bar u))\qquad\mbox{where}\qquad \bar J_d(\bar u):=\bar \bfa(\cdot,\nabla \bar u)e_d.
\end{equation*}
Since the coefficients in \eqref{eq:baru} are $e_d$-layered, Proposition~\ref{L:estmorrey} and Corollary~\ref{P:nonlinearlayer} can be applied. For $d\geq3$ this requires the smallness conditions to be full filled: Let $\overline \kappa_0$, $\overline \kappa_1$ be as in Proposition~\ref{L:estmorrey} and Corollary~\ref{P:nonlinearlayer} and choose $\overline\kappa=\overline \kappa(\beta,d)>0$ sufficiently small such that $\overline \kappa\leq \frac{1}{c \bar c}\min\{\overline \kappa_0,\overline \kappa_1\}$, where $c,\bar c\in[1,\infty)$ are given as in \eqref{est:baru2} and \eqref{est:au}. Then \eqref{eq:smalluniren} implies $\|\nabla \bar u\|_{\underline L^2(\frac12 B_0)}\leq \min\{\overline k_0,\overline k_1\}$, and thus by Proposition~\ref{L:estmorrey} and Corollary~\ref{P:nonlinearlayer} there exists $\tilde c=\tilde c(\beta,d)\in[1,\infty)$ such that
\begin{align}
 \sup_{\tau\in(0,1]}\tau^{-\gamma}\|\bar A - (\bar A)_{\tau \frac12 B_0}\|_{\underline L^2(\tau \frac12 B_0)}\leq& \tilde c\|\nabla' \bar u - (\nabla' \bar u)_{\frac12 B_0}\|_{\underline L^2(\frac12B_0)},\label{est:ba0}\\
 \|\nabla \bar u\|_{L^\infty(\frac14B_0)}\leq& \tilde c\|\nabla \bar u\|_{\underline L^2(\frac12 B_0)}\label{est:ba1}.
\end{align}
From \eqref{est:ba0} we get for every $\tau\in(0,\frac12]$,
\begin{align}\label{est:aa1}
 \|A-(A)_{\tau\frac12 B_0}\|_{\underline L^2(\tau \frac12 B_0)}\leq& \|A - (\bar A)_{\tau\frac12 B_0}\|_{\underline L^2(\tau \frac12 B_0)}\notag\\
 \leq&  \|\bar A - (\bar A)_{\tau\frac12 B_0}\|_{\underline L^2(\tau \frac12 B_0)}+\|A - \bar A\|_{\underline L^2(\tau \frac12 B_0)}\notag\\
 \leq& \tilde c \tau^\gamma \|\nabla' \bar u - (\nabla' \bar u)_{\frac12 B_0}\|_{\underline L^2(\frac12 B_0)}+2^{-\frac{d}{2}}\tau^{-\frac{d}{2}}\|A - \bar A\|_{\underline L^2(\frac14 B_0)}.
\end{align}
The first term on the right-hand side in \eqref{est:aa1} can be estimated by
\begin{align}\label{est:aa2}
 \|\nabla' \bar u - (\nabla' \bar u)_{\frac12 B_0}\|_{\underline L^2(\frac12B_0)}\leq 2^\frac{d}{2}\|\nabla' u - (\nabla' u)_{B_0}\|_{\underline L^2(B_0)}+\|\nabla'  u - \nabla' \bar u\|_{\underline L^2(\frac12B_0)}.
\end{align}
A combination of \eqref{est:aa1}, \eqref{est:aa2} and the (yet to be proven) estimate
\begin{equation}\label{est:aa3}
 \|\nabla'  u - \nabla' \bar u\|_{\underline L^2(\frac12B_0)}+\|A - \bar A\|_{\underline L^2(\frac14B_0)} \lesssim \|\dd (\bfa,\bar \bfa)\|_{\underline L^2(B_0)}^q\|A\|_{\underline L^2(B_0)}+\|\divv \bfa(\nabla u)\|_{\underline H^{-1}(B_0)},
\end{equation}
implies for every $\tau\in(0,\frac14]$ that
\begin{align*}
  &\|A-(A)_{\tau B_0}\|_{\underline L^2(\tau B_0)}\\
 \lesssim& \tau^\gamma \|A-(A)_{B_0}\|_{\underline L^2(B_0)}+\tau^{-\frac{d}{2}}\|\dd (\bfa,\overline \bfa)\|_{\underline L^2(B_0)}^q\|A\|_{\underline L^2(B_0)}+\tau^{-\frac{d}{2}}\|\divv \bfa(\nabla u)\|_{\underline H^{-1}(B_0)}.
\end{align*}
Hence, it is left to show \eqref{est:aa3}. Clearly \eqref{est:au} and \eqref{est:baru1} imply
\begin{align}\label{est:nunbaru}
 \|\nabla  u - \nabla \bar u\|_{\underline L^2(\frac12B_0)}
 \leq& c\left(\bar c\|\dd (\bfa,\bar \bfa)\|_{\underline L^2(B_0)}^q\|A\|_{\underline L^2(B_0)}+\|\divv \bfa(\nabla u)\|_{H^{-1}(B_0)}\right).
\end{align}
Using the trivial pointwise estimate $|A-\bar A|\leq |\nabla u - \nabla \bar u| + |\bfa(\nabla u)-\overline \bfa (\nabla \bar u)|$, the Lipschitz-continuity of $\bfa$, cf.~\eqref{ass:lip} and the definition of $\dd(\bfa ,\overline \bfa)$, we obtain
\begin{align*}
 \|A - \bar A\|_{\underline L^2(\frac14 B_0)}\leq& 2^\frac{d}{2}\|\nabla u - \nabla \bar u\|_{\underline L^2(\frac12 B_0)}+\|\bfa (\nabla u) - \bar \bfa (\nabla \bar u)\|_{\underline L^2(\frac14 B_0)}\\
 \leq&2^\frac{d}{2}(1+\tfrac1\beta)\|\nabla u - \nabla \bar u\|_{\underline L^2(\frac12 B_0)}+\|\bfa (\nabla \bar u) - \bar \bfa (\nabla \bar u)\|_{\underline L^2(\frac14 B_0)}\\
 \leq&2^\frac{d}{2}(1+\tfrac1\beta)\|\nabla u - \nabla \bar u\|_{\underline L^2(\frac12 B_0)}+\|\dd(\bfa ,\bar \bfa)\|_{\underline L^2(\frac14B_0)} \|\nabla \bar u\|_{ L^\infty(\frac14 B_0)},
\end{align*}
and thus, by the estimates \eqref{est:nunbaru}, \eqref{est:baru2} and \eqref{est:ba1},
\begin{align*}
 \|A - \bar A\|_{\underline L^2(\frac14 B_0)}\lesssim&\|\dd (\bfa,\bar \bfa)\|_{\underline L^2(B_0)}^q\|A\|_{\underline L^2(B_0)}+\|\divv \bfa(\nabla u)\|_{H^{-1}(B_0)}+\|\dd(\bfa ,\bar \bfa)\|_{\underline L^2(\frac14B_0)} \|\nabla u\|_{\underline L^2( B_0)}.
\end{align*}
By $\|\dd (\bfa ,\overline \bfa)\|_{\underline L^2(B_0)}\leq\frac2\beta$ (cf.~Remark~\ref{rem:a1-a2}), $|\nabla u|\lesssim |A|$ (cf.~\eqref{est:au}), and $q\in(0,1]$, we arrive at 
\begin{align*}
 \|A - \bar A\|_{\underline L^2(\frac14 B_0)}\lesssim \|\dd (\bfa,\bar \bfa)\|_{\underline L^2(B_0)}^q\|A\|_{\underline L^2(B_0)}+\|\divv \bfa(\nabla u)\|_{\underline H^{-1}(B_0)}. 
\end{align*}
In combination with \eqref{est:nunbaru} we get \eqref{est:aa3}.

\step{2} Conditional iteration.

Consider $B_0:=r_0 B$ for some $r_0\in(0,\bar r_0]$ with $\bar r_0=\bar r_0(\beta,d,E,s)\in(0,1]$ specified below, see \eqref{small:r0}. Fix $\tau=\tau(\beta,d)\in(0,1)$ such that
\begin{equation}\label{tau:small}
 c_1\tau^\gamma = \frac1{4},
\end{equation}
where $\gamma=\gamma(\beta,d)\in(0,1)$ and $c_1=c_1(\beta,d)\in[1,\infty)$ are given as in Step~1, and set $B_k:=\tau^k B_0$ for $k\in\N$. Let $\overline\kappa=\overline\kappa(\beta,d)>0$ be as in Step~1. We claim that the following is true: Suppose for some $k\in \N_0$
\begin{equation}\label{cond:small}
 \|A\|_{\underline L^2(B_i)}\leq \begin{cases}
                                         \infty&\mbox{if $d=2$,}\\
                                         \overline \kappa&\mbox{if $d\geq3$,}
                                        \end{cases}\qquad\mbox{for all $i=0,\dots,k$,}
\end{equation}
then it holds (with the understanding $\sum_{i=0}^{k-1}=0$ for $k=0$) 
\begin{align}
\|A-(A)_{B_{k+1}}\|_{\underline L^2(B_{k+1})}\leq& \frac1{2}\|A-(A)_{B_k}\|_{\underline L^2(B_k)}  + \frac1{4} \left(\frac{r_0}{\bar r_0}\right)^{sq}\tau^{ksq}\left(|(A)_{B_k}|+H\right)\label{est:ai0}  \\
|(A)_{B_{k+1}}|\leq& c_2\left(\|A\|_{\underline L^2(B_0)}+ \left(\frac{r_0}{\bar r_0}\right)^{sq}\sum_{i=0}^{k-1} \tau^{isq}\left(|(A)_{B_i}|+ H\right)\right),\label{est:ai2}
\end{align}
where $c_2=c_2(\beta,d)=1+2\tau^{-\frac{d}{2}}$ and
\begin{equation}\label{def:H}
 H:=\sup_{r\in(0,1]}r^{-s}\|\divv \bfa(\nabla u)\|_{\underline H^{-1}(rB)}.
\end{equation}

\substep{2.1} Proof of \eqref{est:ai0}.

In view of Step~1, \eqref{tau:small} and \eqref{cond:small} imply for every $i=0,\dots,k$
\begin{align}\label{est:ai05}
  \|A-(A)_{B_{i+1}}\|_{\underline L^2(B_{i+1})}\leq& \frac1{4}\|A-(A)_{B_i}\|_{\underline L^2(B_i)}+c_1\tau^{-\frac{d}{2}}\|\dd (\bfa,\overline\bfa)\|_{\underline L^2(B_i)}^q\|A\|_{\underline L^2(B_i)}\notag\\
 &\qquad+c_1\tau^{-\frac{d}{2}}\|\divv \bfa(\nabla u)\|_{\underline H^{-1}(B_i)}.
\end{align}
Assumption \eqref{ass:a-abd} and the definition of $H$, see \eqref{def:H}, imply
\begin{equation}\label{decay:dd}
 \|\dd (\bfa ,\overline \bfa)\|_{\underline L^2(B_i)}\leq r_0^s \tau^{is}E\quad\mbox{and}\quad  \|\divv \bfa(\nabla u)\|_{\underline H^{-1}(B_i)}\leq r_0^s\tau^{is} H\leq r_0^{sq}\tau^{isq} H,
\end{equation}
where we use for the last inequality that $q,r_0,\tau \in(0,1]$. Choose $\bar r_0\in(0,1]$ sufficiently small such that
\begin{equation}\label{small:r0}
 c_1\tau^{-\frac{d}{2}}\bar r_0^{sq}\max\{E,1\}^q= \frac1{4}.
\end{equation}
Estimate \eqref{est:ai05}, the elementary inequality $\|A\|_{\underline L^2(B_i)}\leq \|A-(A)_{B_i}\|_{\underline L^2(B_i)}+|(A)_{B_i}|$, \eqref{decay:dd} and \eqref{small:r0} imply for all $r_0\in(0,\bar r_0]$ and all $i=0,\dots,k$ that
\begin{align}\label{est:ai}
  \|A-(A)_{B_{i+1}}\|_{\underline L^2(B_{i+1})}\leq& \frac1{2}\|A-(A)_{B_i}\|_{\underline L^2(B_i)}+\frac1{4} \tau^{isq}\left(\frac{r_0}{\bar r_0}\right)^{sq}\left(|(A)_{B_i}|+H\right),
\end{align}
and thus \eqref{est:ai0}.

\substep{2.2} Proof of \eqref{est:ai2}.

For every $i\in\N_0$, we have the following elementary inequality
\begin{align*}
 |(A)_{B_{i+1}}-(A)_{B_{i}}|\leq \|A - (A)_{B_i}\|_{\underline L^2(B_{i+1})}\leq \tau^{-\frac{d}{2}}\|A - (A)_{B_i}\|_{\underline L^2(B_i)},
\end{align*}
which implies
\begin{align}\label{eq:estai2a}
 |(A)_{B_{k+1}}|=|(A)_{B_0}|+\sum_{i=0}^k|(A)_{B_{i+1}}-(A)_{B_{i}}|\leq|(A)_{B_0}|+\tau^{-\frac{d}{2}}\sum_{i=0}^k\|A-(A)_{B_i}\|_{\underline L^2(B_i)}.
\end{align}
In particular, $|(A)_{B_{1}}|\leq (\tau^{-\frac{d}{2}}+1)\|A\|_{\underline L^2(B_0)}$, which proves the claim for $k=0$. Inequality \eqref{est:ai} yields for $k\in\N$
\begin{align*}
 \sum_{i=1}^k \|A-(A)_{B_i}\|_{\underline L^2(B_i)}\leq& \frac1{2}\sum_{i=0}^{k-1}\|A-(A)_{B_i}\|_{\underline L^2(B_i)}+\frac1{4} \left(\frac{r_0}{\bar r_0}\right)^{sq}\sum_{i=0}^{k-1}\tau^{isq}\left(|(A)_{B_i}|+H\right),
\end{align*}
and thus (by partially absorbing the first term on the right-hand side)
\begin{align}\label{eq:estai2b}
 \sum_{i=1}^k \|A-(A)_{B_i}\|_{\underline L^2(B_i)}\leq& \|A-(A)_{B_0}\|_{\underline L^2(B_0)}+\frac1{2} \left(\frac{r_0}{\bar r_0}\right)^{sq}\sum_{i=0}^{k-1}\tau^{isq}\left(|(A)_{B_i}|+H\right).
\end{align}
Combining \eqref{eq:estai2a} and \eqref{eq:estai2b}, we obtain
\begin{align*}
 |(A)_{B_{k+1}}|\leq& |(A)_{B_0}|+2\tau^{-\frac{d}{2}}\|A-(A)_{B_0}\|_{\underline L^2(B_0)}+\frac{\tau^{-\frac{d}2}}{2}\left(\frac{r_0}{\bar r_0}\right)^{sq} \sum_{i=0}^{k-1} \tau^{isq}\left(|(A)_{B_i}|+H\right),
\end{align*}
and thus \eqref{est:ai2}.

\step{3} Refined estimates.

In this step, we replace the recursively defined estimates for $|(A)_{B_k}|$ and $\|A - (A)_{B_k}\|_{\underline L^2(B_k)}$ in \eqref{est:ai2} and \eqref{est:ai0} by explicit estimates. Fix $t=t(\beta,d,s)\in(0,1]$ such that 
\begin{equation}\label{eq:r0small}
 t^{sq} = \frac{1}{2c_2}\left(\sum_{i=0}^\infty\tau^{isq}\right)^{-1}.
\end{equation}
Fix $r_0\in(0,t\bar r_0]$. We claim that for any $k\in\N_0$, \eqref{cond:small} implies 
\begin{align}\label{claim:ak}
 |(A)_{B_i}|\leq& 2c_2 \left(\|A\|_{\underline L^2(B_0)}+ H\right)\qquad\mbox{for all $i=0,\dots,k+1$},\\
 \|A-(A)_{B_{k+1}}\|_{\underline L^2(B_{k+1})}\leq& 2^{-(k+1)}\|A\|_{\underline L^2(B_0)}+ 2^{-1} \sum_{i=0}^k\tau^{isq}\left(\|A\|_{\underline L^2(B_0)}+ H\right).\label{claim:ak2}
\end{align}
We prove the claim by induction. Suppose \eqref{cond:small} for $k=0$. Then, \eqref{est:ai2} and \eqref{est:ai0} imply
\begin{equation*}
 |(A)_{B_1}|\leq c_2 \|A\|_{\underline L^2(B_0)},\quad \|A-(A)_{B_1}\|_{\underline L^2(B_1)}\leq \tfrac12\|A-(A)_{B_0}\|_{\underline L^2(B_0)}+\tfrac1{4} \left(\tfrac{r_0}{\bar r_0}\right)^{sq}\left(|(A)_{B_0}|+H\right),
\end{equation*}
and thus \eqref{claim:ak} and \eqref{claim:ak2}. Suppose the claim is valid for some $k\in\N_0$. We only need to show that \eqref{cond:small} with $k$ replaced by $k+1$, implies \eqref{claim:ak} and \eqref{claim:ak2} for $k$ replaced by $k+1$.

\substep{3.1} Proof of \eqref{claim:ak} with $k$ replaced by $k+1$.

By induction hypothesis \eqref{claim:ak}, we have $|(A)_{B_i}|\leq 2c_2 \left(\|A\|_{\underline L^2(B_0)}+ H\right)$ for $i=0,\dots,k+1$. Hence, by \eqref{est:ai2} for $r_0\in(0,t\bar r_0]$ we have,
\begin{align*}
  |(A)_{B_{k+2}}|\leq& c_2\left(\|A\|_{\underline L^2(B_0)}+ \left(\frac{r_0}{\bar r_0}\right)^{sq}\sum_{i=0}^{k} \tau^{isq}\left(|(A)_{B_i}|+ H\right)\right)\\
  \leq& c_2\left(\|A\|_{\underline L^2(B_0)}+ t^{sq}\sum_{i=0}^{\infty} \tau^{isq}\left(2c_2 \left(\|A\|_{\underline L^2(B_0)}+ H\right)+ H\right)\right)\\
  =&c_2\left(2\|A\|_{\underline L^2(B_0)}+(1+\frac1{2c_2})H\right)\\
  \leq& 2c_2\left(\|A\|_{\underline L^2(B_0)}+ H\right),
\end{align*}
which proves \eqref{claim:ak} with $k$ replaced by $k+1$.

\substep{3.2} Proof of \eqref{claim:ak2} with $k$ replaced by $k+1$.

Estimate \eqref{est:ai0} yields for $r_0\in(0,t\bar r_0]$ 
\begin{equation*}
\|A-(A)_{B_{k+2}}\|_{\underline L^2(B_{k+2})}\leq \frac1{2}\|A-(A)_{B_{k+1}}\|_{\underline L^2(B_{k+1})} +\frac1{4} t^{sq}\tau^{(k+1)sq}\left(|(A)_{B_{k+1}}|+H\right).
\end{equation*}
The induction hypothesis \eqref{claim:ak2} implies
\begin{equation*}
\frac1{2}\|A-(A)_{B_{k+1}}\|_{\underline L^2(B_{k+1})}\leq 2^{-(k+2)}\|A\|_{\underline L^2(B_0)}+ 2^{-2} \sum_{i=0}^k\tau^{isq}\left(\|A\|_{\underline L^2(B_0)}+ H\right).
\end{equation*}
Moreover, by the induction hypothesis \eqref{claim:ak} and $t^{sq}\leq \frac1{2c_2}$ (cf.\ \eqref{eq:r0small}), we obtain that
\begin{align*}
\frac1{4} t^{sq}\tau^{(k+1)sq}\left(|(A)_{B_{k+1}}|+H\right)\leq&\frac1{4} \frac1{2c_2}\tau^{(k+1)sq}\left(2c_2\left(\|A\|_{\underline L^2(B_0)}+H\right)+H\right)\\
\leq& \frac12 \tau^{(k+1)sq}\left(\|A\|_{\underline L^2(B_0)}+H\right),
\end{align*}
%
and the desired estimate \eqref{claim:ak2} with $k$ replaced by $k+1$ follows.

\step{4} Conclusion.

Let $\bar r_0,t$ and $\tau$ be as above and fix $r_0=t\bar r_0$. We show that there exists $\tilde \kappa=\tilde \kappa(\beta,d,s)>0$ and $c=c(\beta,d,s)\in[1,\infty)$ such that the smallness condition
\begin{equation*}
\max\{\|\nabla u\|_{\underline L^2(B_0)},H\}\leq \begin{cases}\infty&\mbox{if $d=2$,}\\\tilde\kappa&\mbox{if $d\geq3$},\end{cases}
\end{equation*}
implies
\begin{equation}\label{claimqk}
 \|\nabla u\|_{\underline L^2(B_k)}\leq c(\|\nabla u\|_{\underline L^2(B_0)}+H)\qquad\mbox{for all $k\in\N$.}
\end{equation}
Note that since $B_k=\tau^k r_0 B$ with $\tau=\tau(\beta,d)\in(0,1)$, estimate \eqref{claimqk} yields
\begin{equation*}
 \sup_{r\in(0,r_0]}\|\nabla u\|_{\underline L^2(rB)}\lesssim c(\|\nabla u\|_{\underline L^2(B_0)}+H)\leq cr_0^{-\frac{d}{2}}(\|\nabla u\|_{\underline L^2(B)}+H),
\end{equation*}
and thus the claim of Proposition~\ref{P:niren} is proven with the choice $\overline \kappa_0=r_0^\frac{d}{2}\tilde \kappa$.

\substep{4.1} The case $d=2$.

For $d=2$ the smallness condition is not active and \eqref{cond:small} is trivially satisfied for all $k\in\N$. Set 
\begin{equation}\label{def:hatc}
 \hat c:=2c_2+\sum_{i=0}^\infty 2^{-isq},
\end{equation}
where $c_2=c_2(\beta,d)\in[1,\infty)$ is given as in Step~2. In view of \eqref{est:au} it suffices to prove%
\begin{equation}\label{est:final2d}
  \|A\|_{\underline L^2(B_{k})}\leq \hat c \left(\|A\|_{\underline L^2(B_0)}+ H\right),
\end{equation}
to infer \eqref{claimqk} with $c=\hat c \bar c^2$. For \eqref{est:final2d}, we estimate the right-hand side in \eqref{claim:ak2} by
\begin{equation*}
 2^{-(k+1)}\|A\|_{\underline L^2(B_0)}+ 2^{-1} \sum_{i=0}^k\tau^{isq}\left(\|A\|_{\underline L^2(B_0)}+ H\right)\leq \sum_{i=0}^{k+1}2^{-isq}\left(\|A\|_{\underline L^2(B_0)}+ H\right),
\end{equation*}
where we use that by definition $\tau\in(0,\frac14)$ (see \eqref{tau:small}, where $c_1\geq1$ and $\gamma\in(0,1)$). In combination with \eqref{claim:ak} and \eqref{claim:ak2} we get
\begin{align*}
 \|A\|_{\underline L^2(B_k)}\leq& (\|A-(A)_{B_k}\|_{\underline L^2(B_k)}+|(A)_{B_k}|\\
 \leq& \left(\sum_{i=0}^{k+1}2^{-isq}+2c_2\right)\left(\|A\|_{\underline L^2(B_0)}+ H\right)
\end{align*}
and thus \eqref{est:final2d}.

\substep{4.2} The case $d\geq3$.

Let $\bar c=\bar c(\beta,d)\in[1,\infty)$ and $\hat c=\hat c(\beta,d,s)\in[1,\infty)$ be given by \eqref{est:au} and \eqref{def:hatc}. We claim that
\begin{equation}\label{est:finalsmallness}
 \max\{\bar c\|\nabla u\|_{\underline L^2(B_0)},H\}\leq \frac1{2 \hat c} \overline \kappa, 
\end{equation}
implies that \eqref{cond:small} is true for every $k\in\N$. The proof follows by induction. In view of \eqref{est:au}, the case $k=0$ is trivial. Suppose \eqref{cond:small} is true for some $k\in\N$. In view of Step~3, we have
\begin{align*}
 \|A\|_{\underline L^2(B_{k+1})} \leq& \hat c\left(\|A\|_{\underline L^2(B_0)}+H\right)\leq \hat c\left(\bar c\|\nabla u\|_{\underline L^2(B_0)}+H\right)\leq \overline \kappa,
\end{align*}
which proves the claim. Now, we can argue exactly as in Supstep~4.1 to show that \eqref{est:finalsmallness} implies \eqref{est:final2d}.

%

\end{proof}

\section{Uniform Lipschitz estimates}\label{sec:lip:hom}

In this section, we provide a proof of Theorem~\ref{T:Lipschitzeps}. For this we follow the idea introduced by Avellaneda \& Lin \cite{AL87} (for linear systems) to lift the $C^{1,\alpha}$-regularity of the homogenized equation to the operator with oscillating coefficients by appealing to a two-scale expansion. In the nonlinear situation considered here, solutions of the homogenized problem are in general not regular everywhere. However, they satisfy an excess decay provided the excess is small on some scale in the following sense: 
\begin{lemma}[Excess decay for constant coefficient monotone systems]\label{L:regconstant}
Fix $d\geq2$, $\beta\in(0,1]$, a modulus of continuity $\omega$ and $\bfa\in \mathcal A_{\beta,\omega}$.  For every $\gamma\in(0,1)$ there exist $\overline \kappa=\overline \kappa(\beta,d,\omega)>0$ and $c=c(\beta,d,\omega)\in[1,\infty)$ such that the following is true: Suppose that $v\in H^1(B)$ satisfies
\begin{equation}\label{eq:constant}
 \divv \bfa (\nabla v)=0\qquad\mbox{in $\mathscr D'(B)$,}
\end{equation}
and $\|\nabla v - (\nabla v)_{B}\|_{\underline L^2(B)}\leq \overline \kappa_0$. Then, 
\begin{align}\label{est:decaynablavconstant}
 \sup_{r\in(0,1]}r^{-\gamma}\|\nabla v - (\nabla v)_{rB}\|_{\underline L^2(r B)}\leq c\|{\nabla} v - ({\nabla} v)_{B}\|_{\underline L^2(B)}.
\end{align}
Moreover, for $d=2$, there exists $\gamma_{2d}>0$ such that \eqref{eq:constant} implies \eqref{est:decaynablavconstant} with $c=c(\beta)\in[1,\infty)$. 
\end{lemma}

Lemma~\ref{L:regconstant} is well-known and can be found e.g.,\ in \cite{Giaquinta83,Giu03}. A possible proof goes along the lines of Lemma~\ref{P:nonlinearlayer} (Step~1,2 and 4), where $\nabla'$ is replaced by the full gradient $\nabla$, and $\overline{(\nabla v)_Q}$ is replaced by $(\nabla v)_Q$.  

\smallskip

Next we establish an analogous large-scale result in the case of periodic coefficients. We emphasize that the following result holds without (spatially) regularity assumptions on the coefficients.
\begin{proposition}[Large scale excess decay]\label{P:largescale:holder}
Fix $d\geq2$, $\beta\in(0,1]$ and let $\bfa$ be a periodic coefficient field of class $\mathcal A_\beta$. For every exponent $\gamma\in(0,1)$ there exists $c=c(\beta,\gamma,d)\in[1,\infty)$ and $\kappa_0=\kappa_0(\beta,\gamma,d)>0$ such that the following is true: Suppose $u\in H^1(Q_R)$, $R\geq1$, satisfies
\begin{equation}\label{eq:excessdecay}
  \divv\bfa(\nabla u)=0\qquad\mbox{in }\mathscr D'(Q_R),
\end{equation}
and $\operatorname{Exc}(\nabla u,Q_R)\leq \overline \kappa_0$, where for $\Omega\subset\R^d$ open and bounded
\begin{equation}\label{def:excess}
  \operatorname{Exc}(g,\Omega):=\inf_{F\in\R^{d\times d}}\|g-(F+\nabla \phi(F))\|_{\underline L^2(\Omega)}.
\end{equation}
Then,
\begin{equation}\label{est:excessdecay}
  \forall 1\leq r\leq R\,:\qquad \operatorname{Exc}(\nabla u,Q_r)\leq c(\frac r R)^{\gamma}\operatorname{Exc}(\nabla u,Q_R).
\end{equation}
Moreover, for $d=2$ equation \eqref{eq:excessdecay} implies \eqref{est:excessdecay} with $\gamma=\frac{\gamma_{2d}}2$ and $c=c(\beta)\in[1,\infty)$.
\end{proposition}
As mentioned above, the proof of Proposition~\ref{P:largescale:holder} follows the philosophy of \cite{AL87} by comparing $u$ with suitable solutions of the homogenized equation $\divv \bfa_0(\nabla u_0)=0$. In contrast to the compactness argument used in \cite{AL87}, we rely here on quantitative estimates on the homogenization error in the spirit of more recent works on the regularity of solutions for equations with random coefficients \cite{AM16,AS16,GNO14}. In particular, the excess measure defined in \eqref{def:excess}, which invokes the $\bfa$-dependent corrector $\phi$, has been used in \cite{GNO14} to obtain various larges scale regularity estimates.

\begin{proof}[Proof of Proposition~\ref{P:largescale:holder}]
 
Throughout the proof, we write $\lesssim$ whenever it holds $\leq$ up to a multiplicative constant which depends only on $\beta$ and $d$.
 
\smallskip

Let $u_0\in H^1(Q_\frac{R}{2})$ be the unique solution to the homogenized problem
\begin{equation}\label{eq:u0}
    \divv(\bfa_0(\nabla u_0))\,=\,0\mbox{ in }\mathscr D'(Q_\frac{R}2)\qquad  u_0-u\in H_0^1(Q_\frac{R}2).
\end{equation}
The method of \cite{AL87,GNO14} requires to estimate the error of the two-scale expansion, which formally reads $w(x)=u(x)-(u_0(x)+\phi(x,\nabla u_0(x)))$. In fact, in order to make use of the equations for $u$ and $u_0$, we require $w$ to vanish on $\partial Q_{\frac{R}2}$, and thus, one introduces a boundary layer. In the linear case, the corrector term tensorizes, i.e., $\phi(x,\nabla u_0(x))=\phi(x,e_i)\partial_i u_0(x)$. This is not true in the nonlinear case, which introduces additional difficulties, e.g.,~a priori it is not obvious that $x\mapsto \phi(x,\nabla u_0(x))$ is locally weakly-differentiable. To circumvent this issue in our argument we introduce an intermediate scale $r\geq 2$ with $\frac{R}{2r}\in\N$ (for convenience) and approximate $\nabla u_0$ and $\nabla\phi(\nabla u_0)$ on scale $r$ as follows: We introduce a collection of cubes of size $r$ that partition $Q_\frac{R}2$,
\begin{equation*}
  \mathcal Q:=\{Q_r(z)\,:\,z\in r\Z^d,\,Q_r(z)\subset Q_\frac{R}{2}\},
\end{equation*}
and define
\begin{equation*}
  \chi(x):=\sum_{Q\in\mathcal Q}{\bf 1}_{Q}(x)\nabla\phi(x,(\nabla u_0)_Q).
\end{equation*}
Note that by construction we have for any $Q\in\mathcal Q$, and any unit cube $\Box\subset Q$, the identity
\begin{equation*}
  \bfa_0((\nabla v)_Q)=\fint_{\Box}\bfa(x,(\nabla v)_Q+\chi(x))\,dx.
\end{equation*}
Next, we construct an approximate potential $\psi$ for the vector field $\chi$. To that end for $Q=Q_r(z)$ we denote by $\eta_Q\in C^\infty_c(\R^d)$ a cut-off function satisfying $\eta_Q=1$ (resp. $\eta_Q=0$) in $Q_{r-1}(z)$ (resp. outside $Q_r(z)$), $0\leq \eta_Q\leq 1$, and $|\nabla\eta_Q|\leq 2$; and set
\begin{equation*}
  \psi(x):=\sum_{Q\in\mathcal Q}\eta_Q(x)\phi(x,(\nabla u_0)).
\end{equation*}
In the proof we first gather in Step~1 a couple of auxilliary estimates regarding (energy based) regularity estimates for $u_0$ and an approximation error for $\psi$. In Step~2 we establish an error estimate for the approximate two-scale expansion $\nabla u-\nabla(u_0+\psi)$. In Step~3 and~4, we use the regularity in the form of Lemma~\ref{L:regconstant} of $u_0$ to establish the decay of the excess. Finally in Step~5, we comment on the case $d=2$.

\step{1} Auxilliary estimates.

{\it Substep 1.1} $H^1$-regularity and higher integrability of $\nabla u_0$.

We claim that there exists $\mu=\mu(\beta,d)> 0$ such that for all $\rho\in(0,\frac14)$,
\begin{align}
 \|\nabla u_0\|_{\underline L^2(Q_\frac{R}2)}\lesssim&\, \|\nabla u\|_{\underline L^2(Q_\frac{R}2)},\label{EXC:eq:1:1a}\\
 \|\nabla^2 u_0\|_{\underline L^2( (1-\rho)Q_{\frac{R}2})}\lesssim&\, (\rho R)^{-1}\|\nabla u\|_{\underline L^2(Q_{\frac{R}2})},\label{EXC:eq:1:1b}\\
 \|\nabla u_0\|_{\underline L^{2+\mu}(Q_\frac{R}2)}\lesssim&\, \|\nabla u\|_{\underline L^2(Q_{R})}.\label{EXC:eq:1:1c}
\end{align}
Indeed, since $\bfa_0\in \mathcal A_{\beta^3}^0$ (see Lemma~\ref{L:phi}), \eqref{EXC:eq:1:1a} is a standard energy estimate for \eqref{eq:u0}, and \eqref{EXC:eq:1:1b} is a standard interior $H^2$-estimate that can be proven via the difference quotient method. For \eqref{EXC:eq:1:1c} observe that the interior Meyers estimate yields
\begin{equation*}
 \|\nabla u\|_{\underline L^{2+\mu_1}(Q_\frac{R}{2})}\lesssim \|\nabla u\|_{\underline L^2(Q_{R})},
\end{equation*}
for some $\mu_1=\mu_1(\beta,d)>0$, and thus \eqref{EXC:eq:1:1c} is a consequence of a global Meyers estimate (see Lemma~\ref{L:meyerglobal}).
\medskip

{\it Substep 1.2} Approximation error regarding $\psi$.

We claim that for every $Q\in \mathcal Q$
\begin{align}
  \|\nabla\psi\|_{\underline L^2(Q)}\lesssim&\, |(\nabla u_0)_Q|,\label{EXC:eq:1:2a}\\
  \|\nabla\psi-\chi\|_{\underline L^2(Q)}\lesssim&\, r^{-\frac12}|(\nabla u_0)_Q|,\label{EXC:eq:1:2b}\\
  \|\bfa_0(\nabla u_0)-\bfa(\nabla (u_0+\psi))\|_{\underline H^{1}(Q)^*}^2\lesssim&\, \frac1r|(\nabla u_0)_Q|^2+r^2\|\nabla^2u_0\|_{\underline L^2(Q)}^2,\label{EXC:eq:2:2}
\end{align}
where $\|\cdot\|_{\underline H^1(Q)^*}$ denotes the dual norm defined by
\begin{equation*}
  \|g\|_{\underline H^1(Q)^*}:=\sup\Big\{\fint_{Q}\langle g,\nabla\varphi\rangle\,|\,\varphi\in H^1(Q),\, \|\nabla\varphi\|_{\underline L^2(Q)}\leq 1\,\Big\}.
\end{equation*}
For the argument fix $Q\in\mathcal Q$ and set $F:=(\nabla u_0)_Q$. We first prove the auxiliary estimate
\begin{equation}\label{locest:p1}
  \left(\fint_Q\big(|1-\eta_Q|(|\nabla\phi(F)|+|\nabla \sigma(F)|)+|\nabla\eta_Q|(|\phi(F)|+|\sigma(F)|)\big)^2\right)^\frac12\lesssim r^{-\frac12} |F|.
\end{equation}
Indeed, let $\mathcal Q_1\subset\{\Box:=Q_1(x)\,|\,x\in\R^d\}$ denote a disjoint covering of $S_1:=\{x\in Q\,:\,\dist(x,\partial Q)\leq 1\}$. Since $r\geq 2$ we have $\#\mathcal Q_1\lesssim r^{d-1}$. Note that the supports of $\eta_Q-1$ and $\nabla\eta_Q$ are contained in $S$. By appealing additionally to the bounds $|\eta_Q|,|\nabla\eta_Q|\leq 2$, and the periodicity of $\phi$ and $\sigma$, we get
\begin{eqnarray*}
  &&\fint_Q\big(|1-\eta_Q|(|\nabla\phi(F)|+|\nabla\sigma(F)|)+|\nabla\eta_Q|(|\phi(F)|+|\sigma(F)|)\big)^2\\
  &\leq &4r^{-d}\sum_{\Box\in\mathcal Q_1}\int_{\Box}(|\nabla\phi(F)|+|\nabla\sigma(F)|+|\phi(F)|+|\sigma(F)|)^2\\
  &\lesssim& r^{-1}\int_{Q_1(0)}(|\nabla\phi(F)|+|\nabla\sigma(F)|+|\phi(F)|+|\sigma(F)|)^2.
\end{eqnarray*}
To complete the argument, note that the integral on the right-hand side is bounded by $|F|^2$, as follows from the Lipschitz continuity estimate \eqref{est:lipphisigma}, and the fact that $\phi(0)=\sigma(0)\equiv 0$.

\smallskip

Estimate \eqref{EXC:eq:1:2b} is a direct consequence of \eqref{locest:p1}, since $\nabla\psi-\chi=\nabla(\eta_Q\phi(F))-\nabla\phi(F)=(\eta_Q-1)\nabla\phi(F)+\nabla\eta_Q\otimes\phi(F)$ on $Q$. Moreover, \eqref{EXC:eq:1:2a} follows from \eqref{EXC:eq:1:2b} by the triangle inequality, since $\|\nabla\phi(F)\|_{\underline L^2(Q)}\lesssim |F|$.
\medskip

Next, we prove \eqref{EXC:eq:2:2}. By the definition of the dual norm, the definition of $\psi$, a density argument, and the homogeneity of the estimate, it suffices to prove for all $\varphi\in H^2(Q)$ with $\|\nabla\varphi\|_{\underline L^2(Q)}=1$, 
\begin{equation*}
 \fint_Q\langle \bfa_0(\cdot,\nabla u_0)- \bfa(\nabla(u_0+\eta_Q\phi(F))), \nabla\varphi\rangle\lesssim \big(r^{-\frac 12}|(\nabla u_0)_Q|+r\|\nabla^2u_0\|_{\underline L^2(Q)}\big).
\end{equation*}
Consider the decomposition
\begin{eqnarray*}
  \bfa(\cdot,\nabla(u_0+\eta_Q\phi(F))- \bfa_0(\nabla u_0)&=&J(F)+\delta\bfa\\
  \text{with}\qquad J(F)&:=&\bfa(\cdot,F+\nabla\phi(F))-\bfa_0(F),\\
  \delta\bfa&:=&\bfa(\cdot,\nabla(u_0+\eta_Q\phi(F)))-\bfa(\cdot,F+\nabla\phi(F))+\bfa_0(F)-\bfa_0(\nabla u_0).
\end{eqnarray*}
Since $\nabla(\eta_Q\phi(F))=\nabla\phi(F)+(1-\eta_Q)\nabla\phi(F)+\nabla\eta_Q\otimes \phi(F)$, the Lipschitz continuity of $\bfa$ and $\bfa_0$ yields $|\delta\bfa|\lesssim |\nabla u_0-F|+|1-\eta_Q||\nabla\phi(F)|+|\nabla\eta_Q||\phi(F)|$. Combined with \eqref{locest:p1} and Poincar\'e's inequality we get
\begin{equation}\label{locest:p2}
  \fint_Q\langle \delta\bfa,\nabla\varphi\rangle\leq \|\delta\bfa\|_{\underline L^2(Q)}\lesssim r\|\nabla^2u_0\|_{\underline L^2(Q)}+r^{-\frac12}|F|.
\end{equation}
To estimate $J(F)$ we recall that by Lemma~\ref{L:phi} we have $J(F)=\nabla\cdot \sigma(F)$. Since $\sigma_{ijk}$ (resp. $\partial_j\partial_k\varphi$) is skew-symmetric (resp. symmetric) in $j,k$, an integration by parts yields
\begin{equation*}
  -\sum_{i,j,k=1}^d \int_Q\partial_k(\eta_Q\sigma_{ijk}(F))\partial_j\varphi^i=\sum_{i,j,k=1}^d \int_Q\eta_Q\sigma_{ijk}(F))\partial_j\partial_k\varphi^i=0.
\end{equation*}
Combined with the identity $J(F)=\nabla\cdot (\eta_Q\sigma+(1-\eta_Q)\sigma)$, Cauchy-Schwarz inequality, and \eqref{locest:p1}, we obtain
\begin{equation*}
  \fint_Q\langle J(F),\nabla\varphi\rangle\lesssim\left(\fint_Q|\nabla((1-\eta)\sigma)|^2\right)^\frac12\lesssim r^{-\frac12}|F|.
\end{equation*}
Combined with \eqref{locest:p2} the claimed estimate follows.

\step{2} Error estimate for the two-scale expansion.

Suppose that the intermediate scale satisfies $2\leq r<\frac{R}{8}$ and fix $\rho\in(2\frac{r}{R},\frac14)$ --- the relative thickness of a boundary layer. We claim 
\begin{eqnarray*}
    I:=\|\nabla u-\nabla(u_0+\psi)\|_{\underline L^2(Q_\frac{R}{2})}&\lesssim&  (r^{-\frac12}+(\frac{r}{\rho R})+\rho^\frac{\mu}{4+2\mu})\|\nabla u\|_{\underline L^2(Q_R)}.
\end{eqnarray*}
For the argument set $w:=u-(u_0+\psi)$ and note that $w\in H^1_0(Q_R)$. By monotonicity, and the equations for $u$ and $u_0$, 
\begin{eqnarray*}
  I^2=\fint_{Q_\frac{R}2}|\nabla w|^2&\lesssim&\fint_{Q_\frac{R}2}\langle\bfa (\nabla u)-\bfa (\nabla (u_0+\psi)),\nabla w\rangle=\fint_{Q_R}\langle\bfa_0 (\nabla u_0)-\bfa (\nabla (u_0+\psi)),\nabla w\rangle\\
  &=&(\frac{2r}{R})^d\sum_{Q\in\mathcal Q}\fint_{Q}\langle\bfa_0 (\nabla u_0)-\bfa (\nabla (u_0+\psi)),\nabla w\rangle.
\end{eqnarray*}
Appealing to Young's inequality and the definition of $\|\cdot\|_{\underline H^1(Q)^*}$ in Step~1, we obtain
\begin{equation*}
 I^2\lesssim (\frac{r}{R})^d\sum_{Q\in\mathcal Q}\|\bfa_0 (\nabla u_0)-\bfa (\nabla (u_0+\psi))\|_{\underline H^1(Q)^*}^2.
\end{equation*}
We estimate the contribution for interior cubes and cubes close to the boundary $\partial Q_\frac{R}2$ differently. Therefore, set $\mathcal Q_{\rho}:=\{\,Q\in\mathcal Q\,:\,Q\subset Q_{(1-\rho)\frac{R}2}\,\}$ and $S:=Q_\frac{R}2\setminus (\bigcup_{Q\in\mathcal Q_{\rho}}Q)$. For cubes $Q\in\mathcal Q_{\rho}$ we appeal to \eqref{EXC:eq:2:2} and \eqref{EXC:eq:1:1b},
\begin{eqnarray*}
  (\frac{r}{R})^d\sum_{Q\in\mathcal Q_\rho}\|\bfa_0 (\nabla u_0)-\bfa (\nabla (u_0+\psi))\|_{\underline H^1(Q)^*}^2&\lesssim&\fint_{Q_{(1-\rho)\frac{R}2}}\frac1r|\nabla u_0|^2+r^2|\nabla^2 u_0|^2\\
  &\lesssim&(\frac1r+(\frac{r}{\rho R})^2)\fint_{Q_R}|\nabla u|^2.
\end{eqnarray*}
For the remainder we appeal to the higher integrability of $\nabla u_0$ as follows: By the trivial estimate $\|\cdot\|_{\underline H^1(Q)^*}\leq \|\cdot\|_{\underline L^2(Q)}$, the Lipschitz continuity of $\bfa$ and $\bfa_0$, \eqref{EXC:eq:1:2a}, \eqref{EXC:eq:1:1c}, and $\frac{|S|}{|Q_\frac{R}2|}\lesssim \rho$, we get
\begin{eqnarray*}
  &&(\frac{r}{R})^d\sum_{Q\in\mathcal Q\setminus\mathcal Q_\rho}\|\bfa_0 (\nabla u_0)-\bfa (\nabla (u_0+\psi))\|_{\underline H^1(Q)^*}^2\\
  &\leq& (\frac{1}{R})^d\sum_{Q\in\mathcal Q\setminus\mathcal Q_\rho}\int_Q|\bfa_0(\nabla u_0)-\bfa (\nabla (u_0+\psi))|^2\lesssim \fint_{Q_\frac{R}2}{\bf 1}_S|\nabla u_0|^2\,\lesssim\, \rho^{\frac{\mu}{2+\mu}}\|\nabla u\|_{\underline L^2(Q_{2R})}^2,
\end{eqnarray*}
and thus the combination of the previous two estimates yields the claim.

\step{3} One-step improvement.

We claim that there exist $c_1=c_1(\beta,\gamma,d)\in[1,\infty)$, $\overline \kappa=\overline \kappa(\beta,\gamma,d)>0$ and an exponent $q=q(\beta,d)\in(0,\frac{1}{5})$ such that
\begin{equation}\label{EXC:stepdecay}
 \forall\tau\in [R^{-\frac{3}5},\tfrac14]:\qquad \operatorname{Exc}(\nabla u,Q_{\tau R})\,\leq\,c_1\tau^{\gamma}\Big(\tau^{-\frac{d}2-1}R^{-q}+\tau^{\frac12(1-\gamma)}\Big)  \operatorname{Exc}(\nabla u,Q_R)
\end{equation}
provided
\begin{equation}\label{exc:small}
 \operatorname{Exc}(\nabla u,Q_R)\leq \overline \kappa.
\end{equation}
We divide the proof into two steps: We first establish a weaker version of the claim in which $\operatorname{Exc}(\nabla u,Q_R)$ in \eqref{EXC:stepdecay} and \eqref{exc:small} is replaced by $\|\nabla u\|_{\underline L^2(Q_R)}$ and then refine the estimate by appealing to a tilting argument.

\substep{3.1} Preliminary one-step improvement.

We claim that there exist $c_1=c_1(\beta,\gamma,d)\in[1,\infty)$, $\overline \kappa=\overline \kappa(\beta,\gamma,d)>0$ and an exponent $q=q(\beta,d)\in(0,\frac{1}{5})$, such that
\begin{equation}\label{small:u0wrong}
 \|\nabla u\|_{\underline L^2(Q_R)}\leq \overline \kappa
\end{equation}
implies,
\begin{equation}\label{EXC:stepdecaywrong}
 \forall\tau\in [R^{-\frac{3}5},\tfrac14]:\qquad \operatorname{Exc}(\nabla u,Q_{\tau R})\,\leq\,c_1\tau^{\gamma}\Big(\tau^{-\frac{d}2-1}R^{-q}+\tau^{\frac12(1-\gamma)}\Big)  \|\nabla u\|_{\underline L^2(Q_R)}.
\end{equation}
Let $r$, $\rho$ and $w$ be as in Step~2. Then for any $r\leq s<\frac{R}{4}$ we have with $F:=(\nabla u_0)_{Q_s}$ by minimality and the triangle inequality,
\begin{eqnarray*}
  \operatorname{Exc}(\nabla u,Q_s)&\leq&\|\nabla u - (F+\nabla \phi(F))\|_{\underline L^2(Q_s)}\\&\leq& \|\nabla w\|_{\underline L^2(Q_s)}+\|\nabla u_0-F\|_{\underline L^2(Q_s)}+\|\nabla \psi - \nabla \phi(F)\|_{\underline L^2(Q_s)}.
\end{eqnarray*}
We estimate
\begin{itemize}
\item the first term on the right-hand side by appealing to the trivial estimate $\fint_{Q_s}|\cdot|^2\leq (\frac Rs)^d\fint_{Q_R}|\cdot|^2$ and Step~2, 
\item the second term by exploiting the smallness assumption \eqref{small:u0wrong} and Lemma~\ref{L:regconstant}: Recall that $\bfa_0\in\mathcal A_{\beta,\omega'}$ with a modulus of continuity $\omega$ depending only on $\beta$ and $d$, cf.~Lemma~\ref{L:phi}. Hence, by Lemma~\ref{L:regconstant} (applied with $\frac{\gamma+1}2$) and the estimate \eqref{EXC:eq:1:1a}, we find $\overline \kappa=\overline \kappa(\beta,\gamma,d)>0$ and $c=c(\beta,d,\gamma)<\infty$ such that \eqref{small:u0wrong} yields 
\begin{equation}\label{reg:u0}
\forall s'\in(0,\tfrac{R}{2}]:\qquad \|\nabla u_0-(\nabla u_0)_{Q_{s'}}\|_{\underline L^2(Q_{s'})}\leq c (\frac{s}{R})^\frac{\gamma+1}2\|\nabla u\|_{\underline L^2(Q_R)},
\end{equation}
\item the third term by appealing to the estimate
\begin{equation}\label{est:term3}
 \|\nabla\phi(F)-\nabla\psi\|_{\underline L^2(Q_s)}\leq \|\nabla u_0 - F\|_{\underline L^2(Q_{2s})}+r^{-\frac12}\|\nabla u_0\|_{\underline L^2(Q_{2s})},
\end{equation}
which we establish below. In combination with \eqref{reg:u0}, we obtain
\begin{equation*}
 \|\nabla\phi(F)-\nabla\psi\|_{\underline L^2(Q_s)}\leq c(\beta,\gamma,d) ((\frac{s}{R})^{\frac{\gamma+1}2}+r^{-\frac12}(\frac{R}{s})^\frac{d}2)\|\nabla u\|_{\underline L^2(Q_R)}.
\end{equation*}
\end{itemize}
In summary, we conclude that
\begin{equation*}
  \operatorname{Exc}(\nabla u,Q_s)\,\leq\,c(\beta,\gamma,d)\Big((\frac{R}{s})^\frac{d}{2}\big(r^{-\frac12}+\tfrac{r}{\rho R}+\rho^\frac{\mu}{4+2\mu}\big)+(\frac{s}{R})^\frac{\gamma+1}2\Big)\|\nabla u\|_{\underline L^2(Q_R)}.
\end{equation*}
Optimization in $\rho$ and $r$ suggests that $\rho:=(\frac{r}{R})^\frac{4+2\mu}{4+3\mu}$ and $r:=R^{\frac{2\mu}{4+5\mu}}$, which is an admissible choice, iff 
$$
2R^{-\frac{4+3\mu}{4+5\mu}}=2\tfrac{r}{R}<\rho=R^{-\frac{4+2\mu}{4+5\mu}}<\tfrac14,\qquad 2\leq R^{\frac{2\mu}{4+5\mu}}=r\leq s\leq \tfrac{R}{8}.
$$
Setting $s=\tau R$, we obtain with $q=q(\beta,d):=\frac{\mu}{4+5\mu}$
\begin{equation*}
  \forall \tau\in[R^{-\frac{4+3\mu}{4+5\mu}},\tfrac18]:\qquad \operatorname{Exc}(\nabla u,Q_{\tau R})\,\leq\,c\tau^{\gamma}\Big((\tau^{-\frac{d}2-\gamma} R^{-q}+\tau^{\frac{1-\gamma}2}\Big)\|\nabla u\|_{\underline L^2(Q_R)},
\end{equation*}
and \eqref{EXC:stepdecaywrong} follows since $R^{-\frac{4+3\mu}{4+5\mu}}\leq R^{-\frac{3}5}$ for $\mu>0$ and $\gamma<1$.

\bigskip

Finally, we provide an argument for \eqref{est:term3}. We have
\begin{align*}
 \|\nabla \psi - \nabla \phi(F)\|_{\underline L^2(Q_s)}^2\leq 2(\tfrac{r}{s})^d\sum_{Q\in\mathcal Q \atop Q\cap Q_s \neq \emptyset}\left(\|\nabla \psi - \chi\|_{\underline L^2(Q)}^2+\|\chi-\nabla \phi(F)\|_{\underline L^2(Q)}^2\right).
\end{align*}
Using \eqref{EXC:eq:1:1b} and the Lipschitz-continuity of $F\mapsto \phi(F)$ (see Lemma~\ref{L:phi}) together with $2\leq r$, we obtain
\begin{align*}
 \|\nabla \psi - \nabla \phi(F)\|_{\underline L^2(Q_s)}^2 \lesssim&(\tfrac{r}{s})^d\sum_{Q\in\mathcal Q \atop Q\cap Q_s \neq \emptyset}\left(r^{-1} |(\nabla u)_Q|^2+|(\nabla u_0)_Q - F|^2\right)\\
 \leq&s^{-d}\sum_{Q\in\mathcal Q \atop Q\cap Q_s \neq \emptyset}\left(r^{-1} \int_Q|\nabla u|^2+\int_Q|\nabla u_0 - F|^2\right),
\end{align*}
and the estimate \eqref{est:term3} follows since the choice $2\leq r\leq s$ implies $\cup_{Q \in \mathcal Q \atop Q\cap Q_s} Q\subset Q_{2s}$.

\substep{3.2} Conclusion

For given $\widetilde F\in\R^{d\times d}$, we set
\begin{align*}
 \widetilde u(x)&:=u(x)-\widetilde F -\phi(x,\widetilde F)\\
 \widetilde\bfa(\cdot,F)&:=\bfa(\cdot,F+\widetilde F+\nabla\phi(\cdot,\widetilde F))-\bfa (\widetilde F +\nabla \phi(\cdot,\widetilde F))\quad\mbox{for all $F\in\R^{d\times d}$}. 
\end{align*}
Clearly, $\widetilde\bfa$ is a periodic coefficient field of class $\mathcal A_\beta$ and the homogenized coefficients $\widetilde \bfa_0$ and the corresponding correctors $\widetilde \phi$ are given by
\begin{equation}\label{tildephi}
 \widetilde \bfa_0(F):=\bfa_0(F+\widetilde F)-\bfa_0(\widetilde F)\qquad\mbox{and}\qquad \widetilde \phi(F)=\phi(F+\widetilde F)-\phi(\widetilde F).
\end{equation}
By construction, we have $\divv \widetilde \bfa(\nabla \widetilde u)=0$ in $\mathscr D'(Q_R)$. Hence, appealing to Substep~3.1 we obtain that
\begin{equation*}
 \|\nabla \widetilde u\|_{\underline L^2(Q_R)}=\|\nabla u- (\widetilde F + \nabla \phi(\widetilde F))\|_{\underline L^2(Q_R)}\leq \kappa, 
\end{equation*}
implies,
\begin{align*}
  \forall \tau\in[R^{-\frac{3}{5}},\tfrac14]:\qquad \widetilde{\operatorname{Exc}}(\nabla \widetilde u,Q_{\tau R})&:=\inf_{F\in\R^{d\times d}}\|\nabla \widetilde u-(F+\nabla\widetilde \phi(F))\|_{\underline L^2(\tau R)}\\
  &\leq\,c\tau^{\gamma}\Big((\tau^{-\frac{d}2-\gamma} R^{-q}+\tau^{\frac{1-\gamma}2}\Big)\|\nabla \widetilde u\|_{\underline L^2(Q_R)}.
\end{align*}
The definition of $\widetilde u$ and \eqref{tildephi} yield for every $F\in\R^{d\times d}$
$$
\|\nabla \widetilde u -(F+\nabla\tilde \phi(F))\|_{\underline L^2(\Omega)}=\|\nabla u -(F+\widetilde F +\nabla \phi(F+\widetilde F))\|_{\underline L^2(\Omega)},
$$
and thus $\widetilde{\operatorname{Exc}}(\nabla \widetilde u,Q_{\tau R})=\operatorname{Exc}(\nabla u,Q_{\tau R})$. The claim follows by choosing $\widetilde F$ such that $\|\nabla \widetilde u\|_{\underline L^2(Q_R)}=\operatorname{Exc}(\nabla u,Q_R)$.

\step{4} Iteration.

For given $\gamma\in(0,1)$, let $\overline \kappa=\overline \kappa(\beta,\gamma,d)>0$, $c_1=c_1(\beta,\gamma,d)\in[1,\infty)$ and $q=q(\beta,d)\in(0,\frac15)$ be given as in Step~3, and set 
\begin{equation}\label{smalltau}
  \tau:=\min\{(2c_1)^{-\frac2{1-\gamma}},\tfrac14\}.
\end{equation}
Let $R\geq4$ be such that
\begin{equation}\label{bigR}
 R\geq c_2:= \left(2c_1 \tau^{-1-\frac{d}2}\right)^\frac1q,
\end{equation}
and note that $c_2=c_2(\beta,\gamma,d)>0$. Since $q<\frac{3}{5}$, we have that $\tau$ given in \eqref{smalltau} is a valid choice in \eqref{EXC:stepdecay} if $R$ satisfies \eqref{bigR} and we obtain
\begin{equation*}
 \operatorname{Exc}(\nabla u,Q_{\tau R})\,\leq\,\tau^{\gamma} \operatorname{Exc}(\nabla u,Q_R)\qquad\mbox{provided \eqref{exc:small}}.
\end{equation*}
Clearly, the above estimate can be iterated and we obtain that \eqref{exc:small} implies
\begin{equation}\label{iteration}
  \operatorname{Exc}(\nabla u,Q_{\tau^k R})\,\leq\,\tau^{\gamma k} \operatorname{Exc}(\nabla u,Q_R)\qquad\mbox{for all $k\in\N$ satisfying}\quad \tau^{k-1}R\geq c_2.
\end{equation}
Finally, we observe that \eqref{iteration} implies \eqref{est:excessdecay} with some $c=c(\beta,\gamma,d)\in[1,\infty)$. Indeed  
\begin{itemize}
\item If $c_2\leq \tau^kR<r\leq \tau^{k-1}R$ for some $k\in\N$, then, 
$$
 \operatorname{Exc}(\nabla u,Q_r)\leq \tau^{-\frac{d}2}\operatorname{Exc}(\nabla u,Q_{\tau^{k-1} R})\leq \tau^{-\frac{d}2-\gamma}(\frac{r}R)^{\gamma}\operatorname{Exc}(\nabla u,Q_R).
$$
\item If $1\leq r\leq \tau^k R<c_2\leq \tau^{k-1}R$ for some $k\in\N$, then,
\begin{align*}
 \operatorname{Exc}(\nabla u,Q_r)\leq& (c_2 \tau^{-1})^\frac{d}{2}\operatorname{Exc}(\nabla u,Q_{\tau^{k-1} R})\leq (c_2 \tau^{-1})^{\frac{d}2+\gamma}(\frac{r}R)^{\gamma}\operatorname{Exc}(\nabla u,Q_R).
\end{align*}
\item If $1\leq r\leq R<c_2$, then, $\operatorname{Exc}(\nabla u,Q_r)\leq c_2^{\frac{d}2+\gamma}(\frac{r}R)^{\gamma}\operatorname{Exc}(\nabla u,Q_R)$.

\end{itemize}

\step 5 The case $d=2$.

For $d=2$, we can argue without appealing to any smallness of $\operatorname{Exc}(\nabla u,Q_R)$. Indeed, let $\gamma_{2d}>0$ be given as in Lemma~\ref{L:regconstant}. With the obvious changes in Step~3, we obtain instead of \eqref{EXC:stepdecay} the estimate 
\begin{equation*}
 \forall\tau\in [R^{-\frac{3}5},\tfrac14]:\qquad \operatorname{Exc}(\nabla u,Q_{\tau R})\,\leq\,c_1\tau^{\frac12\gamma_{2d}}\Big(\tau^{-\frac{d}2-1}R^{-q}+\tau^{\frac12\gamma_{2d}}\Big)  \operatorname{Exc}(\nabla u,Q_R),
\end{equation*}
where $c_1=c_1(\beta)\in[1,\infty)$ and the claimed decay follows as in Step~4.

\end{proof}

\begin{corollary}[Lipschitz estimate at large scales]\label{P:largeholer1}
Fix $d\geq 2$ and $\beta,s\in(0,1]$. Let $\bfa$ be a periodic coefficient field of class $\mathcal A_\beta$. There exists $\overline \kappa_0=\overline \kappa_0(\beta,d,s)>0$ and $c=c(\beta,d,s)\in[1,\infty)$ such that if $u\in H^1(Q)$, with $Q = Q_R(x_0)$ for some $x_0\in \R^d$ and $R>0$, satisfy
\begin{equation*}
\max\left\{ \sup_{r\in(0,1]}r^{-s}\|\divv \bfa(\nabla u)\|_{\underline H^{-1}(rQ)}, \operatorname{Exc}(\nabla u,Q)\right\}\leq\begin{cases}
                                      \infty&\mbox{if $d=2$,}\\
                                      \overline \kappa_0&\mbox{if $d\geq3$.}
                                     \end{cases}
\end{equation*}
Then, 
\begin{equation*}
  \forall 1\leq r\leq R:\qquad \|\nabla u\|_{\underline L^2(Q_r(x_0))}\leq c(\|\nabla u\|_{\underline L^2(Q)}+ \sup_{r\in(0,1]}r^{-s}\|\divv \bfa(\nabla u)\|_{\underline H^{-1}(r Q)}).
\end{equation*}
\end{corollary}

\begin{proof}

Throughout the proof, we write $\lesssim$ whenever it holds $\leq$ up to a multiplicative constant which depends only on $\beta$ and $d$. Without loss of generality, we suppose $Q=Q_R$. 

\smallskip

\step{1} One-step improvement.

Fix $\gamma=\gamma(\beta,d)>0$ as $\gamma=\frac12 \gamma_{2d}$ for $d=2$ and $\gamma=\frac12$ for $d\geq3$. We claim that there exists $\overline \kappa=\overline \kappa(\beta,d)>0$ and $c_1=c_1(\beta,d)\in[1,\infty)$ such that if $u$ satisfies
\begin{equation}\label{est:ecessmall}
 \operatorname{Exc}(\nabla u,Q_{R'})\leq \begin{cases}
                      \infty&\mbox{if $d=2$,}\\
                      \overline \kappa&\mbox{if $d\geq3$,}
                     \end{cases}
\end{equation}
for some $R'\in[1,R]$, then for all $\tau\in[\frac1{R'},1]$,
\begin{align}
 \operatorname{Exc}(\nabla u,Q_{\tau R'})\leq& c_1 \tau^{\gamma} \operatorname{Exc} (\nabla u,Q_{R'})+ c_1\tau^{-\frac{d}{2}}\|\divv \bfa(\nabla u)\|_{\underline H^{-1}(Q_{R'})}.\label{est:exess1}
\end{align}
Let $v\in H^1(Q_{R'})$ be the unique solution of
\begin{equation}\label{eq:vzerorhs}
 \divv \bfa (\nabla v)=0\quad\mbox{in $\mathscr D'(Q_{R'})$}\quad\mbox{with}\quad u-v\in H_0^1(Q_{R'}).
\end{equation}
We first observe that
\begin{align}\label{est:vlip}
 \beta\|\nabla u-\nabla v\|_{\underline L^2(Q_{R'})}\leq& \|\bfa (\nabla u)\|_{\underline H^{-1}(Q_{R'})},\qquad \beta^2\operatorname{Exc}(\nabla v,Q_{R'})\leq \operatorname{Exc}(\nabla u,Q_{R'}).
\end{align}
Indeed, the first inequality is a straightforward consequence of the monotonicity of $\bfa$, cf.~\eqref{ass:monotonicity}, and the definition of $\underline H^{-1}$. For the second inequality, we use that for every $F\in\R^{d\times d}$ it holds
\begin{align*}
 \beta\|\nabla v-(F+\nabla \phi(F))\|_{\underline L^2(Q_{R'})}^2\leq& \fint_{Q_{R'}}\langle \bfa (\nabla v)-\bfa(F+\nabla\phi(F)),\nabla v-(F+\nabla \phi(F))\rangle\\
 =&\fint_{Q_{R'}}\langle \bfa (\nabla v)-\bfa(F+\nabla\phi(F)),\nabla u-(F+\nabla \phi(F))\rangle,
\end{align*}
and thus by the Lipschitz-continuity of $\bfa$, cf.~\eqref{ass:lip}, and the definition of $\operatorname{Exc}(\nabla v,Q_{R'})$, cf.~\eqref{def:excess},
\begin{align*}
 \beta^2\operatorname{Exc}(\nabla v,Q_{R'})\leq  \beta^2\|\nabla v -(F+\nabla \phi(F))\|_{\underline L^2(Q_{R'})}\leq \|\nabla u-(F+\nabla \phi(F))\|_{\underline L^2(Q_{R'})}.
\end{align*}
Minimizing the right-hand side in $F\in\R^{d\times d}$, we obtain the second inequality in \eqref{est:vlip}.

\bigskip

Combining \eqref{eq:vzerorhs}, \eqref{est:vlip} and Proposition~\ref{P:largescale:holder}, we obtain that if \eqref{est:ecessmall} holds with $\overline \kappa=\beta^2\overline \kappa_0(\beta,\frac12,d)$ and $d\geq3$, where $\overline \kappa_0$ as in Proposition~\ref{P:largescale:holder}, then
\begin{equation*}
\operatorname{Exc}(\nabla v,Q_{\tau R'})\lesssim \tau^\gamma \operatorname{Exc}(\nabla v,Q_{R'})\lesssim \tau^\gamma \operatorname{Exc}(\nabla u,Q_{R'}),
\end{equation*}
and thus
\begin{align*}
 \operatorname{Exc}(\nabla u,Q_{\tau R'})\leq& \operatorname{Exc}(\nabla v,Q_{\tau R'})+\|\nabla v -\nabla u\|_{\underline L^2(Q_{\tau R'})}\\
 \lesssim& \tau^\gamma \operatorname{Exc}(\nabla u,Q_{R'})+\tau^{-\frac{d}2}\|\nabla v -\nabla u\|_{\underline L^2(Q_{R'})}.
\end{align*}
This proves \eqref{est:exess1}.

\step{2} Iteration

As in \eqref{def:H}, we set
\begin{equation*}
 H:=\sup_{r\in(0,1]}r^{-s}\|\divv \bfa(\nabla u)\|_{\underline H^{-1}(rQ)}.
\end{equation*}
Choose $\tau=\tau(\beta,d)\in[\frac1{R'},\frac14)$ such that
\begin{equation}\label{smalltau2}
 c_1\tau^\gamma\leq \tfrac14,
\end{equation}
where $c_1=c_1(\beta,d)\in[1,\infty)$ is given in Step~1. For $R\geq1$ satisfying
\begin{equation}\label{bigR2}
 R\geq (4c_1 \tau^{-\frac{d}2})^{\frac1s}\tau^{-1}=:c_2(\beta,d,s)\quad\mbox{we set}\quad R':=\frac{R}{c_2 \tau}\in[\tau^{-1},R].
\end{equation}
We claim that the following is true: Suppose that
\begin{equation}\label{est:smalliter}
\max\{\operatorname{Exc}(\nabla u,Q_{R'}), 2H\}\leq\begin{cases}
       \infty&\mbox{if $d=2$,}\\ \overline \kappa&\mbox{if $d\geq3$,}
      \end{cases}
\end{equation}
where $\overline \kappa=\overline \kappa(\beta,d)>0$ is given as in Step~1. Then, for all $k\in\N$, 
\begin{align}
 R'\geq \tau^{-k}\quad\mbox{implies}\quad\operatorname{Exc}(\nabla u,Q_{\tau^k R'})\leq& 4^{-k}\operatorname{Exc}(\nabla u,Q_{R'})+ 2^{-ks} H.\label{eq:homlip:it1}
\end{align}
The argument is via induction. The case $k=1$ is contained in Step~1. Indeed, the smallness assumption \eqref{est:smalliter} ensures that \eqref{est:exess1} is valid and the choice of $\tau$ and $R'$ in \eqref{smalltau2} and \eqref{bigR2} yield
\begin{align*}
 \operatorname{Exc}(\nabla u,Q_{\tau R'})\leq& \tfrac14 \operatorname{Exc} (\nabla u,Q_{R'})+ c_1\tau^{-\frac{d}{2}}(\tfrac{R'}{R})^sH\leq \tfrac14 \operatorname{Exc} (\nabla u,Q_{R'})+ \tfrac14H.
\end{align*}
Thus, \eqref{eq:homlip:it1} follows. We next turn to the induction step. Suppose \eqref{eq:homlip:it1} holds for a $k\in \N$. Then, for $d\geq3$ we have,
\begin{equation*}
 \operatorname{Exc}(\nabla u,Q_{\tau^k R'}) \leq 4^{-k} \operatorname{Exc}(\nabla u,Q_{R'})+ 2^{-ks}  H\leq (4^{-k}+2^{-(1+ks)})\overline\kappa\leq \overline \kappa.
\end{equation*}
Assume that $R'\geq \tau^{-(k+1)}$. By Step~1, \eqref{smalltau2}, \eqref{bigR2} and the induction hypothesis, we have
\begin{eqnarray*}
\operatorname{Exc}(\nabla u,Q_{\tau^{k+1}R'})&\leq& c_1 \tau^{\gamma} \operatorname{Exc} (\nabla u,Q_{\tau^k R'})+ c_1\tau^{-\frac{d}{2}}\|\divv \bfa(\nabla u)\|_{\underline H^{-1}(Q_{\tau^kR'})}\\
&\leq& 4^{-1} \operatorname{Exc} (\nabla u,Q_{\tau^k R'})+8^{-1}\tau^{ks}H\\
&\stackrel{\eqref{eq:homlip:it1}}{\leq}& 4^{-(k+1)}\operatorname{Exc} (\nabla u,Q_{R'})+(4^{-1}2^{-ks}+8^{-1}4^{-ks}) H\\
&\leq& 4^{-(k+1)}\|\nabla u\|_{\underline L^2(B_0)}+2^{-(k+1)s} H.
\end{eqnarray*}

\step{3} Conclusion

Let $\overline \kappa$, $\tau$ and $R'$ be as in Step~2 and suppose that \eqref{est:smalliter} is valid. We show that there exist $c=c(\beta,d,s)\in[1,\infty)$ such that for all $k\in\N$ with $R'\geq \tau^{-k}$ 
\begin{equation}
 \|\nabla u\|_{\underline L^2(\tau^k R')}\leq c( \|\nabla u\|_{\underline L^2(R')}+H).\label{eq:homlip:it3}
\end{equation}
Evidently, this estimate implies the claim of Corollary~\ref{P:largeholer1} for $\overline \kappa_0=(c_2\tau)^{-\frac{d}2}\overline \kappa$. 
We prove \eqref{eq:homlip:it3}. To that end, we first derive some auxiliary estimate for the \textit{approximate gradient} $F_r$, $r\in[1,R]$, defined as the unique matrix satisfying $\operatorname{Exc}(\nabla u,Q_r)=\|\nabla u - (F_r + \nabla \phi(F_r))\|_{\underline L^2(Q_{r})}$. 
We claim that 
\begin{equation}\label{eq:st01010b}
|F_{r}|\lesssim \|\nabla u\|_{\underline L^2(Q_{r})}\lesssim |F_r|+\operatorname{Exc}(\nabla u,Q_r).
\end{equation}
Indeed, if we denote by $\lfloor r\rfloor$ the largest integer not greater than $r$, then by periodicity we have
\begin{eqnarray*}
  |F_r|&\leq& \left(\fint_{Q_{\lfloor r\rfloor}}|F_r+\nabla\phi(F_r)|^2\right)^\frac12\lesssim \|F_r+\nabla\phi(F_r)\|_{\underline L^2(Q_r)}\leq \operatorname{Exc}(\nabla u,Q_r)+\|\nabla u\|_{\underline L^2(Q_r)}\\
  &\leq&2\|\nabla u\|_{\underline L^2(Q_r)}.
\end{eqnarray*}
Moreover, 
\begin{eqnarray*}
  \|\nabla u\|_{\underline L^2(Q_{r})}\leq \|F_r+\nabla\phi(F_r)\|_{\underline L^2(Q_{r})}+\operatorname{Exc}(\nabla u,Q_r)\lesssim |F_r|+\operatorname{Exc}(\nabla u,Q_r),
\end{eqnarray*}
where the last estimate is due to Lemma~\ref{L:phi}. Next, let $1\leq r_1\leq r_2\leq R$ and note that
\begin{eqnarray*}
 |F_{r_1}-F_{r_2}|&=&|\fint_{Q_{r_1}} F_{r_1}+\nabla \phi(F_{r_1})-(F_{r_2}+\nabla \phi(F_{r_2}))|\\
 &\leq&\left(\fint_{Q_{r_1}}|\nabla u-(F_{r_1}+\nabla \phi(F_{r_1}))|^2\right)^\frac12+\left(\fint_{Q_{r_1}}|\nabla u-(F_{r_2}+\nabla \phi(F_{r_2}))|^2\right)^\frac12\\
 &\leq&\operatorname{Exc}(\nabla u,Q_{r_1})+(\frac{r_2}{r_1})^\frac{d}{2}\operatorname{Exc}(\nabla u,Q_{r_2}).
\end{eqnarray*}
Combined with the trivial estimate $\operatorname{Exc}(\nabla u,Q_{r_1})\leq (\frac{r_2}{r_1})^\frac{d}2\operatorname{Exc}(\nabla u,Q_{r_2})$, we get
\begin{eqnarray*}
 |F_{r_1}-F_{r_2}|&\leq&2(\frac{r_2}{r_1})^\frac{d}{2}\operatorname{Exc}(\nabla u,Q_{r_2}).
\end{eqnarray*}
Hence, for any $k\in\N$ with $R'\geq\tau ^{-k}$ we obtain by a telescopic-sum argument and \eqref{eq:homlip:it1},
\begin{eqnarray*}
  |F_{\tau^k R'}|&\leq& |F_{R'}|+\sum_{i=1}^{k}|F_{\tau^i R'}-F_{\tau^{i-1} R'}| \leq |F_{R'}|+2\tau^{-\frac{d}{2}}\sum_{i=0}^\infty \big(4^{-i}\operatorname{Exc}(\nabla u,Q_{R'})+ 2^{-is}H\big)\notag\\
  &\lesssim& |F_{R'}|+\|\nabla u\|_{\underline L^2(Q_{R'})}+H.
\end{eqnarray*}
Combined with \eqref{eq:st01010b} and \eqref{eq:homlip:it1}, we get
\begin{eqnarray*}
  \|\nabla u\|_{\underline L^2(\tau^k R')}&\lesssim& |F_{\tau^kR'}|+\operatorname{Exc}(\nabla u,Q_{\tau^k R'})\lesssim \|\nabla u\|_{\underline L^2(Q_{R'})}+\operatorname{Exc}(\nabla u,Q_{R'})+H\\
  &\lesssim&\|\nabla u\|_{\underline L^2(Q_{R'})}+H,
\end{eqnarray*}
and thus \eqref{eq:homlip:it3} follows.

\end{proof}

\begin{proof}[Proof of Theorem~\ref{T:3}]
 
We show that there exists $\kappa>0$ and $c\in[1,\infty)$ with the claimed dependences such that \eqref{eq:uepsf} and \eqref{eq:uepssmall} for some $B=B_R(0)$ with $R>0$ imply 
\begin{equation*}
 \sup_{r\in(0,\e)}\|\nabla u\|_{\underline L^2(B_r(0))}\leq c\left(\|\nabla u\|_{\underline L^2(B)}+R\|f\|_{\underline L^q(B)}\right).
\end{equation*}
Clearly, the claimed estimate then follows by translation arguments and the Lebesgue differentiation theorem.

Set $u_\e:=\frac1\e u(\e \cdot)$ and $f_\e=\e f(\e \cdot)$ and note that
\begin{equation*}
 \divv \bfa (x,\nabla u_\e)=f_\e\qquad\mbox{in $\mathscr D'(B_{\frac{R}\e}(0))$.}
\end{equation*}

\step{1} Application of large-scale regularity.

We claim that there exists $c=c(\beta,d,s,q)\in[1,\infty)$ and $\overline \kappa=\overline \kappa(\beta,d,s,q)>0$ such that \eqref{eq:uepsf} and \eqref{eq:uepssmall} imply 
\begin{equation*}
 \|\nabla u\|_{\underline L^2(B_\e(0))}=\|\nabla u_\e\|_{\underline L^2(B_1(0))}\leq c\left(\|\nabla u\|_{\underline L^2(B)}+R\|f\|_{L^q(B)}\right).
\end{equation*}
In view of Corollary~\ref{P:largeholer1}, it suffices to show that
\begin{equation}\label{est:divaf}
 \sup_{r\in(0,1]}r^{-(1-\frac{d}{q})}\|\divv \bfa(\nabla u_\e)\|_{H^{-1}(\frac{r}\e B)}\leq R\|f\|_{\underline L^q(B)}.
\end{equation}
Let $v_\e\in H_0^1(\frac1\e B)$ be the unique solution to $-\Delta v_\e=f_\e$ in $H_0^1(\frac1\e B)$. By maximal regularity, we find $c=c(d)<\infty$ such that
\begin{equation*}
 \|\nabla^2 v\|_{\underline L^q(\frac1\e B)}\leq c\|f_\e\|_{\underline L^q(\frac1\e B)}=c\e \|f\|_{\underline L^q(B)}.
\end{equation*}
Hence,
\begin{align*}
 \|\divv \bfa(\nabla u_\e)\|_{H^{-1}(\frac{r}\e B)}&=\|\divv \nabla v\|_{H^{-1}(\frac{r}\e B)}=\|\nabla v - (\nabla v)_{\frac{r}\e B}\|_{\underline L^2(\frac{r}\e B)}\\
 &\lesssim \frac{rR}\e\|\nabla^2 v_\e\|_{\underline L^2(r \frac1\e B)}\leq r^{1-\frac{d}{q}}\frac{R}{\e}\|\nabla^2 v\|_{\underline L^q(\frac1\e B)}\lesssim r^{1-\frac{d}{q}} R\|f\|_{\underline L^q(B)}
\end{align*}
which implies \eqref{est:divaf}.

\step{2} Application of small scale regularity. 
 
Combining Proposition~\ref{P:niren} and Step~1, we find $\tilde \kappa=\tilde \kappa(\beta,d,s,q,E)>0$ and $\tilde c=\tilde c(\beta,d,s,q,E)\in[1,\infty)$ such that \eqref{eq:uepsf} and \eqref{eq:uepssmall} imply
\begin{equation*}
 \sup_{r\in(0,\e)}\|\nabla u\|_{\underline L^2(B_r(0))}=\sup_{r\in(0,1)}\|\nabla u_\e\|_{\underline L^2(B_r(0))}\leq \tilde c\left(\|\nabla u\|_{\underline L^2(B)}+R\|f\|_{\underline L^q(B)}\right),
\end{equation*}
which proves the claim.
 
\end{proof}

\section{Acknowledgments}
The authors acknowledge funding by the Deutsche Forschungsgemeinschaft (DFG, German Research Foundation) -- project number 405009441, and in the context of TU Dresden's Institutional Strategy \textit{``The Synergetic University''}.

\appendix

\section{Appendix}\label{appendix}

\subsection{Proof of Lemma~\ref{L:estlin1a}}\label{sec:proofkinderlehrer}

\begin{proof}[Proof of Lemma~\ref{L:estlin1a}]
 
We present the proof only in the case $R=1$, the statement for arbitrary $R>0$ follows by scaling.
 
\step{1} For all $k\in\N$ there exists $c=c(\beta,k)<\infty$ such that
\begin{equation}\label{est:l2goode}
 \sum_{\ell=0}^k\|{\nabla'}^\ell \nabla u\|_{L^2(B_\frac{1}{2})}\leq c\|\nabla u\|_{L^2(B_1)}.
\end{equation}
Indeed, this is routine and follows by differentiating the equation via difference quotients. 

\step{2} As in Proposition~\ref{L:estmorrey}, we set $J_d(u):=(\mathbb L\nabla u)e_d$. Inequality \eqref{est:l2goode} yields corresponding estimates for the derivatives of $J_d(u)$ in $x'$-direction: For every $k\in\N$ there exists $c=c(\beta,d,k)<\infty$ such that
\begin{equation*}
 \sum_{\ell=0}^{k}\|{\nabla'}^\ell J_d(u)\|_{L^2(B_{\frac{1}{2}})}\leq c \|\nabla u\|_{L^2(B_1)}.
\end{equation*}
In order to obtain suitable estimates for $\partial_d J_d(u)$, we use equation \eqref{eq:linpde1d} in the form of $\partial_d J_d(u)=-\sum_{j=1}^{d-1} \partial_j(\mathbb L\nabla u)e_j$. Hence, we find for every $k\in\N$ a constant $c=c(\beta,d,k)<\infty$ such that
\begin{align*}
 \sum_{\ell=0}^k\|{\nabla'}^\ell \partial_d J_d(u)\|_{L^2(B_{\frac{1}{2}})}\leq c \|\nabla u\|_{L^2(B_1)}.
\end{align*}
 
\step{3} We claim that there exists $c=c(\beta,d)<\infty$ such that
\begin{equation*}
 \|\nabla u\|_{L^\infty(B_\frac12)}\leq c\|\nabla u\|_{L^2(B_1)}.
\end{equation*}
The following anisotropic Sobolev inequality can be found in \cite[Lemma~2.2]{LN03}: Suppose that ${\nabla'}^\ell f\in L^2(B_1)$ and ${\nabla'}^\ell \partial_d f\in L^2(B_1)$ for all $\ell=0,\dots,K$ with $\frac{d-1}2<K$. Then $f\in C^0(B_1)$ and
\begin{equation}\label{ineq:aniso}
 \|f\|_{L^\infty(B_1)}\leq C(d)\sum_{\ell=0}^{K}(\|{\nabla'}^\ell\partial_d f\|_{L^2(B_1)}+\|{\nabla'}^\ell f\|_{L^2(B_1)}).
\end{equation}
By Step~1 and 2, we can apply \eqref{ineq:aniso} to ${\nabla'}u$ and $J_d(u)$ and obtain that there exists $c=c(\beta,d)<\infty$ such that
\begin{align*}
 \|{\nabla'}u\|_{L^\infty(B_{\frac{1}{2}})}+\|J_d(u)\|_{L^\infty(B_\frac12)}\leq&c\|\nabla u\|_{L^2(B_1)}.
\end{align*}
Finally, the $L^\infty$ estimate for $\partial_d u$ follows from Lemma~\ref{l:edmonoton} (with $\bfa(F):=\mathbb LF$). 
 
\end{proof}

\subsection{Proof of Lemma~\ref{L:phi}}\label{sec:rega}

As already mentioned, Lemma~\ref{L:phi} is completely standard and well-known, see e.g.,\ \cite{CPZ05}. We only provide an argument for the continuity of $D\bfa$ claimed in the second bullet point. 

\begin{proof}
Throughout the proof we write $\lesssim$ if $\leq$ holds up to a multiplicative constant depending only on $\beta$ and $d$.

Let us first recall that $\bfa_0\in C^1(\R^{d\times d},\R^{d\times d})$ and for every $F,G\in\R^{d\times d}$ it holds 
\begin{equation*}
 D\bfa_0(F)[G]:=\int_{Q_1}\mathbb L_F(G+\nabla \psi_G(x,F))\,dx
\end{equation*}
where $\mathbb L_F\in L^\infty(\R^d,\R^{d^4})$ denotes the fourth order tensor satisfying
\begin{equation*}
 \mathbb L_F(x)G=D\bfa(x,F+\nabla \phi(x,F))[G]\qquad\mbox{for every $G\in\R^{d\times d}$ and a.e.~$x\in\R^d$,}
\end{equation*}
and $\psi_G(F)\in H_\per^1(Q_1)$ is uniquely given by
\begin{equation}\label{eq:psig}
 \divv \mathbb L_F(G+\nabla \psi_G(F))=0\quad\mbox{in $\mathscr D'(\R^d)$},\qquad (\psi_G(F))_{Q_1}=0.
\end{equation} 
The prove of this result can easily deduced from \cite[Theorem 5.4]{Mueller93} see also \cite{NS17} (in particular proof of Lemma~3). 

\step 1 We claim that there exists $c=c(\beta,d),q=q(\beta,d)\in[1,\infty)$ such that for all $F_1,F_2,G\in\R^{d\times d}$ with $|G|=1$ it holds
\begin{equation}\label{est:dpsig}
 \|\nabla \psi_G(F_1)-\nabla \psi_G(F_2)\|_{L^2(Q_1)}\leq c\omega(c|F_1-F_2|)^\frac1q.
\end{equation}  
\substep{1.1} We claim 
\begin{equation}\label{est:difflf}
 \|\mathbb L_{F_2} - \mathbb L_{F_2}\|_{L^1(Q_1)}\lesssim \omega((1+c)|F_1-F_2|),
\end{equation}
where $c=c(\beta)<\infty$ denotes the constant in \eqref{est:lipphisigma}. Indeed, by \eqref{ass:Dareg} and concavity of $\omega$, we obtain 
\begin{align*}
 \|\mathbb L_{F_2} - \mathbb L_{F_1}\|_{L^1(Q_1)}\lesssim & \|\omega(|F_1+\nabla \phi(F_1)-(F_2+\nabla \phi(F_2))|)\|_{L^1(Q_1)}\\
 \leq&\omega\left(\| F_1+\nabla \phi(F_1)-(F_2+\nabla \phi(F_2))\|_{L^1(Q_1)}\right)
\end{align*}
and the claim follows by H\"older's inequality and \eqref{est:lipphisigma}.

\substep{1.2} Conclusion. Equation \eqref{eq:psig} and Meyers estimate imply that there exists $\mu=\mu(\beta,d)>0$ such that
\begin{equation}\label{est:psigmeyer}
 \|\psi_G(F)\|_{W^{1,2+\mu}(Q_1)}\lesssim|G|=1.
\end{equation}
A combination of \eqref{eq:psig} and \eqref{est:difflf} yields 
\begin{align*}
 &\|\nabla \psi_G(F_1)-\nabla \psi_G(F_2)\|_{L^2(Q_1)}^2\\
 \lesssim& \int_{Q_1}\langle \mathbb L_{F_1}(x)(\nabla \psi_G(F_1)- \nabla \psi_G(F_2)),\nabla \psi_G(F_1)- \nabla \psi_G(F_2)\rangle\,dx\\
 =&\int_{Q_1}\langle (\mathbb L_{F_2}(x)-\mathbb L_{F_1}(x))(G+\nabla \psi_G(F_2)),\nabla \psi_G(F_1)- \nabla \psi_G(F_2)\rangle\,dx\\
 \leq&\|\mathbb L_{F_1}-\mathbb L_{F_2}\|_{L^\frac{2(2+\mu)}{\mu}(Q_1)}\|G+\nabla \psi_G(F_2)\|_{L^{2+\mu}(Q_1)}\|\nabla \psi_G(F_1)-\nabla \psi_G(F_2)\|_{L^2(Q_1)}.
\end{align*}
The claim \eqref{est:dpsig} (with $q=\frac{2(2+\mu)}{\mu}$) follows by H\"olders inequality, \eqref{est:difflf} and \eqref{est:psigmeyer}.

\step 2 Conclusion. Let $F_1,F_2,G\in\R^{d\times d}$ with $|G|=1$ be given. Then, \eqref{est:dpsig} and \eqref{est:difflf} yield 
\begin{align*}
 &|(D\bfa_0(F_1)-D\bfa_0(F_2))[G]|\\
 =&|\fint_{Q_1}\mathbb L_{F_1}(x)(G+\nabla \psi_G(F_1))-\mathbb L_{F_2}(x)(G+\nabla \psi_G(F_2))\,dx|\\
 \leq&\|\mathbb L_{F_1}-\mathbb L_{F_2}\|_{L^2(Q_1)}\|G+\nabla \psi_G(F_1)\|_{L^2(Q_1)}+\|\mathbb L_{F_2}\|_{L^2(Q_1)}\|\nabla \psi_G(F_1)-\nabla \psi_G(F_2)\|_{L^2(Q_1)}\\
 \lesssim&\omega(c|F_1-F_2|)^\frac12+\omega(c|F_1-F_2|)^\frac1q,
\end{align*}
which proves the claim.

\end{proof}

\subsection{Meyers estimate}

In Lemma~\ref{L:perturbation} and Proposition~\ref{P:largescale:holder}, we use a global version of Meyers estimate. For convenience of the reader, we here give a short proof of this well-known result. The key ingredient is the following classic higher integrability result
\begin{theorem}[\cite{Giu03} Theorem~6.6]\label{T:gehring}
Let $K>0$, $m\in(0,1)$, $s>1$ and $B=B_R(x_0)$ for some $x_0\in\R^d$ and $R>0$ be given. Suppose that $f\in L^1(B)$ and $g\in L^s(B)$ are such that for every $z\in\R^d$ and $r>0$ with $B_r(z)\subset B$ it holds
\begin{equation*}
 \|f\|_{\underline L^1(B_\frac{r}2(z))}\leq K\left(\|f\|_{\underline L^m(B_r(z))}+\|g\|_{\underline L^1(B_r(z))}\right).
\end{equation*}
Then there exist $q=q(K,m,s)\in(1,s]$ and $c=c(K,m,s)\in[1,\infty)$ such that $f\in L^q(\frac12 B)$ and it holds
\begin{equation*}
 \|f\|_{\underline L^q(\frac12 B)}\leq c \left(\|f\|_{\underline L^1(B)}+\|g\|_{\underline L^q(B)}\right).
\end{equation*}

\end{theorem}

Next, we state a well-known global version of Meyer's estimate.
\begin{lemma}\label{L:meyerglobal}
Fix $d\geq 2$, $\beta\in(0,1]$ and let $\bfa$ be a coefficient field of class $\mathcal A_\beta$. Then, for every $p>2$ there exists $p_0=p_0(\beta,d,p)\in(2,p]$ and $c=c(\beta,d,p)\in[1,\infty)$ such that if $u\in H_0^1(B)$ and $g\in L^p(B)$, where $B$ denotes either a ball or a cube, satisfy
\begin{equation*}
 \divv \bfa(\nabla u)=\divv g\qquad\mbox{in $\mathscr D'(B)$,}
\end{equation*}
then
\begin{equation*}
 \|\nabla u\|_{\underline L^{p_0} (B)}\leq c \|g\|_{\underline L^{p_0} (B)}.
\end{equation*}

\end{lemma}

\begin{proof}
Without loss of generality, we consider $B=B_1(0)$, the general case follows by scaling and translation. Throughout the proof, we write $\lesssim$ if $\leq$ holds up to a multiplicative constant depending on $\beta$ and $d$. 

We show that for every $z\in\R^d$ and $r>0$ it holds
\begin{equation}\label{meyerest:claim}
 \|\nabla u\|_{\underline L^2(B_\frac{r}{2}(z))}\lesssim \|\nabla u\|_{\underline L^{\frac{2d}{d+2}}(B_r(z))}+\|g\|_{\underline L^2(B_r(z))},
\end{equation}
where we extend $u$ and $g$ by zero on $\R^d\setminus B$. The claim follows by applying Theorem~\ref{T:gehring}.

Following \cite[Proposition B.6]{AM16}, we distinguish between the cases $B_r(z)\subset B$ and $B_r(z)\setminus B\neq \emptyset$. Suppose that $B_r(z)\subset B$. Then, a combination of the Caccioppoli inequality and the Sobolev inequality yields
\begin{align}\label{meyerest:interior}
 \|\nabla u\|_{\underline L^2(B_\frac{r}2(z))}\lesssim& r^{-1}\|u-(u)_{B_r(z)}\|_{\underline L^2(B_r(z))}+\|g\|_{\underline L^2(B_r(z))}\notag\\
 \lesssim& \|\nabla u\|_{\underline L^{\frac{2d}{d+2}}(B_r(z))}+\|g\|_{\underline L^2(B_r(z))}.
\end{align}
Next, we suppose that $B_r(z)\setminus B\neq\emptyset$. We easily obtain the following Caccioppoli inequality
\begin{equation*}
 \|\nabla u\|_{\underline L^2(B_\frac{r}2(z))}\lesssim r^{-1}\|u\|_{\underline L^2(B_{r}(z))}+\|g\|_{\underline L^2(B_{r}(z))}.
\end{equation*}
Clearly, we find $c=c(d)>0$ such that
\begin{equation*}
 |B_{2r}(z)\setminus B|\geq c|B_{2r}(z)|.
\end{equation*}
Hence, we can apply a version of Poincar\'e inequality (see e.g.,\ \cite[Theorem~3.16]{Giu03}) to obtain
\begin{align}\label{meyerest:boundary}
 \|\nabla u\|_{\underline L^2(B_\frac{r}2(z))}\lesssim& r^{-1}\|u\|_{\underline L^2(B_{2r}(z))}+\|g\|_{\underline L^2(B_{2r}(z))}\notag\\
 \lesssim&\|\nabla u\|_{\underline L^{\frac{2d}{d+2}}(B_{2r}(z)}+\|g\|_{\underline L^2(B_{2r}(z)}.
\end{align}
Clearly, \eqref{meyerest:interior}, \eqref{meyerest:boundary} and a simple covering argument imply \eqref{meyerest:claim} which finishes the proof.

\end{proof}

\end{document}